\documentclass[a4paper,10pt,reqno]{amsart}
\usepackage{dsfont,amsmath,amsfonts,amscd,amssymb,graphicx,mathrsfs,eufrak,upgreek,pdflscape,colonequals,euscript,amsmath}
\usepackage[usenames]{color}
\usepackage{setspace}
\usepackage{array,multirow,booktabs,longtable}
\usepackage{comment}

\renewcommand*\partial{\textnormal{\reflectbox{6}}}

\usepackage{array,multirow,booktabs,longtable}
\newcommand{\marginparstretch}{0.6}
\let\oldmarginpar\marginpar
\renewcommand\marginpar[1]{\-\oldmarginpar[\framebox{\setstretch{\marginparstretch}\begin{minipage}{\marginparwidth}{\raggedleft\tiny #1}\end{minipage}}]{\framebox{\setstretch{\marginparstretch}\begin{minipage}{\marginparwidth}{\raggedright\tiny #1}\end{minipage}}}}

\newcommand{\xdownarrow}[1]{%
  {\left\downarrow\vbox to #1{}\right.\kern-\nulldelimiterspace}
}

\usepackage[colorlinks]{hyperref}

\hypersetup{
  colorlinks   = true,          %Colors links instead of ugly boxes
  urlcolor     = blue,          %Color for external hyperlinks
  linkcolor    = purple,          %Color of internal links
  citecolor   = blue             %Color of citations
}

\makeatletter
\let\c@figure\c@table
\let\ftype@figure\ftype@table
\let\ext@figure\ext@table

\makeatother 
\usepackage{tikz,mathrsfs}
\usetikzlibrary{arrows,decorations.pathmorphing,decorations.pathreplacing,positioning,shapes.geometric,shapes.misc,decorations.markings,decorations.fractals,calc,patterns}

\tikzset{
        cvertex/.style={circle,draw=black,inner sep=1pt,outer sep=3pt},
        vertex/.style={circle,fill=black,inner sep=1pt,outer sep=3pt},
        DB/.style={circle,draw=black,circle,fill=black,inner sep=0pt, minimum size=4pt},
        DBlue/.style={circle,draw=black,circle,fill=white,inner sep=0pt, minimum size=4pt},
        DW/.style={circle,draw=black,inner sep=0pt, minimum size=4pt},
        star/.style={circle,fill=yellow,inner sep=0.75pt,outer sep=0.75pt},
        tvertex/.style={inner sep=1pt,font=\scriptsize},
        gap/.style={inner sep=0.5pt,fill=white}}

\newtheorem{thm}{Theorem}[section]
\newtheorem{prop}[thm]{Proposition}
\newtheorem{lemma}[thm]{Lemma}
\newtheorem{cor}[thm]{Corollary}

\theoremstyle{definition} 
\newtheorem{defin}[thm]{Definition}

\newtheorem{example}[thm]{Example}

\newtheorem{summary}[thm]{Summary}

\newtheorem{remark}[thm]{Remark}
\newtheorem{conj}[thm]{Conjecture}

\newtheorem{problem}[thm]{Problem}
\newtheorem{notation}[thm]{Notation}

\numberwithin{equation}{section}

\newcommand{\m}{\mathfrak{m}}
\newcommand{\n}{\mathfrak{n}}

\def\Hom{\mathop{\rm Hom}\nolimits}

\def\End{\mathop{\rm End}\nolimits}
\def\Ext{\mathop{\rm Ext}\nolimits}

\def\Ker{\mathop{\rm Ker}\nolimits}

\def\Sing{\mathop{\rm Sing}\nolimits}
\def\Supp{\mathop{\rm Supp}\nolimits}

\def\Spec{\mathop{\rm Spec}\nolimits}
\def\GKdim{\mathop{\rm GKdim}\nolimits}
\def\JRdim{\mathop{\rm Jdim}\nolimits}

\DeclareMathOperator{\CM}{CM}

\def\Id{\mathop{\rm{Id}}\nolimits}

\newcommand{\con}{\mathrm{con}}

\newcommand{\CA}{\mathrm{A}_{\con}}

\newcommand{\AB}{\mathrm{A}}

\def\ab{\mathop{\rm ab}\nolimits}

\def\redu{\mathop{\rm red}\nolimits}

\newcommand\Curve[1]{C_{#1}}

\newcommand{\N}{\mathbb N}
\newcommand{\Z}{\mathbb Z}
\newcommand{\De}{\Delta}

\newcommand{\R}{\mathbb R}

\renewcommand{\R}{\mathbb R}
\newcommand{\C}{\mathbb C}
\newcommand{\la}{\uplambda}
\DeclareMathOperator{\ord}{ord}

\newcommand{\sfb}{\mathsf{b}}
\newcommand{\sff}{\mathsf{f}}
\newcommand{\sfg}{\mathsf{g}}
\newcommand{\sfh}{\mathsf{h}}
\newcommand{\sfp}{\mathsf{p}}
\newcommand{\sfq}{\mathsf{q}}

\newcommand{\sft}{\mathsf{t}}
\newcommand{\bigO}{\scrO}

\newcommand{\scrA}{\EuScript{A}}

\newcommand{\scrC}{\EuScript{C}}

\newcommand{\scrH}{\EuScript{H}}

\newcommand{\scrJ}{\EuScript{J}}

\newcommand{\scrM}{\EuScript{M}}
\newcommand{\scrN}{\EuScript{N}}
\newcommand{\scrO}{\EuScript{O}}
\newcommand{\scrP}{\EuScript{P}}

\newcommand{\scrR}{\EuScript{R}}
\newcommand{\scrS}{\EuScript{S}}

\newcommand{\scrU}{\EuScript{U}}
\newcommand{\scrV}{\EuScript{V}}

\newcommand{\scrX}{\EuScript{X}}
\newcommand{\scrY}{\EuScript{Y}}
\newcommand{\scrZ}{\EuScript{Z}}

\newcommand{\hatM}{\mathfrak{n}}
\DeclareMathOperator{\sym}{cyc}		% more clear
\DeclareMathOperator{\Aut}{Aut}
\newcommand{\dcyc}{\updelta}

\newcommand{\llangle}{\langle\kern -2.5pt\langle}
\newcommand{\rrangle}{\rangle\kern -2.5pt\rangle}
\newcommand{\llcurve}{\{\kern -3pt\{}
\newcommand{\rrcurve}{\}\kern -3pt\}}
\newcommand{\llsq}{[\![}
\newcommand{\rrsq}{]\!]}
\newcommand{\lcl}{(\kern -2.5pt(}
\newcommand{\rcl}{)\kern -2.5pt)}
\newcommand{\Jac}{\scrJ\mathrm{ac}}
\newcommand{\AlgJac}{\mathrm{A}\scrJ\mathrm{ac}}

\renewcommand{\phi}{\upvarphi}

\newcommand{\ring}{\C\llangle x,y\rrangle}
\newcommand{\freering}{\C\langle x,y\rangle}
\newcommand{\al}{\upalpha}
\newcommand{\be}{\upbeta}
\newcommand{\ga}{\upgamma}
\newcommand{\de}{\updelta}

\newcommand{\fringd}{\C\langle \mathsf{x}\rangle}
\newcommand{\ringd}{\C\llangle \mathsf{x}\rrangle}
\newcommand{\ringtwo}{\C\llangle x,y\rrangle}
\newcommand{\ringdy}{\C\llangle x_1,\dots,x_{d-1},y\rrangle}
\newcommand{\polyringd}{\C\langle x_1,\dots,x_d\rangle}
\newcommand{\corank}{\scrC\mathrm{rk}}

\newcommand{\SP}{\mathsf{SP}}
\newcommand{\ESP}{\mathsf{ESP}}
\newcommand{\fJ}{\mathfrak{J}}
\newcommand{\jd}{\upnu}  % notation for the value of J-dim

\begin{document}
\title{Local Normal Forms of Noncommutative Functions}
\author{Gavin Brown}
\address{Gavin Brown, Mathematics Institute, Zeeman Building, University of Warwick, Coventry, CV4 7AL, UK.}
\email{G.Brown@warwick.ac.uk}
\author{Michael Wemyss}
\address{Michael Wemyss, School of Mathematics and Statistics,
University of Glasgow, University Place, Glasgow, G12 8QQ, UK.}
\email{michael.wemyss@glasgow.ac.uk}
\begin{abstract}
This article describes local normal forms of functions in noncommuting variables, up to equivalence generated by isomorphism of noncommutative Jacobi algebras, extending singularity theory in the style of Arnold's commutative local normal forms into the noncommutative realm. This generalisation unveils many new phenomena, including an ADE classification when the Jacobi ring has dimension zero and, by taking suitable limits, a further ADE classification in dimension one. These are natural generalisations of the simple singularities and those with infinite multiplicity in Arnold's classification.  We obtain normal forms away from some exceptional Type~$E$ cases. Remarkably these normal forms have no continuous parameters, and the key new feature is that the noncommutative world affords larger families.

This theory has a range of immediate consequences to the birational geometry of 3-folds.  The normal forms of dimension zero \emph{are} the analytic classification of smooth 3-fold flops, and one outcome of NC singularity theory is the first list of all Type~$D$ flopping germs, generalising Reid's famous pagoda classification of Type~$A$, with variants covering Type~$E$.  The normal forms of dimension one have further applications to divisorial contractions to a curve.  In addition, the general techniques also give strong evidence towards new contractibility criteria for rational curves.
\end{abstract}
%Primary: 58K40 Classification; finite determinacy of map germs
%Secondary: 16S15 Finite generation, finite presentability, normal forms (diamond lemma, term-rewriting)
%16S38 Rings arising from noncommutative algebraic geometry [See also 14A22]
%32S10 Invariants of analytic local rings
%14E30 Minimal model program (Mori theory, extremal rays)

\maketitle

%-----------------------------------------------------------------------------------------
\section{Introduction}
%Every smooth point on a $d$-dimensional variety is, after completion, simply the power series ring in $d$ variables.  In other words, complete locally, smooth points are unique.  A similar phenomenon applies to more interesting singular points: as observed by Arnold and others in various landmark papers \cite{Arnold, Mori, BGS}, if we pass to the completion, there are many fewer singular points than might naively be expected, and they often have surprising connections to other parts of mathematics.  The purpose of this paper is to extend both the theory and the key classifications into the noncommutative context. In the process, we uncover surprising algebraic facts with geometric corollaries.

This article establishes a noncommutative analogue of the classical singularity theory of function germs as set out in Arnold's landmark paper \cite{Arnold}. The fundamental components of such a theory remain: one must (1) work with germs, or locally in some sense, (2) establish suitable notions of equivalence, (3) determine discrete parameters to distinguish families with similar properties, (4) classify the families for `small' values of the discrete parameters, (5) develop general theory for where classification is difficult, and crucially (6) use the classification to give applications in other areas of mathematics.

We outline our noncommutative approach to components (1--3) in \S\ref{sec!NST} below, with full details given in \S\ref{sec!jacobi} and \S\ref{sec:NC sing theory}. The first classifications of (4) are discussed in \S\ref{tables section}--\ref{alg sect intro}, and their proofs in \S\ref{sec!typeA} and \S\ref{sec!typeD} use the general theory of \S\ref{sec!automorphisms}--\ref{sec:NC sing theory} and Appendix~\ref{J appendix}, which initiate~(5). Arnold remarks \cite[2]{Arnold} that the definition and naming conventions of families may only become clear after classification, and so although we use the ADE names throughout, it is only in \S\ref{central sect} that their intrinsic definition is established.  As for applications, we instigate component (6) in \S\ref{geo section}, giving a classification of various rational neighbourhoods in $3$-fold birational geometry, with further applications to curve counting.

%-----------------------------------------------------------------------------------------
\subsection{Noncommutative Singularity Theory}\label{sec!NST}
For $d\geq 1$ consider the noncommutative formal power series ring $\ringd=\mathbb{C}\llangle x_1,\hdots,x_d\rrangle$, which is the complete local version of the free algebra. 
From the perspective of this paper, the algebra $\ringd$ replaces the commutative power series ring $\C\llsq x_1,\hdots,x_d\rrsq$ from classical singularity theory.

For any $f\in\ringd$, it is possible to cyclically differentiate $f$ with respect to a variable~$x_i$ to obtain an element $\dcyc_if$.  The collection of such elements generate a closed two-sided ideal $\lcl \dcyc_1f,\hdots,\dcyc_df \rcl$, the details of which are recalled in \S\ref{completion section}.  The resulting quotient  
\[
\Jac(f)=\frac{\mathbb{C}\llangle x_1,\hdots,x_d\rrangle}{\lcl \dcyc_1f,\hdots,\dcyc_df \rcl}
\]
is called the \emph{Jacobi algebra} of $f$, and the element $f$ is called the \emph{potential}. 

We will regard $f$ and $g$ as being equivalent if their Jacobi algebras are isomorphic, remarking that in the noncommutative setting, given the hidden dependence on cyclic equivalence, naive versions of the Tjurina algebra do not exist (see \ref{Milnor=Tjurina}).  With the ring $\ringd$ fixed and the equivalence relation established, the overarching aim of singularity theory remains: to classify all equivalence classes of potentials satisfying numerical criteria, and to develop powerful theory in the situation where classification is not possible.

\medskip
Whenever $d>1$ the algebra $\ringd$ is not noetherian, and the exponential explosion in its growth means that factoring by only $d$ elements often results in Jacobi algebras with pathological properties. As in the classical case, pathologies turn out not to matter: the complexity of \emph{some} singularities prevents neither the development of a general theory, nor various classification results for those which satisfy reasonable numerical conditions.  
\medskip

Writing $\mathfrak{J}$ for the Jacobson radical of $\Jac(f)$, the first and natural restriction to impose on $f$ is to numerically constrain the growth of successive quotients of the chain of ideals
\[
\Jac(f)\supseteq\fJ\supseteq\fJ^2\supseteq\cdots.
\] 
This numerical growth, defined in \ref{Jdim defin}, is called the $\mathfrak{J}$-dimension, and will be written $\JRdim\Jac(f)$. As explained in \ref{GK bad remark},  since $\Jac(f)$ is a factor of a complete ring, there is no reasonable Gelfand--Kirillov dimension, and the $\mathfrak{J}$-dimension replaces it.

Alongside the development of a more general theory, the motivating problem is to extend Arnold-style classification of germs \cite{Arnold} into the above noncommutative setting.  

\begin{problem}\label{motivating problem}
For any finite $\jd\geq 0$, produce a small set of potentials~$\scrS_\jd$ that realise every Jacobi algebra of $\mathfrak{J}$-dimension~$\jd$, up to isomorphism.\end{problem}

Ideally, the elements of $\scrS_\jd$ would be \emph{normal forms}, namely that if $f,g\in\scrS_\jd$ with $f\neq g$, then the resulting Jacobi algebras are not isomorphic.  Building on foundational algebraic results of Iyudu--Shkarin \cite{IS}, in Appendix~\ref{J appendix} we show that, for small $\mathfrak{J}$-dimension~$\jd$, \ref{motivating problem} essentially reduces to a problem in $d\le2$ variables. 

\medskip
We will focus mainly on the situation $\jd\leq 1$, which is already highly non-trivial. Below we will observe, in the noncommutative context, exactly the same phenomena as in Arnold \cite{Arnold}, whereby such precise numerical restrictions are often only motivated afterwards, by their answer and by the incredibly rich families that they describe.  The restriction $\jd\leq 1$ is also, happily, the condition needed for the applications to birational geometry. We do however remark that it is not even clear that the set $\scrS_0$ is countable, never mind $\scrS_1$, and there is certainly no prima facie reason why ADE should enter.

\subsection{Noncommutative ADE Normal Forms}\label{tables section}
We now introduce the ADE families that will turn out to solve  \ref{motivating problem} when $\jd\leq 1$. The main results regarding what precisely these families classify are stated later, in \S\ref{alg sect intro}.  

It is a feature of singularity theory that it is often not possible to rigorously define a series until {\em after} it has  been classified.  In the subsections that follow we will use various different phenomena to explain the ADE names of the families, but it is only after classification that one can make the moves needed to extract this ADE information. As such, the definition of the families below follows the usual pattern of classical singularity theory: their definition comes first, and their justification comes afterwards.

Below, we view the families with $\jd=0$ as the noncommutative version of simple singularities of \cite{Arnold}, and we view the $\jd=1$ families as the `limit' of the $\jd=0$ case, thus the noncommutative versions of the singularities $A_\infty$ and $D_\infty$ of \cite{BGS}.

\renewcommand{\arraystretch}{1.2}
\begin{table}[ht]
\begin{tabular}{ccp{5.5cm}p{3.75cm}}
\toprule
{\bf Type}&{\bf Name}&{\bf Normal form}&{\bf Conditions}\\
\midrule
A & $A_n$& $z_1^2+\hdots+z_{d-2}^2+x^2+y^n$ & $n\geq 2$\\[6pt]
\multirow{2}*{D} &$D_{n,m}$ &$z_1^2+\hdots+z_{d-2}^2+xy^2+x^{2n}+ x^{2m-1}$ & $n,m\geq 2, m\leq 2n-1$\\
&$D_{n,\infty}$ &$z_1^2+\hdots+z_{d-2}^2+xy^2+x^{2n}$ & $n\geq2$\\[6pt]
\multirow{2}*{E} & $E_{6,n}$&$z_1^2+\hdots+z_{d-2}^2+x^3+xy^3+y^n$ & $n\geq 4$\\
&&$z_1^2+\hdots+z_{d-2}^2+x^3 + \bigO_4$ & \text{(various cases)}\\[2pt]
\bottomrule
\end{tabular}\\[4pt]
\caption{$\mathfrak{J}$-dimension~$0$ normal forms\label{tab!zero}}
\end{table}

\medskip
With the above caveats, for any $d\geq 2$ consider the normal forms in Table~\ref{tab!zero}. % which below will fully classify the $\jd=0$ case in \ref{motivating problem}. 
The big~$\bigO$ notation is explained in \S\ref{conventions}. 
%\medskip
%\noindent
It is possible to write Type~$D$ in the unified manner $z_1^2+\hdots+z_{d-2}^2+xy^2+x^{2n}+ \upvarepsilon x^{2m-1}$ where $\upvarepsilon$ is either $0$ or $1$, but often it will be preferable to regard them as two distinct families, both of Type~$D$.  In addition to the fact that Type~$D$ is larger than in the classical case, what is perhaps much more remarkable is that in Type~$E$ there are infinitely many cases: the family $E_{6,n}$ stated, together with various other examples all of the form $x^3 + \bigO_4$, whose expressions are more complicated and will be optimised elsewhere \cite{BW2}.

Taking the limit $n\to\infty$ of the above forms gives the normal forms of Table~\ref{tab!one}, where again all are optimised, except the very last line.
The classical case admits precisely two examples, namely the singularities $A_\infty$ and $ D_\infty$ of \cite{BGS}.  The noncommutative families are thus again larger: Type~$D$ splits into two, there are infinitely many examples within $D_{\infty,m}$, and Type~$E$ is no longer empty.

\begin{table}[ht]
\begin{tabular}{ccp{5cm}p{3.25cm}}
\toprule
{\bf Type}&{\bf Name}&{\bf Normal form}&{\bf Conditions}\\
\midrule
A & $A_\infty$& $z_1^2+\hdots+z_{d-2}^2+x^2$ & \\[6pt]
\multirow{2}*{D} & $D_{\infty,m}$ & $z_1^2+\hdots+z_{d-2}^2+xy^2+ x^{2m-1}$ & $m\geq 2$\\
& $D_{\infty,\infty}$ & $z_1^2+\hdots+z_{d-2}^2+xy^2$ & \\[6pt]
\multirow{2}*{E} & $E_{6,\infty}$&$z_1^2+\hdots+z_{d-2}^2+x^3+xy^3$ &\\
&& $z_1^2+\hdots+z_{d-2}^2+x^3+\scrO_4$&\\[2pt]
\bottomrule
\end{tabular}\\[4pt]
\caption{$\mathfrak{J}$-dimension~$1$ normal forms\label{tab!one}}
\end{table}

\medskip
With the benefit of hindsight, there are two reasons why one might expect the $\jd=1$ case to be the limit of $\jd=0$.  First, on taking limits, the simple $A_n$ and $D_n$ families give rise to the germs $x^2$ and $xy^2$, and the noncommutative families above generalise this passage from the isolated to the non-isolated.  Second, in terms of the birational geometry of \S\ref{geo intro sect} below, contraction algebras should make sense of the feeling that divisor-to-curve contractions are limits of infinite families of flops.  

\medskip

In this paper, we prove that every $\Jac(f)$ with $\JRdim\Jac(f)\leq 1$ is isomorphic to a normal form in Type~$A$ or~$D$ above, or has the general form stated for $E$.
We remark that the precise Type~$E$ normal forms stated, namely $E_{6,n}$ and $E_{6,\infty}$, are indeed genuine examples with $\mathfrak{J}$-dimension zero and one respectively. However, we refrain from describing the general case here, as we will treat all the exceptional Type~$E$ cases together, in a more technical companion paper \cite{BW2}.  

We now outline our results in more detail, before describing their applications.

%-----------------------------------------------------------------------------------------
\subsection{Main Noncommutative Singularity Theory Results}\label{alg sect intro}
Since constants differentiate to zero, and elements with linear terms differentiate to units, we can and do assume that $f$ has only quadratic and higher terms, which we write as $f\in\ringd_{\geq 2}$ or equivalently as an explicit sum of its homogeneous pieces 
\[
f=f_2+f_3+f_4+\hdots
\]
Just as in the classical theory, a Splitting Lemma~\ref{thm!splitting} identifies a coordinate system which separates variables of the non-degenerate quadratic part from variables of a higher order potential, so that without loss of generality
\[
f=x_1^2 +\cdots+x_r^2 + f_{\ge3}(x_{r+1},\dots,x_d)
\]
and thus we may turn attention to the potential~$f_{\ge3}$ in, typically, fewer variables.  The number $d-r$ is called the \emph{corank}, and as in the classical case there is a more intrinsic way of characterising it (\ref{defin corank}), namely as
\begin{equation}
\corank(f) = d - \dim_{\mathbb{C}} \left(\frac{\n^2 +I}{\n^2}\right)
\label{eqn:corank intro}
\end{equation}
where $\n=(x_1\hdots,x_d)$ and $I=\lcl \updelta_1f,\hdots\updelta_df\rcl$.
By the above and \ref{thm: appendix main} it turns out, in a manner pleasantly reminiscent of classical simple singularities, that the case when $\JRdim\Jac(f)\leq 1$ reduces to that of two variables.  We rename the variables $z_1,\hdots,z_{d-2},x,y$ to emphasise this fact. 

The following, a consequence of the Splitting Lemma together with a degree three preparation result, then characterises commutative Jacobi algebras in two variables.  These are precisely our Type~$A$ families in \S\ref{tables section}.
Below we adopt the convenient abuse of notation $f\cong g$ to mean $\Jac(f)\cong\Jac(g)$.  
\begin{prop}[\ref{thm!pagoda}, \ref{when comm}]\label{typeA intro}
If $f\in\ringd_{\geq 2}$, then the following hold.
\begin{enumerate}
\item $\corank(f)\le1$ if and only if 
\[
f\cong
\begin{cases}
z_1^2+\hdots+z_{d-2}^2+x^2&\\
z_1^2+\hdots+z_{d-2}^2+x^2+y^n&\mbox{for some }n\geq 2.\\
\end{cases}
\]
Each member of the bottom family has finite dimensional Jacobi algebra, whereas in the top case the algebra is infinite dimensional, with $\JRdim\Jac(f) = 1$.  
\item If $d=2$, i.e.\ $f\in\ringtwo$, then $\Jac(f)$ is commutative if and only if $\corank(f)\le1$.
\end{enumerate}

\end{prop}

Thus Jacobi algebras are commutative if and only if they are Type~A, and so new noncommutative invariants are needed to classify other types. The equation \eqref{eqn:corank intro} does admit an obvious generalisation, namely the \emph{higher coranks} defined in \S\ref{sec!splitting}, where for $f\in\ringd_{\geq 3}$ the second corank is
\begin{equation}\label{intro:corankdef}
\corank_{2}(f) =  d^{2} - \dim_{\mathbb{C}} \left(\frac{\n^{3} + I}{\n^{3}}\right).
\end{equation}
In classifying all $f$ with $\JRdim\Jac(f)\leq 1$, \ref{thm: appendix main} together with \ref{typeA intro} then reduces us to the case where $\corank(f)=2$ and $\corank_{2}(f)=2,3$.  The lowest case $\corank_{2}(f)=2$ turns out to be given by the Type~$D$ families in the tables of \S\ref{tables section}.

\begin{thm}[\ref{main Type D}]\label{Type D intro}
Suppose that $f\in\ringd_{\geq 2}$ with $\corank(f)=2$ and $\corank_2(f)=2$.
\begin{enumerate}
\item Then either
\[
f\cong
\left\{
\begin{array}{lll}
z_1^2+\hdots+z_{d-2}^2+xy^2&&D_{\infty,\infty}\\
z_1^2+\hdots+z_{d-2}^2+xy^2 + x^{2m+1}&\mbox{with }m\geq 1 &D_{\infty,m}\\
z_1^2+\hdots+z_{d-2}^2+xy^2 + x^{2n}&\mbox{with }n\geq 2&D_{n,\infty}\\
z_1^2+\hdots+z_{d-2}^2+xy^2 + x^{2n} + x^{2m+1}&\mbox{with }2n-2\geq m\geq n\geq 2&D_{n,m}\\
z_1^2+\hdots+z_{d-2}^2+xy^2 + x^{2m+1} + x^{2n}&\mbox{with }n>m\geq 1&D_{n,m}
\end{array}
\right.
\]
These $f$ all have mutually non-isomorphic Jacobi algebras.
\item  Furthermore, those labelled  \textnormal{$D_{\infty,\ast}$} satisfy $\JRdim\Jac(f)= 1$, whilst those labelled \textnormal{$D_{n,\ast}$} satisfy $\JRdim\Jac(f)= 0$.
\end{enumerate}
\end{thm}

It is remarkable that all normal forms are polynomial, and even more remarkable that all coefficients are integers. Indeed, all coefficients equal $1$, and there are no continuous parameters.

\medskip
The last remaining case for which $\JRdim\Jac(f)\leq 1$ holds is when $\corank(f)=2$ and $\corank_2(f)=3$. After a suitable change in coordinates, all such $f$ have the form
\[
f\cong z_1^2+\hdots+z_{d-2}^2+x^3 + f_{\geq 4}(x,y)
\] 
with some extra conditions on $f_{\geq 4}(x,y)$ to ensure that $\JRdim\Jac(f)\leq 1$.  We refer to these potentials as Type~$E$.   The families described in both Types~$E_{6,n}$ and~$E_{6,\infty}$ in \S\ref{tables section} are genuine examples. However there are many others; see \cite{BW2}.
Their classification depends, in a rather more subtle manner, on naturally-defined higher coranks (see~\ref{defin higher corank}). For example, the potential $x^3+xy^3$ of Type~$E_{6,\infty}$ has second corank equal to~$3$, with all higher coranks equal to $4$, while in contrast the potentials $f_n$ of Type~$E_{6,n}$ for $n\ge5$ trim those coranks to
\[
\corank_2(f_n),\corank_3(f_n),\ldots,\corank_{n+6}(f_n) = 3,4,4,\dots,4,4,3,3,2,1,1.
\]
In particular $\Jac(f_n)$ has dimension $4(n+3)$.  Controlling normal forms in such situations is both theoretically and computationally more difficult.

%-----------------------------------------------------------------------------------------
\subsection{Intrinsically Extracting ADE}
It turns out that there are two, completely distinct, ways to extract ADE behaviour from the families defined above, and thus explain the ADE naming conventions.  In this section we explain the purely algebraic method; the birational geometry method is explained in \S\ref{geo cor section intro} below.  

The first method is the most surprising.  Consider the six algebras defined explicitly by taking the quotient of $\mathbb{C}\langle x,y\rangle$ by one of the following six two-sided ideals
\begin{comment}
\begin{equation}\label{eqn:6keyalgebras}
\begin{array}{cccccc}
\arraycolsep=0pt
{\small\begin{array}{c}
x+y+z,\\
x, y, z
\end{array}}
&
{\small\begin{array}{c}
x+y+z,\\
x^2, y^2, z^2
\end{array}}
&
{\small\begin{array}{c}
x+y+z,\\
x^2, y^3, z^3
\end{array}}
&
{\small\begin{array}{c}
x+y+z,\\
x^2, y^3, z^4
\end{array}}
&
{\small\begin{array}{c}
x^2+y+z,\\
\mathrm{(see\,\, \ref{rem:5thisstrange})}
\end{array}}
&
{\small\begin{array}{c}
x+y+z,\\
x^2, y^3, z^5
\end{array}}
\end{array}
\end{equation}
\end{comment}
\begin{comment}
\begin{equation}\label{eqn:6keyalgebras}
\arraycolsep=0pt
\begin{array}{cccccc}
\left(
{\small\begin{array}{c}
x+y+z\\
x, y, z
\end{array}}
\right)
&&
\left(
{\small\begin{array}{c}
x+y+z\\
x^2, y^2, z^2
\end{array}}
\right)
&&
\left(
{\small\begin{array}{c}
x+y+z\\
x^2, y^3, z^3
\end{array}} 
\right) \\
&
\left(
{\small\begin{array}{c}
x+y+z \\
x^2, y^3, z^4
\end{array}} 
\right) 
&&
\left(
{\small\begin{array}{c}
x^2+y+z\\
\mathrm{(see\,\, \ref{rem:5thisstrange})}
\end{array}}
\right) 
&&
\left(
{\small\begin{array}{c}
x+y+z\\
x^2, y^3, z^5
\end{array}}
\right) 
\end{array}
\end{equation}
\end{comment}
\begin{equation}\label{eqn:6keyalgebras}
\arraycolsep=0pt
\begin{array}{l}
\left(
{\small\begin{array}{c}
x+y+z\\
x, y, z
\end{array}}
\right)
\qquad
\left(
{\small\begin{array}{c}
x+y+z\\
x^2, y^2, z^2
\end{array}}
\right)
\qquad
\left(
{\small\begin{array}{c}
x+y+z\\
x^2, y^3, z^3
\end{array}} 
\right) \\
\phantom{\left(
{\small\begin{array}{c}
x+y+z\\
x, y, z
\end{array}}
\right)
\qquad
\qquad
}
\qquad
\left(
{\small\begin{array}{c}
x+y+z \\
x^2, y^3, z^4
\end{array}} 
\right) 
\qquad
\left(
{\small\begin{array}{c}
x^2+y+z\\
\mathrm{(see\,\, \ref{rem:5thisstrange})}
\end{array}}
\right) 
\qquad
\left(
{\small\begin{array}{c}
x+y+z\\
x^2, y^3, z^5
\end{array}}
\right) 
\end{array}
\end{equation}
These have dimension $1, 4, 12, 24, 40$, and $60$ respectively.  
%Five of the presentations in \eqref{eqn:6keyalgebras} are evidently related to ADE triples, whereas the ideal which involves a reference to \ref{rem:5thisstrange} is slightly more subtle.  
A presentation-free description of all six algebras, which is conceptually more compelling, uniformly describes each in terms of the preprojective algebra of ADE Dynkin diagrams (see \S\ref{subsec:sixalgebras}).

\medskip
The following result allows us to associate ADE information directly to the normal forms in \S\ref{tables section}, by asserting that \emph{all} such Jacobi algebras generically slice to one of the six algebras in \eqref{eqn:6keyalgebras} above.  This is particularly striking, since nothing in the definition of the families has involved any mention of only six algebras, nor any mention of the preprojective algebra, and aside from our naming conventions, any mention of  ADE.  It is not even clear that if $\JRdim\Jac(f)\leq 1$, then $\Jac(f)$ admits a non-unit central element.

In order to consider all cases $\JRdim\Jac(f)\leq 1$ together, below we adopt the convention that each $\upvarepsilon_i$ can be either 0 or 1.

\begin{thm}[\ref{central element main}]\label{central element main intro}
Consider the normal forms $A_n$, $D_{n,m}$, $D_{n,\infty}$, $E_{6,n}$, $A_{\infty}$, $D_{\infty,m}$, $D_{\infty,\infty}$ and $E_{6,\infty}$ from \S\textnormal{\ref{tables section}}.  In each case, define an element $s$ as follows
\[
\begin{tabular}{clp{3cm}c}
\toprule
{\bf Type}&{\bf Normal form}&{\bf Conditions}&$s$\\
\midrule
\textnormal{A} & $z_1^2+\hdots+z_{d-2}^2+x^2+\upvarepsilon_1y^n$ & $n\in\mathbb{N}_{\geq 2}\cup\{\infty \}$ & $y$\\[4pt]
\textnormal{D} & $z_1^2+\hdots+z_{d-2}^2+xy^2+\upvarepsilon_2x^{2n}+\upvarepsilon_3 x^{2m-1}$ & $m,n\in\mathbb{N}_{\geq 2}\cup\{\infty \}$ & $x^2$\\[4pt]
\textnormal{E} & $z_1^2+\hdots+z_{d-2}^2+x^3+xy^3+\upvarepsilon_4 y^n$ & $n\in\mathbb{N}_{\geq 4}$&{$g_{6,n}$}\\
\bottomrule
\end{tabular}
\]
where $g_{6,n}$ is defined in \S\textnormal{\ref{central sect}}. Then the following statements hold.
\begin{enumerate}
\item  The element $s$ is central in $\Jac(f)$, and furthermore $\Jac(f)/(s)$ is isomorphic to one of the six algebras in \eqref{eqn:6keyalgebras}.  
\item More specifically, $\Jac(f)/(s)$ is isomorphic to the first algebra in \eqref{eqn:6keyalgebras} when $f$ is in the family $A_*$, the second algebra in \eqref{eqn:6keyalgebras} when $f$ is in the family $D_{*,*}$, and the third algebra in \eqref{eqn:6keyalgebras} when $f$ is in the family $E_{6,*}$.  
\item For any generic central element $g$, the quotient $\Jac(f)/(g)$ is isomorphic to the first algebra in \eqref{eqn:6keyalgebras} when $f$ is in the family $A_*$, and the second algebra in \eqref{eqn:6keyalgebras} when $f$ is in the family $D_{*,*}$.
\end{enumerate} 
\end{thm}

Most of the content in the theorem lies within the third part since generic elements, defined in \ref{def:generic element}, provide an \emph{intrinsic} method of extracting the ADE information.  The choice of central element $g_{6,n}$, which is rather involved, works for Type $E_{6,*}$, and there is also strong evidence that generic elements there also quotient to give the correct algebra in the sequence \eqref{eqn:6keyalgebras}.  Establishing this is computationally much harder, and will be addressed elsewhere \cite{BW2}.  We remark that all other examples we know within Type~$E$, but which are not explicitly stated above, also factor to one of the six algebras in \eqref{eqn:6keyalgebras}.

In the geometric context of~\S\ref{geo intro sect} below, the generic central element $g$ of \ref{central element main intro} should be thought of as the noncommutative version of Reid's general elephant \cite[(1.14)]{Pagoda}. Remarkably, the above theorem neither implies, nor is implied by, Reid's version.  
%In both cases, taking a generic central element is essential  to avoid pathological behaviour (see \ref{rem: need generic}). 

%-----------------------------------------------------------------------------------------
\subsection{The Classification of Flops}\label{geo cor section intro}
The noncommutative singularity results in \S\ref{alg sect intro} have immediate applications in birational geometry.  The slogan is simple:  whilst Arnold's commutative normal forms classify Du Val singularities, noncommutative normal forms  classify compound Du Val (cDV) singularities.

The key and most remarkable special case is that the normal forms in Table~\ref{tab!zero} classify smooth 3-fold flops \cite{DW1, August, JKM}.  We first very briefly recall the geometric setting, where general background is left to e.g.\ \cite{KMM}, before outlining the new results.

\medskip
Given any crepant projective birational morphism 
 $\scrX\to\Spec \scrR$, where $\scrR$ is complete local cDV singularity, there is an associated \emph{contraction algebra} $\AB_{\con}$ formed by considering noncommutative deformations of the curves above the unique closed point \cite{DW1, DW3}.  This is the finest known curve invariant associated to the contraction.  When the contraction is furthermore \emph{simple}, namely the reduced fibre above the origin is $\mathbb{P}^1$, and further $\scrX$ is smooth, then it is well known \cite{DW1,VdBCY} that $\AB_{\con}\cong\Jac(f)$ for some $f\in\mathbb{C}\llangle x,y\rrangle$ (see e.g.\ \cite[3.1(2)]{BW}).

Since cDV singularities are normal, necessarily $\JRdim\CA\leq 1$, and there is a natural geometric dichotomy.  Indeed, as explained in \ref{cor: Acon has Jdim leq 1}, if $\CA$ is a contraction algebra associated to a crepant $\scrX\to\Spec\scrR$ as above, then
\begin{enumerate}
\item   $\JRdim\CA=0$ if and only if $\scrX\to\Spec\scrR$ is a flop, and
\item $\JRdim\CA=1$ if and only if $\scrX\to\Spec\scrR$ is a divisorial contraction to a curve. 
\end{enumerate} 

The only other fact we will require is that every $\scrX\to\Spec\scrR$ has an associated ADE type, since by Reid's general elephant \cite[(1.14)]{Pagoda} a generic $g\in\m$ slices to give an ADE surface singularity $\scrR/g$.  We will say $\scrR$ has Type~$D$ if the generic slice is Type~$D$, etc.

%\subsubsection{Classification of Type $D$ flops}
\medskip
With this in mind, the results in \S\ref{alg sect intro} have immediate consequences.  After first using the normal forms to classify contraction algebras, the following then gives the analytic classification of length two flops and beyond. It also gives the second, geometric, method to extract ADE information from the normal forms in \S\ref{tables section}.

\begin{thm}[\ref{the only Type A CAs}, \ref{the only Type D CAs}, \ref{all type D are geometric}, \ref{thm:actualclassification}]\label{main Type D intro}
With notation as above, the following hold:
\begin{enumerate}
\item\label{main Type D intro 1}
The only contraction algebras for Type~$A$ and~$D$ flops are, up to isomorphism, the Jacobi algebras of Type $A$ and $D$ potentials in Table~\textnormal{\ref{tab!zero}}.
\item\label{main Type D intro 2}
All Jacobi algebras in \textnormal{\ref{typeA intro}} and \textnormal{\ref{Type D intro}} are contraction algebras.
\item\label{main Type D intro 3} 
\begin{enumerate}
\item Type $A$ flops are classified by Type $A$ normal forms in Table~\textnormal{\ref{tab!zero}}.
\item Type $D$ flops are classified by Type $D$ normal forms in Table~\textnormal{\ref{tab!zero}}.
\end{enumerate}
\end{enumerate}
Furthermore, Type $E$ flops are classified by Type $E$ normal forms.
\end{thm}

In the process of establishing \ref{main Type D intro} we use the examples of flops given in our previous work \cite{BW}, together with their generalisations \cite{Okke, Kawamata}.  Whilst the above classifies flops using noncommutative data, the following geometric description is perhaps more desirable. 

\begin{thm}[\ref{thm:TypeDgeoclassdiag}]\label{thm:TypeDclassintro}
There is a one-to-one correspondence between lattice points in Figure~\textnormal{\ref{fig:TypeDintro}} and the base singularities $0\in\Spec\scrR$ of Type $D$ flops, given by 
\[
(n,m)\mapsto \Spec\left( \frac{\mathbb{C}\llsq u,v,x,y\rrsq}{ u^2+v^2y - x(y^{2n+1}+(x+\upvarepsilon y^m)^2)}\right)
\] 
where $\upvarepsilon=1$ if the lattice point is contained within the shaded region, and $\upvarepsilon=0$ otherwise.
\begin{figure}
\begin{center}
\begin{tikzpicture}[xscale=0.8,yscale=0.8]

\draw[->] (0,0) -- (9.5,0);
\draw[->] (0,0) -- (0,8.5);

\filldraw[color=gray!50!white] (0.5,1.6)--(3.5,7.5)--(8.6,7.5)--(8.6,0.5)--(0.5,0.5)--cycle;

\node at (9.8,0) {$n$};
\node at (0,8.8) {$m$};

\def\coords{1,...,8}
\foreach \y in {1,...,3}
\foreach \x in \coords \node at (\x,\y) {$\scriptstyle \bullet$};
\foreach \x in {2,...,8} \node at (\x,4) {$\scriptstyle \bullet$};
\foreach \x in {2,...,8} \node at (\x,5) {$\scriptstyle \bullet$};
\foreach \x in {3,...,8} \node at (\x,6) {$\scriptstyle \bullet$};
\foreach \x in {3,...,8} \node at (\x,7) {$\scriptstyle \bullet$};
\foreach \x in \coords \node at (\x,-0.3) {\textnormal{\x}};
\foreach \y in {1,...,7} \node at (-0.3,\y) {\textnormal{\y}};

\node at (8.75,-0.3) {$\hdots$};

\node at (9,4) {$\hdots$};
\node at (7,8) {$\vdots$};
\node at (-0.3,8) {$\vdots$};

\foreach \x in \coords \node[color=green!50!black] at (\x+0.2,0.7) {$\scriptstyle 4,\x$};
\foreach \x in {2,...,8} \node[color=green!50!black] at (\x+0.2,1.7) {$\scriptstyle 6,\x$};
\foreach \x in {3,...,8} \node[color=green!50!black] at (\x+0.2,2.7) {$\scriptstyle 8,\x$};
\foreach \x in {4,...,8} \node[color=green!50!black] at (\x+0.2,3.7) {$\scriptstyle 10,\x$};
\foreach \x in {5,...,8} \node[color=green!50!black] at (\x+0.2,4.7) {$\scriptstyle 12,\x$};
\foreach \x in {6,...,8} \node[color=green!50!black] at (\x+0.2,5.7) {$\scriptstyle 14,\x$};
\foreach \x in {7,...,8} \node[color=green!50!black] at (\x+0.2,6.7) {$\scriptstyle 16,\x$};

\node[color=green!50!black] at (1.2,1.6) {$\scriptstyle 5,1$};
\node[color=green!50!black] at (2.2,2.6) {$\scriptstyle 7,2$};
\node[color=green!50!black] at (3.2,3.6) {$\scriptstyle 9,3$};
\node[color=green!50!black] at (4.2,4.6) {$\scriptstyle 11,4$};
\node[color=green!50!black] at (5.2,5.6) {$\scriptstyle 13,5$};
\node[color=green!50!black] at (6.2,6.6) {$\scriptstyle 15,7$};

\draw[rounded corners=1.5mm] (0.8,1.8)--(1.2,1.8)--(1.2,3.2)--(0.8,3.2)--cycle;
\draw[rounded corners=1.5mm] (1.8,2.8)--(2.2,2.8)--(2.2,5.2)--(1.8,5.2)--cycle;
\draw[rounded corners=1.5mm] (2.8,3.8)--(3.2,3.8)--(3.2,7.2)--(2.8,7.2)--cycle;

\draw[rounded corners=1.5mm] (3.8,7.8) -- (3.8,4.8) -- (4.2,4.8) --(4.2,7.8);
\draw[rounded corners=1.5mm] (4.8,7.8) -- (4.8,5.8) -- (5.2,5.8) --(5.2,7.8);
\draw[rounded corners=1.5mm] (5.8,7.8) -- (5.8,6.8) -- (6.2,6.8) --(6.2,7.8);

\end{tikzpicture}
\end{center}
\caption{Classifying Type $D$ flops}\label{fig:TypeDintro}
\end{figure}
In particular, Type $D$ flops do not admit moduli.  Furthermore, the following hold.
\begin{enumerate}
\item The quasi-homogeneous Type D flops are precisely those outside the shaded region, and these are the standard Laufer family.   
\item The \textnormal{GV} invariants $n_1,n_2$ of the flopping contraction associated to a point $(n,m)$ are written in green. The ovals group together flops with the same \textnormal{GV} invariants.
\end{enumerate}
\end{thm}
It is possible to instead index the GV invariants to the classifying potentials, which we do in Figure~\ref{fig!Ddiagram} on page~\pageref{fig!Ddiagram}.  Either way, the important point is that not all pairs of GV invariants $n_1,n_2$ can be realised.
\begin{cor}[\ref{GV for D}]
There are no simple flopping contractions with \textnormal{GV} invariants $5,n$ with $n\geq 2$.  Similarly for $2m+1,n$ with $m\geq 2$ and $n\neq m-1$. 
\end{cor}

More generally, the $E_{6,n}$ normal forms in Table~\ref{tab!zero} predict the first ever infinite family of $E_6$ flops, and indeed this family turns out to exist. Furthermore, the various other type $E$ normal forms not stated precisely in Table~\ref{tab!zero} both predict and classify $E_7$ and $E_8$ flops. Details will appear elsewhere \cite{BW2}, with the point being that noncommutative singularity theory predicts that GV invariants are extremely constrained.

\subsection{Other cDV Applications}
Our results also have applications to 3-fold divisorial contractions to a curve.  Whilst there is an extensive literature \cite{Tziolas1,Tziolas2,ducat16} on extremal ($K$\nobreakdash-negative) divisorial contractions in the presence of terminal singularities, the $K$\nobreakdash-trivial case considered here is much less studied, aside from the notable  \cite{Wilson}. 

In contrast to the previous section \S\ref{geo cor section intro}, the analogue of the Donovan--Wemyss \cite{DW1} conjecture for divisorial contractions to a curve remains open.  A positive solution to this conjecture would immediately imply that the normal forms in Table~\ref{tab!one} give the full classification of 3-fold divisorial contractions to a curve, exactly as in \S\ref{geo cor section intro} where the normal forms in Table~\ref{tab!zero} classified flops.  For now, whilst the generalised conjecture remains open, our noncommutative normal forms still have many consequences, and give unexpected predictions.

\medskip
First of all, our control of normal forms allows us to constrain possible deformations of the fibres, by classifying the  contraction algebras that can arise.

\begin{prop}[\ref{the only Type A CAs d2c}, \ref{the only Type D CAs d2c}]\label{intro:divcurclass}
The only contraction algebras for Type~$A$ and~$D_4$ smooth divisor-to-curve contractions are, up to isomorphism, the Jacobi algebras of the Type $A$ and $D$ potentials in Table~\textnormal{\ref{tab!one}}.
\end{prop}

The normal forms in Table~\ref{tab!one}, together with \ref{intro:divcurclass} and the generalised Donovan--Wemyss conjecture, then predict the first and only infinite family of Type $D$ divisorial contractions to a curve.  None of these have been seen before; the following realises the whole family.

\begin{prop}[\ref{D4 div to curve main}]\label{D4 div to curve intro}
Consider the element of $\mathbb{C}\llsq X,Y,Z,T\rrsq$ defined by 
\renewcommand{\arraystretch}{1}
\[
F_m\colonequals
\left\{
\begin{array}{rl}
 Y(X^{m}+Y)^2 + XZ^2 -T^2& \mbox{if }m\geq 1\\
 Y^3 + XZ^2 -T^2& \mbox{if }m=\infty\\
 \end{array}
 \right.
\]
and set $\scrR_m=\mathbb{C}\llsq X,Y,Z,T\rrsq/F_m$. Then the following statements hold.
\begin{enumerate}
\item $\Sing(\scrR_m)^{\mathrm{red}}=(X^M+Y,Z,T)$ if $m\geq 1$, and $(Y,Z,T)$ if $m=\infty$.
\item In either case, blowing up this locus gives rise to a crepant Type~$D$ divisorial contraction to a curve $\scrX_m\to\Spec \scrR_m$ where $\scrX_m$ is smooth.
\item The contraction algebra of $\scrX_m\to\Spec \scrR_m$ is isomorphic to $\Jac(xy^2+x^{2m+1})$ when $m\geq 1$, respectively $\Jac(xy^2)$ when $m=\infty$.
\end{enumerate}
\end{prop}

Thus in call cases the noncommutative forms $D_{\infty,*}$ are geometrically realised by~$F_*$.  The case $m=\infty$ appeared in \cite[2.18]{DW4}, the infinite family is new. 

More generally, the rather lonely $E_{6,\infty}$ normal form in Table~\ref{tab!one} predicts a divisorial contraction to a curve of Type $E_6$.  This also turns out to exist, and details will again appear elsewhere \cite{BW2}.  In fact, all the evidence now strongly suggests that $E_{6,\infty}$ is the final potential satisfying $\JRdim\Jac(f)\leq 1$, which gives the striking geometric prediction that divisorial contractions to a curve of Type $E_7$ and $E_8$ do not exist.

%-----------------------------------------------------------------------------------------
\subsection{Contractibility and Realisation}\label{geo intro sect}
It was conjectured in \cite{Kawamata_survey,BW3} that in addition to being the classifying structure of contractible curves, noncommutative deformation theory (implicit in the above) also detects which curves can be contracted.  Specifically, the conjecture asserts that a collection of crepant rational curves contracts to a point suitably locally, without contracting a divisor, if and only if its associated noncommutative deformation algebra is finite dimensional.   This should be viewed as a wide-ranging generalisation of celebrated work of Artin \cite{Artin} (for surfaces) and Jim\'{e}nez \cite{Jimenez}.%The conjecture, referred to as the 3-fold contractibility conjecture, is formulated generally for any collection of rational curves.  

The key test case is when the curve is irreducible.  One consequence of Appendix~\ref{J appendix} is that the conjecture is very reasonable: the only open case is now that of $(-3,1)$ curves.

\begin{thm}[\ref{thm:lastcontractmain}]
Let $\mathrm{C}\subset\scrX$ be an irreducible rational curve in a smooth \textnormal{CY} \textnormal{3}-fold, with \textnormal{NC} deformation algebra $\Lambda_{\mathrm{def}}$, such that $\scrN_{\mathrm{C}|\scrX}\neq (-3,1)$.  Then $\mathrm{\Curve}\subset\scrX$ contracts to a point suitably locally, without contracting a divisor, if and only if $\dim_\mathbb{C}\Lambda_{\mathrm{def}}<\infty$.
\end{thm}

Another consequence of this paper is a prediction regarding realising Jacobi algebras from geometry.  We will call $f\in\ringd$ \emph{geometric} if it arises from geometry, that is, $\Jac(f)$ isomorphic to the contraction algebra of some  $\scrX\to\Spec\scrR$ described in \S\ref{geo cor section intro}.   Based partly on the results in this paper, and partly on extensive computer algebra searches using the software \cite{magma, DGPS}, we conjecture the following.

\begin{conj}[The Realisation Conjecture]\label{realisation conj}
Every $f\in\ringd$ whose Jacobi algebra satisfies $\JRdim\Jac(f)\leq 1$ is geometric.
\end{conj}

The conjecture being true would imply that every finite dimensional $\Jac(f)$ is symmetric \cite[2.6]{August}, that is $\Hom_{\C}(\Jac(f),\C)\cong\Jac(f)$ as bimodules, a property which itself is far from clear.  In 2014 our original expectation was that contraction algebras are a strict subset of Jacobi algebras and the task was to recognise them, but since then all computer searches and all papers (e.g.\ \cite{Davison}) which have tried to disprove the conjecture have inadvertently ended up giving more evidence for it.  This paper is no different. 
 \begin{cor}[\ref{realisation main}]\label{realisation main intro}
Conjecture~\textnormal{\ref{realisation conj}} is true, except possibly for the one remaining unresolved case when $f\cong x^3+\scrO_4$, where some further analysis is required.
\end{cor}
 
In the remaining cases, it does now seem likely that all potentials $f\cong x^3+\scrO_4$ for which $\JRdim\Jac(f)\leq 1$ are isomorphic to contraction algebras of $cE_n$ singularities.

\subsection{Notation and Conventions}\label{conventions}
Throughout we work over the complex numbers $\C$, which is necessary for various statements to hold, although any algebraically closed field of characteristic zero would suffice. In addition, we adopt the following notation.

\begin{enumerate}
\item Throughout $d\geq 1$ is fixed to be the number of variables. Set $\mathsf{x}=x_1,\hdots,x_d$, and $\ringd=\mathbb{C}\llangle x_1,\hdots,x_d\rrangle$.
\item
Vector space dimension will be written  $\dim_{\mathbb{C}}V$.
\item
$\ringd_i$ or $\fringd_i$ will denote the vector subspace of $\ringd$ consisting of homogeneous degree~$i$ polynomials. For a formal power series $g\in\ringd$ we denote the graded (necessarily polynomial) piece of degree~$i$ of $g$ by $g_i\in\ringd_i$.
\item
Write $g_{<d} = \sum_{i<d}g_i$ and $g_{>d}=\sum_{i>d}g_i$, with natural self-documenting variations such as $g_{\ge d}$. 
Thus, for example, $g=g_3+g_4+g_{\ge5}$ is a power series with no terms in degrees 0, 1 and~2, and no further conditions.
\item
Given $g,h\in\ringd$, write $g = h + \bigO_d$ as a shorthand for $g_{<d}=h_{<d}$. 
\item
The previous conventions on degree introduce one typographical difficulty, namely the compatibility with sequences.  We will frequently work with sequences $(\sff_n)_{n\ge1}$ of power series $\sff_n\in\ringd$, and analogously we write $(\sff_n)_d$, $(\sff_n)_{<d}\in\fringd$ and $(\sff_n)_{>d}\in\ringd$ for its pieces in the indicated degrees.
To scrupulously avoid confusion, we will systematically use Greek font $\sff_n$ 
to denote the $n$th power series in a sequence, and not the $n$th degree graded piece of a single power series.
\item
The notation \S$x.y$ refers to Subsection $x.y$, $(n.m)$ refers to displayed equation $(n.m)$, and $n.m$ refers to statement~$n.m$, where the type of statement -- Definition, Theorem, and so on -- is usually left unspecified.
\end{enumerate}

\subsection*{Funding} The authors were supported by grants~EP/R009325/1 and EP/R034826/1, and by the ERC Consolidator Grant 101001227 (MMiMMa).

\subsection*{Open Access} For the purpose of open access, the authors have applied a Creative Commons Attribution (CC:BY) licence to any Author Accepted Manuscript version arising from this submission.

\subsection*{Acknowledgements}
It is a pleasure to thank Tom Bridgeland, Kenny Brown, Ben Davison, Okke van Garderen, Osamu Iyama, Natalia Iyudu, Yujiro Kawamata, Miles Reid, Agata Smoktunowicz and Geordie Williamson for helpful discussions and correspondence, and several anonymous reviewers for their helpful remarks.

%-----------------------------------------------------------------------------------------
\section{Formal Automorphisms}
\label{sec!automorphisms}

This section reviews properties of the noncommutative formal power series $\ringd$, and also constructions of various automorphisms of $\ringd$, mainly following \cite[\S2]{DWZ}.  From the viewpoint of noncommutative singularity theory, it is the construction in \S\ref{chasing subsection} leading to \ref{isomain}\eqref{main chase cor} that will be used heavily in later sections.

\subsection{Polynomial and Power Series Notation}\label{sect: power notation}
As in the introduction, write $\ringd$ for formal noncommutative power series in $d$ variables, and further write $\fringd=\polyringd$ for the free algebra in $d$ variables.  For either $f\in\fringd$ or $\ringd$
write $f$ in terms of its homogeneous pieces as
\[
f=f_0+f_1+f_2+f_3+f_4+\hdots,
\]
and define the \emph{order} of $f$ to be $\ord(f) = \min\{i\mid f_i\neq 0\}$, where by convention $\ord(0) = \infty$.   For any $t\geq0$ set $\ringd_{\geq t}=\{f\in\ringd\mid f_i=0\mbox{ if }i<t\}=\{f\in\ringd\mid\ord(f)\geq t\}$, and note that this contains the zero element.

%-----------------------------------------------------------------------------------------
\subsection{Complete Completions}\label{completion section}
To fix notation,
let $\m=(x_1,\dots,x_d)$ denote the two-sided maximal ideal of the free algebra $\fringd$. 
The $\m$-adic completion of~$\fringd$ is 
\[
\varprojlim \fringd/\m^n
\] 
which is the set of sequences $(a_n)_{n\geq 1}$ of $a_n \in \fringd/\m^n$ that satisfy $a_{n+1} + \m^n = a_n + \m^n$ for all $n$, sometimes called {\em coherent sequences}.  

On the other hand, consider the formal power series ring $\ringd$ 
in noncommutative variables $x_1,\dots,x_d$, with two-sided maximal ideal $\hatM$ containing those power series with zero constant term.  There is an isomorphism
\[
\ringd\cong \varprojlim \fringd/\m^n
\]
which sends a formal power series $f$ to the coherent sequence $(f_{<n} + \m^n)_{n\geq 1}$.  Below we will freely make this identification, and further that the following diagrams for all $i\ge j$ form an inverse limit system
\[
\begin{tikzpicture}[node distance=1.2cm]
 \node (cA) {$\ringd$};
 \node (AMi)  [below of=lim, left of=cA]   {$\fringd/\m^i$};
 \node (AMj)  [below of=lim, right of=cA]   {$\fringd/\m^j$};
 \draw[->] (cA)   -- node [left] {$\scriptstyle \uppi_{i}$} (AMi);
 \draw[->] (cA)   --  node [right] {$\scriptstyle \uppi_{j}$} (AMj);
 \draw[->] (AMi)   --(AMj);
\end{tikzpicture}
\]
where the map $\uppi_i$ sends $f\mapsto f_{<i}+\m^i$, and the horizontal map is the natural one.

Given a sequence $(\sff_i)_{i\ge1}$ of elements of $\ringd$, and $f\in\ringd$, recall the following:

\begin{itemize}
\item
$(\sff_i)$ \emph{converges to} $f$ if  $\forall\,n\geq 1, \exists\, N$ such that $\sff_i-f\in \hatM^n$ for all $i\ge N$.
\item $(\sff_i)$ is \emph{Cauchy} if $\forall\,n\geq 1, \exists\, N$ such that $\sff_i-\sff_j\in \hatM^n$ for all $i,j\ge N$.
\end{itemize}
Taking completions of non-noetherian rings in general can be subtle.
However, in the situation here, since for all $i$,
\[
\hatM^i=\Ker(\uppi_{i})=\{ f\in\ringd \mid f_0=\hdots=f_{i-1}=0\},
\]
it is clear that $\ringd$ is complete with respect to its $\n$-adic topology. That is, every Cauchy sequence in $\ringd$ converges.

 The algebra $\ringd$ is a \emph{topological algebra} with basis of the topology given by the ideals $\{\hatM^i\}$, where $\hatM^i$ is both open and closed.
The free algebra $\fringd$ embeds as a dense subalgebra of $\ringd$,
 and the ideal $\hatM^n$ is the closure of~$\m^n$, or equivalently, $\n^n$ is
the smallest closed ideal that contains all monomials $x_1^{a_1}\hdots x_d^{a_d}$ of
degree $\sum_{k=1}^d a_k = n$.

%-----------------------------------------------------------------------------------------
\subsection{Formal Automorphisms}
As input, consider a sequence of algebra isomorphisms $(\upphi_i\colon  \fringd/\m^i\to \fringd/\m^i)_{i\geq 1}$ for which
\begin{equation}
\begin{array}{c}
\begin{tikzpicture}
\node (A) at (0,0) {$\fringd/\m^i$};
\node (B) at (2,0) {$\fringd/\m^j$};
\node (a) at (0,-1.5) {$\fringd/\m^i$};
\node (b) at (2,-1.5) {$\fringd/\m^j$};
\draw[->] (A)--node[left] {$\scriptstyle \upphi_i$}(a);
\draw[->] (B)--node[right] {$\scriptstyle \upphi_j$}(b);
\draw[->] (A)--(B);
\draw[->] (a)--(b);
\end{tikzpicture}
\end{array}\label{need to comm}
\end{equation}
commutes for all $i\geq j$. Then the universal property for the $\m$-adic completion lifts these to an algebra automorphism $\upphi\colon \ringd\to\ringd$ such that the following diagram commutes:
\begin{equation}
\begin{array}{c}
\begin{tikzpicture}
 \node (cA) at (1.5,1) {$\ringd$};
\node (A) at (0,0) {$\fringd/\m^i$};
\node (B) at (3,0) {$\fringd/\m^j$};
 \node (ca) at (1.5,-1) {$\ringd$};
\node (a) at (0,-2) {$\fringd/\m^i$};
\node (b) at (3,-2) {$\fringd/\m^j$};
\draw[->] (A)--node[left] {$\scriptstyle \upphi_i$}(a);
\draw[->] (B)--node[right] {$\scriptstyle \upphi_j$}(b);
\draw[->,black!50!white] (A)--(B);
\draw[->] (a)--(b);
 \draw[->] (cA)   -- node [pos=0.4,left] {$\scriptstyle \uppi_{i}$} (A);
 \draw[->] (cA)   --  node [pos=0.4,right] {$\scriptstyle \uppi_{j}$} (B);
 \draw[->] (ca)   -- node [pos=0.4,left] {$\scriptstyle \uppi_{i}$} (a);
 \draw[->] (ca)   --  node [pos=0.4,right] {$\scriptstyle \uppi_{j}$} (b);
 \draw[->] (cA) -- node[pos=0.7,right] {$\scriptstyle \upphi$}(ca);
\end{tikzpicture}
\end{array}\label{uni prop hom}
\end{equation}
The following special case will be important later.  For any fixed $\sff_1,\dots,\sff_d\in\hatM^2\subset\ringd$, consider the algebra homomorphisms
\[
\phi_i\colon \fringd/\m^i\to \fringd/\m^i
\]
defined by sending $x_k+\m^i\mapsto x_k+(\sff_k)_{<i} +\m^i$ for each $1\le k\le d$.  On the truncated finite dimensional algebras $\fringd/\m^i$, clearly each $\phi_i$ is an algebra isomorphism, and further since the truncation of a truncation is itself a truncation, \eqref{need to comm} applied to the $\phi_i$ commutes.  As a consequence, \eqref{uni prop hom} induces an automorphism $\phi\colon \ringd\to\ringd$.

\begin{defin}\label{unitriangular def}
Given $\sff_1,\dots,\sff_d\in\hatM^2$, the above $\phi\colon \ringd\to\ringd$ is called a \emph{unitriangular automorphism}.  We will abuse notation slightly and write 
\[\arraycolsep=2pt
\begin{array}{rcl}
\ringd&\to&\ringd\\
x_k&\mapsto& x_k+\sff_k
\end{array}
\]
for $\phi$, since indeed $\phi$ is induced by such morphisms on the truncations $\fringd/\m^i$.  For $e\ge1$ we say that $\phi$ has \emph{depth} $e$ provided that $\sff_1,\dots,\sff_d\in\hatM^{e+1}$.
%By convention, when we say $\phi$ has depth $e\le0$ then $\phi$ is the identity map.
\end{defin}

\begin{lemma}\label{lem:depth}
With notation as above, the following statements hold.
\begin{enumerate}
\item\label{lem:depth 1}
A $\mathbb{C}$-algebra homomorphism
$\phi\colon\ringd\to\ringd$ is a unitriangular automorphism of depth~$e\ge1$
if and only if $\phi(f)_{\le e}=f_{\le e}$ for every $f\in\ringd$.
\item\label{lem:depth 2}
If $\phi$ and $\uppsi$ are unitriangular automorphisms of $\ringd$ of depth $e_1\ge1$ and $e_2\ge1$ respectively, then their composition $\uppsi\circ\phi$ is a unitriangular automorphism, of depth $\min\{e_1,e_2\}$.
\end{enumerate}
\end{lemma}

\begin{remark}\label{max to max}
Any homomorphism $\phi\colon\ringd\to\ringd$ is continuous. Indeed, $\phi^{-1}(\n)$ is the kernel of the surjective composition
\[
\ringd\buildrel{\phi}\over{\longrightarrow}\ringd\longrightarrow\ringd/\n
\]
hence $\ringd/\phi^{-1}(\n)\cong\mathbb{C}$ and so $\phi^{-1}(\n)=\n$ since $\n$ is the unique maximal ideal.  In particular, in the language of \cite[5.10]{Warner}, any algebra automorphism of $\ringd$ is automatically a topological isomorphism, since its inverse is automatically continuous.
\end{remark}

%-----------------------------------------------------------------------------------------
\subsection{Limits of unitriangular automorphisms}\label{limit hom section}
Under specific situations, it is possible to build a sequence of automorphisms $\phi^1,\phi^2,\hdots$ of $\ringd$, and take their limit.

For this, consider any  $d$ sequences $(\sfg_i^{1})_{i\geq 1},\dots,(\sfg_i^{d})_{i\geq 1}$,
where each $\sfg_i^k\in\hatM^{i+1}$.
By \ref{unitriangular def} these give rise to a sequence of unitriangular automorphisms 
$\phi^1,\phi^2,\hdots$ where
\[\arraycolsep=2pt
\begin{array}{rcl}
\phi^i\colon\ringd&\to&\ringd\\
x_k&\mapsto& x_k+\sfg_i^{k}.
\end{array}
\]
Again, the above are induced from the corresponding maps $x_k+\m^j\mapsto x_k+(\sfg_i^{k})_{<j} +\m^j$ on the truncations $\fringd/\m^j$, and where each $\phi^i$ has depth~$i$.
To ease the subscripts in the notation below, we will also write $\phi^i$ for these morphisms viewed on the truncations.  

Given this abuse of notation, for all $i\ge j\geq 1$ we claim that the following diagram commutes,
where if $i=1$ or $j=1$ then the corresponding vertical map is the identity.
\begin{equation}
\begin{array}{c}
\begin{tikzpicture}
\node (A) at (0,0) {$\fringd/\m^i$};
\node (B) at (2,0) {$\fringd/\m^j$};
\node (a) at (0,-1.5) {$\fringd/\m^i$};
\node (b) at (2,-1.5) {$\fringd/\m^j$};
\draw[->] (A)--node[left] {$\scriptstyle \phi^{i-1}\circ\cdots\circ\phi^{1}$}(a);
\draw[->] (B)--node[right] {$\scriptstyle \phi^{j-1}\circ\cdots\circ\phi^{1}$}(b);
\draw[->] (A)--(B);
\draw[->] (a)--(b);
\end{tikzpicture}
\end{array}\label{need to comm 2}
\end{equation}
To see this, note that since each $\sfg_i^k\in\hatM^{i+1}$, it follows (in the case $i>j$) that the bottom square in the following diagram commutes:
\[
\begin{array}{c}
\begin{tikzpicture}[yscale=1.2]
\node (A) at (0,0) {$\fringd/\m^i$};
\node (B) at (2,0) {$\fringd/\m^j$};
\node (a) at (0,-1) {$\fringd/\m^i$};
\node (b) at (2,-1) {$\fringd/\m^j$};
\node (a2) at (0,-1.8) {$\vdots$};
\node (b2) at (2,-1.8) {$\vdots$};
\node (a25) at (0,-1.8) {};
\node (b25) at (2,-1.8) {};
\node (a3) at (0,-2.7) {$\fringd/\m^i$};
\node (b3) at (2,-2.7) {$\fringd/\m^j$};
\node (a4) at (0,-3.7) {$\fringd/\m^i$};
\node (b4) at (2,-3.7) {$\fringd/\m^j$};
\draw[->] (A)--(B);
\draw[->] (a)--(b);
\draw[->] (A)--node[left] {$\scriptstyle \phi^{1}$}(a);
\draw[->] (B)--node[right] {$\scriptstyle \phi^{1}$}(b);
\draw[->] (0,-1.3)--node[left] {$\scriptstyle \phi^{2}$}(a25);
\draw[->] (2,-1.3)--node[right] {$\scriptstyle \phi^{2}$}(b25);
\draw[->] (0,-2.1)--node[left] {$\scriptstyle \phi^{j-1}$}(a3);
\draw[->] (2,-2.1)--node[right] {$\scriptstyle \phi^{j-1}$}(b3);
\draw[->] (a3)--(b3);
\draw[->] (a3)--node[left] {$\scriptstyle \phi^{i-1}\circ\cdots\circ\phi^{j}$}(a4);
\draw[->] (b3)--node[right] {$\scriptstyle \phantom{\phi^{i-1}\hdots\circ\phi^{j}}$}(b4);
\draw[->] (b3)--node[right] {$\scriptstyle \Id$}(b4);
\draw[->] (a4)--node[above]{$\scriptstyle $}(b4);
\end{tikzpicture}
\end{array}
\]
Since we are abusing notation, the higher squares commute simply since the truncation of a truncation is itself a truncation. Thus all squares commute, establishing \eqref{need to comm 2}.

Setting $\upvartheta_i\colonequals \phi^{i-1}\circ\cdots\circ\phi^{1}\colon \fringd/\m^i\to \fringd/\m^i$, again with the convention that $\upvartheta_1=\Id$, then each $\upvartheta_i$ is an automorphism since each $\phi^t$ is. Thus \eqref{need to comm} induces, through \eqref{uni prop hom}, an automorphism of $\ringd$ such that for all $i\geq j$ the following diagram commutes.
\begin{equation}
\begin{array}{c}
\begin{tikzpicture}
 \node (A1) at (0,0) {$\ringd$};
 \node (A2) at (3,0) {$\ringd$};
\node (a1) at (0,-1.5) {$\fringd/\m^i$};
\node (a2) at (3,-1.5) {$\fringd/\m^i$};
\draw[->, densely dotted] (A1)--node[above]{$\scriptstyle \exists$}(A2);
\draw[->] (A1)--node[left] {$\scriptstyle \uppi_i$}(a1);
\draw[->] (A2)--node[right] {$\scriptstyle \uppi_i$}(a2);
\draw[->] (a1) -- node[above] {$\scriptstyle \upvartheta_i$}(a2);    
\end{tikzpicture}
\end{array}\label{limit of hom diagram}
\end{equation}
Write $\varprojlim \phi^n\cdots\phi^1$ for this induced automorphism.

\begin{lemma}\label{LHS limit}
With notation and assumptions as directly above, for any $f\in\ringd$ the sequence $(\phi^n\cdots\phi^1(f))_{n\geq 1}$ has limit $\varprojlim \phi^n\cdots\phi^1(f)$.
\end{lemma}
\begin{proof}
Set $F=\varprojlim\phi^n\cdots\phi^1$, then it suffices to prove that for all $n\geq 1$, there exists $N$ such that $\phi^t\cdots\phi^1(f)-F(f)\in\hatM^n$ for all $t\geq N$.  This follows since for all $i> n$
\[
F(f)+\n^n
\stackrel{\scriptstyle\eqref{limit of hom diagram}}{=}
\phi^{n-1}\cdots\phi^1(f)+\n^n
\stackrel{\scriptstyle\eqref{need to comm 2}}{=}
\phi^{i-1}\cdots\phi^1(f)+\n^n.\qedhere
\]
\end{proof}

%-----------------------------------------------------------------------------------------
\subsection{Closure and Cyclic Permutation}

\begin{defin}\label{def:closure}
For any subset $\scrS\subset \ringd$, its \emph{closure} is defined to be
\[
\overline{\scrS}=\bigcap_{i= 0}^{\infty}(\scrS + \hatM^i).
\]
That is, $b\in\overline{\scrS}$ if and only if for all $i\geq 0$, there exists $s_i\in\scrS$ such that $b-s_i\in\hatM^i$.
\end{defin}

\begin{notation}
For $\scrA\colonequals\ringd$, consider $\{ \scrA,\scrA\}$, the commutator vector space of $\ringd$. That is, elements of $\{ \scrA,\scrA\}$ are finite sums
\[
\sum_{i=1}^n\uplambda_i(a_ib_i-b_ia_i)
\]
for elements $a_i,b_i\in\ringd$ and $\uplambda_i\in\C$.
Write $\llcurve \scrA,\scrA\rrcurve$ for the closure of the commutator vector space $\{ \scrA,\scrA\}$. 
Note that $\llcurve \scrA,\scrA\rrcurve$ is only a vector space, not an ideal.
\end{notation}

\begin{defin}\label{DWZ cyclic fact}
Two elements $f,g\in \ringd$ are called cyclically equivalent, or $f$ is said to cyclically permute to $g$, if $f-g\in\llcurve \scrA,\scrA\rrcurve$.  We write $f\sim g$ in this case.
\end{defin}

\begin{remark}
This notion of cyclic equivalence applied a pair of polynomials is finite and elementary:
it is generated over~$\C$ by commutators $[m_1,m_2]$ of monomials $m_i\in \fringd$.
With that in mind, \ref{DWZ cyclic fact} is then the natural notion for formal power series, as $f\sim g$
means precisely that $f_d\sim g_d$ in every degree~$d$, and no more:
the closure merely handles the possibility that $f$ and $g$ may differ by
infinitely many such operations.
\end{remark}

%-----------------------------------------------------------------------------------------
\subsection{Chasing into Higher Degrees}\label{chasing subsection}
The following will be one of our main techniques for producing normal forms of potentials in $\ringd$.  The basic idea is to start with a given~$f$, then produce an infinite sequence of automorphisms which chase terms into higher and higher degrees. Taking limits then gives a single automorphism which takes $f$ to the desired normal form. The subtle point is that at each stage the automorphisms in \eqref{main chase lemma 2} below only give the desired elements up to cyclic permutation. As such, the content in the following is that, with care, limits interact well with cyclic permutation. 

\begin{thm}\label{main chase lemma}
Let $f\in\ringd$, and set $\mathsf{f}_1=f$.  Suppose that there exist elements $\mathsf{f}_2,\mathsf{f}_3,\hdots$ and automorphisms $\phi^1,\phi^2,\hdots$ such that
\begin{enumerate}
\item\label{main chase lemma 1} Every $\phi^i$ is unitriangular, of depth of $\geq i$, and
\item\label{main chase lemma 2} $\phi^i(\mathsf{f}_i)-\mathsf{f}_{i+1}\in\llcurve \scrA,\scrA\rrcurve\cap\n^{i+1}$, for all $i\geq 1$.
\end{enumerate}
Then $\lim \mathsf{f}_i$ exists, and there exists an automorphism $F$ such that $F(f)\sim \lim \mathsf{f}_i$.
\end{thm}
\begin{proof}
The proof follows the strategy used in \cite[4.7]{DWZ}, but as the axiomatics are different here, we give the full proof. By \S\ref{limit hom section} there is an automorphism $F\colonequals \varprojlim\phi^n\cdots\phi^1$.

Since the depth of $\phi^i$ is $\geq i$, by \ref{lem:depth}\eqref{lem:depth 1} $\phi^i(\mathsf{f}_i)$ differs from $\mathsf{f}_i$ only in degrees $> i$.  By~\eqref{main chase lemma 2}, $\phi^i(\mathsf{f}_i)$ differs from $\mathsf{f}_{i+1}$ only in degrees $> i$.  Hence $\mathsf{f}_{i+1}$ differs from  $\mathsf{f}_i$ only in degrees $> i$, from which it easily follows that $(\mathsf{f}_n)$ is a Cauchy sequence.  Since Cauchy sequences converge in $\ringd$, the limit $\lim \mathsf{f}_i$ exists.

Set $c_i=\phi^i(\mathsf{f}_i)-\mathsf{f}_{i+1}\in\llcurve \scrA,\scrA\rrcurve\cap\n^{i+1}$.  Since $f=\mathsf{f}_1$, it is easy to see that
\begin{align}
\phi^n\cdots\phi^1(f)&=\mathsf{f}_{n+1}+\sum_{t=1}^n\phi^n\cdots\phi^{t+1}(c_t)\nonumber \\
 &=\mathsf{f}_{n+1}+\phi^n\cdots\phi^1\left(\sum_{t=1}^n(\phi^t\cdots\phi^{1})^{-1}(c_t)\right)\label{eq: come back to}
\end{align}
where $\phi^n\cdots\phi^{t+1}$ is the identity when $t=n$.
By \ref{LHS limit} the left hand side has limit $F(f)$.  The first part of the right hand side has limit $\lim \mathsf{f}_i$, which exists by above.  We next claim that the rightmost term has limit $F(g)$, where $g$ is the limit of the sequence $(\sum_{t=1}^n(\phi^t\cdots\phi^{1})^{-1}(c_t))_{n\geq 1}$.  

First, $g$ exists, since by \eqref{main chase lemma 2} $c_i\in\hatM^{i+1}$, and so since automorphisms preserve the maximal ideal, $(\phi^t\cdots\phi^{1})^{-1}(c_t)\in\hatM^{t+1}$ for all $t$.  It follows easily that the sequence $\left(\sum_{t=1}^n(\phi^t\cdots\phi^{1})^{-1}(c_t)\right)_{n\geq 1}$ is Cauchy, and so its limit $g$ exists in $\ringd$. Given this, the fact that the sequence $\left(\phi^n\cdots\phi^1(\sum_{t=1}^n(\phi^t\cdots\phi^{1})^{-1}(c_t))\right)_{n\geq 1}$ has limit $F(g)$ follows, since for all $i>n$
\begin{align*}
F(g)+\n^{n+1}
&=
\phi^{n}\cdots\phi^1(\uppi_{n+1}(g))+\n^{n+1}\tag{by \eqref{limit of hom diagram}}\\
&=
\phi^{n}\cdots\phi^1\left(\sum_{t=1}^n(\phi^t\cdots\phi^{1})^{-1}(c_t)+\hatM^{n+1}\right)+\n^{n+1}\tag{since $(\phi^t\cdots\phi^{1})^{-1}(c_t)\in\hatM^{t+1}$}\\
&=
\phi^{n}\cdots\phi^1\left(\sum_{t=1}^{i}(\phi^t\cdots\phi^{1})^{-1}(c_t)+\hatM^{n+1}\right)+\n^{n+1}\tag{add zero}\\
&=
\phi^{i}\cdots\phi^1\left(\sum_{t=1}^i(\phi^t\cdots\phi^{1})^{-1}(c_t)\right)+\n^{n+1}.\tag{by \eqref{need to comm 2}}
\end{align*}
Combining with \eqref{eq: come back to} and taking limits it follows that
\begin{equation}
F(f)=\lim \mathsf{f}_i + F(g).\label{limit DWZ}
\end{equation}
Now, it is easy to check that automorphisms preserve $\llcurve\scrA,\scrA\rrcurve$, so each term in the sequence $\left(\phi^n\cdots\phi^1(\sum_{t=1}^n(\phi^t\cdots\phi^{1})^{-1}(c_t))\right)_{n\geq 1}$ belongs to $\llcurve\scrA,\scrA\rrcurve$.  But since $\ringd$ is complete, every Cauchy sequence within a closed set has limit in that closed set.  It follows that the limit $g\in\llcurve\scrA,\scrA\rrcurve$.  One final application of the fact that automorphisms preserve $\llcurve\scrA,\scrA\rrcurve$ shows that $F(g)\in\llcurve\scrA,\scrA\rrcurve$, and so  $F(f)\sim \lim \mathsf{f}_i $.
\end{proof}

%------------------------------------------------------------------------------------------
\subsection{Elementary Properties of Closed Ideals}
We finish this section with some technical results on closed ideals that are used throughout
\S\ref{sec!typeD}--\S\ref{geo section}.

\begin{notation}
When $I$ is an ideal, write $\lcl I\rcl$ for its closure (in the sense of \ref{def:closure}), which is again an ideal since the ring operations are continuous.
Note that $\lcl I\rcl$ need not be finitely generated, even if $I$ is.
\end{notation}

For a finite set of elements $\scrS$ in $\ringd$, consider the closed ideal $\lcl\scrS\rcl=\lcl s\mid s\in\scrS\rcl$.
\begin{lemma}\label{closed ideals 101}
Let $\scrS$ be a finite subset of elements in $\ringd$, and $f_1,\dots,f_s\in\ringd$. Then the following statements hold.
\begin{enumerate}
\item\label{closed ideals 101 A} $\lcl f_1,\dots,f_s\rcl=\lcl f_1u_1,\dots,f_su_s\rcl$ for any units $u_1,\dots,u_s\in\ringd$.
\item\label{closed ideals 101 B2} $\lcl f+\lcl \scrS\rcl\,\rcl= \lcl f,s\mid s\in\scrS\rcl/ \lcl \scrS\rcl$ in $\ringd/\lcl \scrS\rcl$.
\item\label{closed ideals 101 C} If $\uppsi\colon\ringd/\lcl\scrS\rcl \to \ringd/ \lcl \scrS\rcl$ is a topological isomorphism which sends $f+\lcl \scrS\rcl \mapsto g+\lcl \scrS\rcl$, 
for two elements $f,g\in\ringd$, then there is an induced topological isomorphism
\[
\frac{\ringd}{\lcl f,s\mid s\in\scrS\rcl}\buildrel{\cong}\over\longrightarrow \frac{\ringd}{\lcl g,s\mid s\in\scrS\rcl}.
\]
\end{enumerate}
\end{lemma}

\begin{proof}
(1) $\lcl f_1,\dots,f_s\rcl$ is the smallest closed ideal containing all $f_i$. Since $f_i=(f_iu_i)u_i^{-1}\in\lcl f_1u_1,\dots,f_su_s\rcl$ for each~$i$, by minimality $\lcl f_1,\dots,f_s\rcl\subseteq\lcl f_1u_1,\dots,f_su_s\rcl$.  Repeating the same argument to $f_iu_i\in\lcl f_1,\dots,f_s\rcl$, the converse inclusion also holds.

\noindent
(2) Certainly  $\lcl f+\lcl \scrS\rcl\,\rcl=I/\lcl \scrS\rcl$ for some ideal $I$, given it is an ideal of the quotient. This ideal $I$ is closed by \cite[5.2]{Warner}, since the map to the quotient is continuous, and hence the inverse image of a closed set is closed.  This closed ideal $I$ contains both $f$ and $\scrS$, and so $\lcl f,s\mid s\in\scrS\rcl \subseteq I$.

On the other hand $\lcl f+\lcl \scrS\rcl\,\rcl$ is the smallest closed ideal containing $f+\lcl \scrS\rcl$.  Setting $A=\ringd$, $J=\lcl \scrS\rcl$, and $H=\lcl f,s\mid s\in\scrS\rcl$, the third isomorphism theorem for topological rings \cite[5.13]{Warner} asserts that there is a topological isomorphism
\[
(A/J)/(H/J)\cong A/H.
\] 
In particular, $A/H$ is Hausdorff, since $H$ is closed in $A$, by \cite[5.7(1)]{Warner} applied to $A$.  This being the case, $H/J$ is closed in $A/J$, by \cite[5.7(1)]{Warner} applied to $A/J$.  Hence $\lcl f,s\mid s\in\scrS\rcl/\lcl\scrS\rcl$ is a closed ideal, which clearly contains $f+\lcl\scrS\rcl$.  By minimality $\lcl f+\lcl \scrS\rcl\,\rcl=I/\lcl \scrS\rcl\subseteq \lcl f,s\mid s\in\scrS\rcl/\lcl\scrS\rcl$ and thus $I\subseteq \lcl f,s\mid s\in\scrS\rcl$.  Combining inclusions, the required equality holds.

\noindent
(3) Since $\uppsi$ is a continuous isomorphism, the closed ideal generated by $f+\lcl \scrS\rcl$ corresponds to the closed ideal generated by $g+\lcl \scrS\rcl$. Thus there is a topological isomorphism
\[
\frac{\ringd/\lcl \scrS\rcl}{\lcl f+\lcl \scrS\rcl\kern 1pt\rcl} \longrightarrow \frac{\ringd/ \lcl \scrS\rcl}{\lcl g+\lcl \scrS\rcl\kern 1pt\rcl}.
\]
Now  by \eqref{closed ideals 101 B2} we have $\lcl f+\lcl \scrS\rcl\,\rcl= \lcl f,s\mid s\in\scrS\rcl/ \lcl \scrS\rcl$, and likewise for $g$.  The statement follows by the third isomorphism theorem for topological rings  \cite[5.13]{Warner}.
\end{proof}

%-----------------------------------------------------------------------------------------
\section{Jacobi Algebras}
\label{sec!jacobi}

\subsection{Differentiation}\label{sec:differentiation}
Consider the $\C$-linear maps $\partial_{i}\colon \ringd\to\ringd$ which simply `strike off' the leftmost $x_i$ of each monomial, in other words act on monomials via the rule
\begin{equation}
\partial_{i}(m)=\begin{cases} n & \text{if } m = x_in \\ 0 & \text{otherwise.} \end{cases}\label{defin:strike off}
\end{equation}
The  $\C$-linear cyclic symmetrisation
map $\sym\colon \ringd\to\ringd$ on monomials sends
\[
x_{i_1}\hdots x_{i_t} \mapsto \sum_{j=1}^t x_{i_j}x_{i_{j+1}}\hdots x_{i_t}\cdot x_{i_1}\hdots x_{i_{j-1}}.
\]
Combining these two gives the cyclic derivatives. These are the $\C$-linear maps $\dcyc_{i}\colon \ringd\to\ringd$ which on monomials send
\begin{equation}
x_{i_1}\hdots x_{i_t} \mapsto 
\partial_{i}\sym(x_{i_1}\hdots x_{i_t})=
\sum_{j=1}^t \partial_{i}(x_{i_j}x_{i_{j+1}}\hdots x_{i_t}\cdot x_{i_1}\hdots x_{i_{j-1}}).\label{Kdiff}
\end{equation}

\begin{defin}\label{defin:Jacobi}
For $f\in\ringd$, the Jacobi algebra is defined to be
\[
\Jac(f) = \frac{\ringd}{\lcl\dcyc_{1}f,\dots, \dcyc_{d}f\rcl}
\]
where $\lcl\dcyc_{1}f,\dots, \dcyc_{d}f\rcl\,\reflectbox{\colonequals}\, \lcl \updelta f\rcl$ is the closure of the two-sided ideal $(\dcyc_{1}f,\dots, \dcyc_{d}f)$.
\end{defin}
In general, the quotient of a complete topological ring by a closed ideal is always separated, but it need not be complete. 

\begin{notation}\label{radical remark}
For any ring $R$, write $\mathfrak{J}(R)$ for its Jacobson radical. If $I$ is any ideal of $R$ contained in $\mathfrak{J}(R)$, then $\mathfrak{J}(R/I)=\mathfrak{J}(R)/I$ (see e.g.\  \cite[4.6]{Lam}).  
\begin{enumerate}
\item\label{radical remark 1} It is clear that $\mathfrak{J}(\ringd)=\n$.  If $f\in\ringd_{\geq 2}$, then $\lcl \updelta f\rcl$ is contained in~$\n$, and so $\mathfrak{J}(\Jac(f))=\n/\lcl \updelta f\rcl$, and furthermore
$\mathfrak{J(\Jac(f))}^n = \big(\n^n+\lcl \updelta f \rcl \big)/\lcl \updelta f\rcl$ for $n\ge2$.

\item\label{radical remark 2} The topology on $\ringd$ is an ideal topology generated by powers of $\n$, so the natural quotient topology on the quotient $\Jac(f)$ is induced by powers of the image of $\n$ in the quotient \cite[5.5]{Warner}.  Thus, by (1), provided $f\in\ringd_{\geq 2}$ then the topology on both $\ringd$ and $\Jac(f)$ is the radical-adic topology.  Since $\lcl\updelta f \rcl$ is closed, $\Jac(f)$ is Hausdorff \cite[5.7(1)]{Warner}. Under extra assumptions it is also complete; see \ref{lovely completion fact}\eqref{lovely3}.
\end{enumerate}
\end{notation}

\begin{remark}\label{def:cyc symm}
A (polynomial or) power series $f = \sum f_i\in\ringd$ is called {\em cyclically symmetric}
if $\sym(f_i) = if_i$ for each graded piece $f_i\in\fringd$.
It is  possible to phrase the whole paper using only cyclically symmetric potentials, however this becomes notationally unmanageable in \S\ref{sec:NC sing theory}--\S\ref{sec!typeD}, since the property of being cyclically symmetric is not preserved under change variables.  Thus from the viewpoint of noncommutative singularity theory, it is much more natural to work with plain old elements of $\ringd$. There are times when passing to cyclically symmetric potentials is convenient, but this is confined entirely to \S\ref{sec: exact pot}.
\end{remark}

\subsection{Dimension}
Being a quotient of formal noncommutative power series, determining which dimension to use for $\Jac(f)$ is a subtle point.

\begin{defin}\label{Jdim defin}
For $f\in\ringd_{\ge2}$, we say that $\Jac(f)$ has \emph{polynomial growth}
(\emph{with respect to~$\mathfrak{J}$}) if there exist $c,r\in\R$ such that $
\dim\Jac(f) / \mathfrak{J}^n \le cn^r$ for all $n\in\N$.
In the case that $\Jac(f)$ has polynomial growth, then the \emph{$\mathfrak{J}$-dimension of $\Jac(f)$} is the degree of that growth, precisely
\[
\JRdim\Jac(f)\colonequals \inf \left\{ r\in\R_{\geq 0} \mid 
\textrm{for some $c\in\R$, }\dim\Jac(f) / \mathfrak{J}^n \le cn^r \text{ for every $n\in\N$} \right\},
\]
and $\JRdim\Jac(f)=\infty$ otherwise.
\end{defin}

The $\mathfrak{J}$-dimension is analogous to the usual dimension of
a commutative noetherian local ring $(A,\m)$, defined as
the degree of the characteristic polynomial $\chi_\m(n)=\ell(A/\m^n)$, where,
in that context, the dimension is necessarily an integer \cite[11.4, 11.14]{AM}.

\begin{lemma}\label{lem:Jdim}
If $\JRdim\Jac(f)\le1$, then either $\JRdim\Jac(f)=0$ or $\JRdim\Jac(f)=1$.
\end{lemma}
\begin{proof}
Certainly $\Jac(f)/\mathfrak{J}^{n+1} \twoheadrightarrow \Jac(f)/\mathfrak{J}^n$ for all $n\geq 1$, with equality if and only if $\mathfrak{J}^{n+1}=\mathfrak{J}^{n}$. If each such map has nontrivial kernel, then
$\dim\Jac(f)/\mathfrak{J}^{n}\ge n$ and so $\JRdim\Jac(f)\ge1$.
Otherwise, by Nakayama's Lemma, $\mathfrak{J}^{n}=0$ for some $n$, hence $\n^{n}\subset \lcl \updelta f\rcl$ and so
$\dim_\C\Jac(f)\le\dim\ringd/\n^{n}=2^{n}-1$ and $\JRdim\Jac(f)=0$.
\end{proof}

\begin{remark}\label{GK bad remark}
The $\mathfrak{J}$-dimension is used throughout, since it is better suited to the complete local situation than the GK dimension \cite{KL}. Indeed, it is well-known that the GK dimension does not behave well with respect to completions. For example, $\GKdim\mathbb{C}\llsq x\rrsq=\infty$ whereas $\JRdim\mathbb{C}\llsq x\rrsq=1$.  Compare \cite[\S3.4]{AB}, and in particular \cite[\S5.6]{AB}. Furthermore, $\JRdim\Jac(f)=0$ if and only if $\dim_\C\Jac(f)<\infty$, a property which does not hold for GK dimension since $\Jac(f)$ is not finitely generated.
\end{remark}

\subsection{Equivalences and Isomorphisms}
In what follows, recognising and producing isomorphisms of Jacobi algebras will be key.  The following techniques will be used extensively. The first is trivial, but worth recording since it gives great flexibility in proofs;
the second two are more substantial with Part \eqref{DWZisom} being \cite[3.7]{DWZ}, and Part \eqref{main chase cor} following from \eqref{DWZisom}, together with~\ref{main chase lemma}.
Recall the notation $f\sim g$ from \ref{DWZ cyclic fact}.
\begin{summary}\label{isomain}
Suppose that $f\in\ringd$.
\begin{enumerate}
\item
If $f$ cyclically permutes to $g$, so $f\sim g$, then $\Jac(f)\cong \Jac(g)$.
\item\label{DWZisom} 
If $\upvarphi\in\Aut \ringd$ then
$\Jac(f)\cong \Jac(g)$, where $g=\upvarphi(f)$.
\item\label{main chase cor}
Set $\mathsf{f}_1=f$. If there exist $\mathsf{f}_2,\mathsf{f}_3,\hdots$ and automorphisms $\phi^1,\phi^2,\hdots$ such that
\begin{enumerate}
\item every $\phi^i$ is unitriangular of depth of $\geq i$, and
\item $\phi^i(\mathsf{f}_i)-\mathsf{f}_{i+1}\in\llcurve \scrA,\scrA\rrcurve\cap\n^{i+1}$, for all $i\geq 1$,
\end{enumerate}
then the sequence $(\sff_i)_{i\ge1}$ converges and $\Jac(f)\cong \Jac(g)$ where $g=\lim \mathsf{f}_i$.
\end{enumerate}
\end{summary}

\begin{lemma}\label{lem!sym}
Let $f\in\ringd$, and $m\in \ringd$ be a monomial. Then the following hold.
\begin{enumerate}
\item\label{lem!sym1}
$\sym(m)\sim \deg(m)m$.
\item\label{lem!sym2}
If $f$ contains $\uplambda m$, then $f \sim f + \uplambda\left(\frac{1}{\deg m}\sym(m) - m\right)$.
\item\label{lem!sym4}
Let $h$ be the sum of terms of $f$ whose monomials appear in $\sym(m)$. Then
\[
f\sim f - h + \upalpha \sym(m) \sim f - h + \upalpha\deg(m)\, m,
\]
for some $\upalpha\in\C$.
\end{enumerate}
\end{lemma}
\begin{proof}
Writing $m=m_1m_2\ldots m_r$, where each $m_i$ is a variable $x_{j(i)}$,
we have
\begin{align*}
rm - \sym(m) &= (r-1)m_1\ldots m_r - m_2\ldots m_rm_1 \\
 &\qquad\qquad- m_3\ldots m_rm_1m_2 - \cdots - m_rm_1\ldots m_{r-1} \\
&= [m_1,m_2\ldots m_r] + [m_1m_2,m_3\ldots m_r] + \cdots + [m_1\ldots m_{r-1},m_r]
\end{align*}
and \eqref{lem!sym1} follows.
\eqref{lem!sym2} follows at once from~\eqref{lem!sym1}.
The final claim \eqref{lem!sym4} follows by applying~\eqref{lem!sym2} to each monomial of $h$ in turn.
\end{proof}

Below it will be convenient to work with the following three equivalence relations.
\begin{defin} 
For elements $f,g\in\ringd$, (recall and) define
\begin{enumerate}
%\item $f \sim g$ if $\dcyc_xf = \dcyc_xg$ and $\dcyc_yf = \dcyc_yg$.
\item
$f\sim g$ if $f-g\in \llcurve\ringd,\ringd\rrcurve$ (see \ref{DWZ cyclic fact}).
\item $f \simeq g$ if there is an equality of ideals $\lcl\dcyc_{1}f,\dots, \dcyc_{d}f\rcl = \lcl\dcyc_{1}g,\dots, \dcyc_{d}g\rcl$.
\item $f \cong g$ if there is an isomorphism of algebras $\Jac(f) \cong \Jac(g)$.
\end{enumerate}
\end{defin}
Clearly $f\sim g$ implies $f\simeq g$ implies $f\cong g$, but the converse implications do not hold. 
The relation $\sim$ is additive by definition, but $\simeq$ is not: $x^2 + y^3 \simeq x^2 + 2y^3$ but $x^2 \not\simeq x^2 + y^3$. 

The Jacobi isomorphism relation $\cong$ is the equivalence relation that we will
classify up to, but the others help understand the structure of the various arguments.
For example, by \ref{main chase lemma}, the symmetrisation relation $\sim$ 
behaves well in limits. It appears to permit creation from the void, in
the sense that $0\sim xy-yx$, but of course this form has all derivatives zero,
so does not contribute to Jacobi ideals.  The relation $\simeq$ is useful for cancelling high order terms in potentials (see e.g.\ the proof of \ref{NC3rootsnormal}), whereas $\cong$ is most suited to, and is often a by-product of, analytic changes in coordinates.

%-----------------------------------------------------------------------------------------
\section{NC Singularity Theory 101}\label{sec:NC sing theory}

%-----------------------------------------------------------------------------------------
\subsection{Corank and the Splitting Lemma}
\label{sec!splitting}

The closed vector subspace of commutators $\llcurve\ringd,\ringd\rrcurve$
generates the much larger closed ideal of commutators, and the quotient of $\ringd$ by this
ideal is the ring of commutative power series $\C\llsq x_1,\dots,x_d\rrsq$.
The quotient, or `abelianisation', map $\ringd\to\C\llsq x_1,\hdots, x_d\rrsq$ written $g \mapsto g^{\ab}$ simply takes the expression for $g$ to the same expression in the commutative ring.

\begin{lemma}\label{lem!jacab}
With notation as above, the following hold:
\begin{enumerate}
\item\label{jacab1}
The abelianisation map $\ringd\to\C\llsq x_1,\hdots, x_d\rrsq$
is continuous and surjective.
\item\label{jacab2}
For any $f\in\ringd$, the map $f\mapsto f^{\ab}$ descends to a surjection
\begin{equation}\label{ontoMilnor}
\Jac(f)\twoheadrightarrow \frac{\C\llsq x_1,\hdots, x_d\rrsq}{\lcl\partial f^{\ab}/\partial x_i\mid i=1,\hdots,n\rcl}.
\end{equation}
\end{enumerate}
\end{lemma}
\begin{proof}
\eqref{jacab1}
At the level of ideals, $\n^k\twoheadrightarrow \n_{\ab}^k$ for every $k\ge0$,
since abelianisation is a ring homomorphism mapping each $x_i$ to $x_i$.

\noindent
\eqref{jacab2}
Since $(\dcyc_{i}f)^{\ab} = \partial f^{\ab}/\partial x_i$, where $\partial/\partial x_i$ is the usual differentiation of commutative functions, surjectivity 
at the level of (unclosed) Jacobian ideals follows. Since the abelianisation map is continuous and surjective by~\eqref{jacab1}, this passes to their closures, as claimed. 
\end{proof}

Below we will consider
\[
\Jac(f)^{\ab} \colonequals  \frac{\C\llsq x_1,\hdots, x_d\rrsq}{\lcl\partial f^{\ab}/\partial x_i\mid i=1,\hdots,n\rcl} = 
\frac{\C\llsq x_1,\hdots, x_d\rrsq}{(\partial f^{\ab}/\partial x_i\mid i=1,\hdots,n)}
\]
where, since $\C\llsq x_1,\hdots, x_d\rrsq$ is commutative noetherian,
all ideals are closed \cite[8.1(1)]{Matsumura}.

\begin{remark}\label{Milnor=Tjurina}
In classical singularity theory, for $g\in\C\llsq x_1,\dots,x_d\rrsq$ both the Milnor algebra $\C\llsq x_1,\dots,x_d\rrsq/(\updelta_1g,\dots,\updelta_dg)$ and the Tjurina algebra $\C\llsq x_1,\dots,x_d \rrsq/(g,\updelta_1g,\dots,\updelta_dg)$ are defined, and play a major role.
In the noncommutative setting, the analogous Tjurina algebra is not well defined on $\sim$ classes.  For example, the potentials $0\sim xy-yx$ determine the
same Jacobi algebra, but their naively-defined Tjurina algebras are
$\ringtwo$ and $\C\llsq x,y \rrsq$ respectively.  To have any hope of classifying elements in the completed free algebra, some identification is required, and for us identifying $\sim$ classes is essential for applications. Compare \cite{HZ}, where the lack of a noncommutative Tjurina algebra motivates the use of Hochschild classes to generalise Saito's theorem on homogeneous potentials. 
\end{remark}

\begin{defin}\label{defin corank}
For $f\in\ringd_{\geq 2}$, the {\em corank of $f$} is defined to be
\[
\corank(f) = \dim_{\C}\left(\frac{\mathfrak{J}}{\mathfrak{J}^2}\right)
\]
where $\mathfrak{J}$ is the Jacobson radical of $\Jac(f)$.
\end{defin}

\begin{remark}\label{rem:corank}
Clearly $0\le\corank(f)\le d$.
Since $\mathfrak{J}/\mathfrak{J^2} \cong (\n+I)/(\n^2+I)$, where $I=\lcl \updelta f\rcl$, 
the exactness of the sequence of $\ringd/\n=\C$-vector spaces
\[
0\longrightarrow \frac{\n^2 + I}{\n^2} \longrightarrow \frac{\n}{\n^2} \longrightarrow \frac{\n+I}{\n^2+I} \longrightarrow 0
\]
shows that $\corank(f) = d - \dim_{\mathbb{C}} \left(\frac{\n^2 +I}{\n^2}\right)$,
so that the corank is determined by the linear conditions imposed by derivatives,
and is therefore uniquely determined by~$f_2$.
\end{remark}

\begin{thm}[Splitting Lemma]
\label{thm!splitting}
Let $f\in\ringd$. Then $f\cong x_1^2+\cdots+x_r^2+g$ for some $g\in\C\llangle x_{r+1},\hdots, x_d\rrangle_{\ge3}$, where $d-r=\corank(f)$. In particular,
\[
\Jac(f) \cong \frac{\C\llangle x_{r+1},\hdots, x_d\rrangle}{\lcl \dcyc_{x_{r+1}}g,\hdots,\dcyc_{x_d}g \rcl}.
\]
\end{thm}
\begin{proof}
This is \cite[4.5]{DWZ} for the $d$-loop quiver.  Since $\ord g\ge3$,
the derivatives of $g$ impose no linear conditions, so necessarily $d-r=\corank(f)$.
\end{proof}

%-----------------------------------------------------------------------------------------
\subsection{Golod--Shafarevich--Vinberg}\label{GSV}

The classical approach to growth of algebras comes
from the Golod--Shafarevich theorem \cite{GolodShafarevich},
adapted by Vinberg \cite{Vinberg} to power series;
see also \cite{Ershov}.
This result constrains~$f$ to achieve $\JRdim\Jac(f)<\infty$,
and we develop a stronger version in \ref{thm: appendix main in text} adapted to Jacobi algebras.

\begin{thm}[Golod--Shafarevich, Vinberg]
\label{thm!GSV}
Let $I = (g_1,\dots,g_s)\subset\ringd$ be an ideal,
set $r_i=\ord g_i$ for each $i=1,\dots,s$,
and write $h = 1 - dt + t^{r_1} + \cdots +t^{r_s} \in \R\llsq t \rrsq$.
If the coefficients of $(1-t)/h$ are non negative, then $\dim_\C \ringd/\lcl I\rcl=\infty$.
\end{thm}
In most cases where the result applies, one can in fact show exponential growth. The Golod--Shafarevich--Vinberg estimates readily show that
$\JRdim\ringd/\lcl g_1,\dots,g_d\rcl=\infty$ in the following cases:
\begin{enumerate}
\item
$d=2$ with
either $r_1\ge3$, $r_2\ge8$,
or $r_1\ge4$, $r_2\ge5$.
\item
$d=3$ with either $r_1\ge2$, $r_2, r_3\ge3$, or $r_1=r_2  = 2$, $r_3\ge5$.
\item
$d\ge4$ with $r_i\ge2$ for every~$i$.
\end{enumerate}
For example, in the case $d=4$, it is sufficient to observe the exponential growth of
\begin{align*}
(1-t)(1-4t + 4t^2)^{-1} &=
(1-t)(1 + 4t + 12t^2 + 32t^3 + 80t^4 + \hdots + (1+k)2^kt^k +\hdots) \\
&=1 + 3t + 8t^2 + 20t^3 + 48t^4 + \hdots + (2+k)2^{k-1}t^k + \hdots
\end{align*}
%\[
%(1-4t+4t^2)^{-1} = 1 + 3t + 8t^2 + 20t^3 + 48t^4 + \cdots + (1+k)2^kt^k + \cdots
%\]
as this bounds the growth of the algebra from below in the case of an order~$3$
potential with four order~$2$ derivatives.
%\end{proof}

Setting aside quadratic terms by the Splitting Lemma, this then puts constraints on
the motivating problem~\ref{motivating problem}. Indeed, if $f\in\ringd_{\geq 3}$ and $\JRdim\Jac(f)<\infty$, then either
\begin{enumerate}
\item
$d=2$, $\ord f\le5$ and $f_{\le5}\not\sim \ell^5$ for a linear form $\ell=\ell(x_1,x_2)$, or
\item
$d=3$, $\ord f=3$, and $f_3\not\sim\ell^3$ for a linear form $\ell=\ell(x_1,x_2,x_3)$.
\end{enumerate}

It turns out that these estimates can be substantially improved, but this requires much more work.  Iyudu and collaborators \cite{ISmok,IS} introduce several new ideas that
exploit the Jacobi structure; in Appendix~\ref{J appendix} we extend their techniques into the  power series context, and establish the following.  Recall that $\mathsf{x}=x_1,\hdots,x_d$.

\begin{thm}[\ref{thm: appendix main}]\label{thm: appendix main in text}
Suppose that $d=2$ and $k\geq 4$, or $d\geq 3$ and $k\geq 3$.  If $f\in\ringd$ has order $k$, then $\JRdim\Jac(f)\geq 3$.
\end{thm}

\begin{remark}\label{rem: down to 2 var}
Together with the Splitting Lemma, the above \ref{thm: appendix main in text} reduces the classification of those $f$ satisfying $\JRdim\Jac(f)\leq 1$ to the case of two variables ($d=2$).
\end{remark}
\subsection{Higher Coranks}\label{sec:highcorank}
Higher-degree versions of the corank exist, and contain more detailed information about Jacobi algebras.
\begin{defin}\label{defin higher corank}
Let $f\in\ringd_{\geq 2}$. For $m\ge1$, the {\em $m$th corank of $f$} is defined to be
\[
\corank_m(f) =\dim_{\C}\left(\frac{\mathfrak{J}^m}{\mathfrak{J}^{m+1}}\right),
\]
where $\mathfrak{J}$ is the Jacobson radical of $\Jac(f)$. 
We also define $\corank_0(f)=\dim_\C\Jac(f)/\fJ = 1$
and note that $\corank_1(f)=\corank(f)$.
\end{defin}

\begin{remark}
Since $\mathfrak{J}^m/\mathfrak{J}^{m+1} \cong (\n^m+I)/(\n^{m+1}+I)$,
the exactness of the sequence of $\ringd/\n=\C$-vector spaces
\[
0\to \frac{\n^{m}\cap I}{\n^{m+1}\cap I}\cong \frac{(\n^{m}\cap I) + \n^{m+1}}{\n^{m+1}} \to \frac{\n^m}{\n^{m+1}} \to \frac{\n^m+I}{\n^{m+1}+I} \to 0
\]
shows that $\corank_m(f) = d^m - \dim_{\mathbb{C}} \left(\frac{\n^{m}\cap I}{\n^{m+1}\cap I}\right)$; compare \cite[(4)]{Vinberg}. Thus the $m$th corank is determined by the conditions imposed on the leading terms of elements of the Jacobian ideal of order exactly~$m$.
In particular, $0\le\corank_m(f)\le d^m$.
If $\ord(f)\ge m+1$, then 
\[
\frac{(\n^{m}\cap I) + \n^{m+1}}{\n^{m+1}} \cong \frac{\n^{m+1}+I}{\n^{m+1}}
\]
matching~\eqref{eqn:corank intro}, \eqref{intro:corankdef}, and~\ref{rem:corank}.

By definition, the $\mathfrak{J}$-dimension is the growth of the sum of coranks.
Calculating the $m$th corank is not necessarily straightforward: essentially
it amounts to calculating a Gr\"obner basis of the Jacobian ideal with a local
monomial order to at least order~$m$.
\end{remark}

To study Jacobi algebras $\Jac(f)$  of $\fJ$-dimension $\le1$, 
\ref{thm: appendix main in text} constrains
the number of variables to $d\le2$ and $k=\ord(f)\le3$.
The corank controls the rank of $f_2$. The main case is when $f_2=0$, 
when it is clear that  $2\le\corank_2(f)\le4$.
The two derivatives $\dcyc_xf_3$ and $\dcyc_yf_3$ are linearly independent
when $\corank_2(f)=2$ and they are
dependent when $\corank_2(f)=3$. The case $\corank_2(f)=4$ holds only when $f_3=0$,
which is ruled out by \ref{thm: appendix main in text}.

This provides a numerical characterisation of the ADE types.
The first Type~A case is when $\corank(f)=0$, in which case $\Jac(f)\cong\C$. 
In addition to this, if $f\in\ringd$ has $\JRdim\Jac(f)\le1$, we say $f$ has Type~A, D or~E according to the following table.
\[
\begin{tabular}{ccc}
\toprule
\textbf{Type} & $\corank(f)$ & $\corank_2(f)$ \\
\midrule
A & 1 & 1 \\
D & 2 & 2 \\
E & 2 & 3\\
\bottomrule
\end{tabular}
\]
The higher coranks provide much more detail.
In Type~A they provide enough information to classify up to isomorphism, however in Type~D this is not true.
\begin{example}\label{ex: d dimensions}
Consider the families $D_{n,\infty}$ and $D_{n,m}$ from the introduction. The higher coranks are given by the following table.
\[
\begin{tabular}{cllp{4cm}p{3.25cm}c}
\toprule
&{\bf Type}&$f$&{\bf Conditions}&$\corank_i(f)$, $i=0,1,\dots$\\
\midrule
& $D_{n,\infty}$ & $xy^2+ x^{2n}$&$n\geq 2$&\multirow{2}*{$1,\overbrace{2,\dots,2}^{2n-1},\overbrace{1,\dots,1}^{2n-2}$ }\\
&$D_{n,m}$ & $xy^2+x^{2n}+x^{2m+1}$ & $2n-2\geq m\geq n\geq2$ & \\[4pt]
&$D_{n,m}$&$xy^2+x^{2m+1}+x^{2n}$ &$n>m\geq 1$& $1,\overbrace{2,\dots,2}^{2n-1},\overbrace{1,\dots,1}^{2m-1} $\\
\bottomrule
\end{tabular}
\]
In particular $\dim_\C\Jac(f)=6n-3$ in the first families, which is independent
of~$m$, whilst $\dim_\C\Jac(f)=4n+2m-2$ in the final family.
\end{example}

%-----------------------------------------------------------------------------------------
\subsection{Linear Changes in Coordinates and Discriminants}\label{linear change section}

In light of~\ref{rem: down to 2 var}, from \S\ref{sec!typeD} onwards we work in two non-commuting variables $x$ and $y$. 

The following is an immediate consequence of the Splitting Lemma and abelianisation.

\begin{lemma}\label{linear changeA}
Let $f\in \mathbb{C}\llangle x,y\rrangle_{\geq 2}$ with $f_2=ax^2 + b_1xy + b_2yx + cy^2\not\sim0$.
Set $b=b_1+b_2$ and consider the discriminant $\De=b^2-4ac$.
Then $f\cong g$, for some $g\in \mathbb{C}\llangle x,y\rrangle_{\geq 2}$ with
\[
g_2=
\begin{cases}
x^2+y^2 &\text{if }\De\ne0\\
x^2 &\text{if }\De=0.
\end{cases}
\]
\end{lemma}

As for the quadratic forms above, up to $\sim$ we may commute variables appearing in
cubic forms in $\ring$, and
we use this to simplify the statement of the following lemma, writing $bx^2y$ rather than
$b_1x^2y + b_2xyx+\dots$, and so on.
Note that, in general,
cyclic equivalence no longer simulates commutativity in higher degree, as $xyxy\nsim x^2y^2$.

\begin{lemma}\cite{Iyudu}\label{linear change}
Let $f\in \mathbb{C}\llangle x,y\rrangle_{\geq 3}$ with 
$f_3\sim ax^3 + bx^2y + cxy^2 + dy^3$ for some $a,b,c,d\in\C$, not all zero.
Let $\De=-27a^2d^2 + 18abcd - 4ac^3 - 4b^3d + b^2c^2\in\C$ be the cubic discriminant.
Then
$f\cong g$, for some $g\in \mathbb{C}\llangle x,y\rrangle_{\geq 3}$ with
\[
g_3=
\begin{cases}
x^3+y^3 &\text{if }\De\ne0\\
x^2y &\text{if }\De=0 \text{ and }
\begin{cases}
a(b^2-3ac)\ne0 \text{ or }\\
(c^2-3bd)d\ne0 \text{ or }\\
a=d=0
\end{cases}\\
x^3 &\text{otherwise}.\\
\end{cases}
\]
Thus these three leading cubic normal forms are characterised by whether
$f_3^{\ab}$ has three, two or one distinct factors respectively. 
\end{lemma}

\begin{proof}
Consider the linear automorphism of $\C[x,y]$
\begin{equation}\label{eq!cubic}
x\mapsto \al x+\be y,\qquad
y\mapsto \ga x + \de y
\qquad\text{for $\al,\be,\ga,\de\in\C$}
\end{equation}
that maps $(f_3)^{\ab}\in\C[x,y]$ to one of the normal forms $x^3+y^3$, $xy^2$ or $x^3$.
The choice of normal form is determined by the cubic determinant. The
additional conditions on the coefficients in the statement are simply that $p\ne0$
in the depressed form after completing the cube
$x^3+pxy^2+qy^3$, in which $\Delta=-4p^3-27q^2$,
and accounting for the fact that $a=0$ or $d=0$ or both are possible.

Let $\phi$ be the linear automorphism of $\ring$ defined by the same formula~\eqref{eq!cubic}
and $g_3\in\ring$ be the corresponding cubic normal form.
Then $\phi(f_3)\sim g_3$, although they are not equal, so $\phi(f)\sim g_3 + \phi(f_{\ge4})$.
Thus $f \cong \phi(f) \cong g_3 + \phi(f_{\ge4})$, as claimed.
\end{proof}

%-----------------------------------------------------------------------------------------
\section{Type \texorpdfstring{$A$}{A} and Commutativity}
\label{sec!typeA}

This section considers the most elementary situation, namely
$f\in\ringd_{\geq 2}$ with large quadratic part.  Normal forms are established in \S\ref{normal A forms}. Together with the linear coordinate changes from \S\ref{linear change section}, this proves in \S\ref{commutativity} that for any $f\in\ringtwo_{\geq 2}$, the algebra $\Jac(f)$ is commutative if and only if  $f$ has corank at most~$1$.  This fact is used in later sections.

\subsection{Normal Forms of Type~A}\label{normal A forms}
It is notationally convenient to identify $y=x_d$ and work in the ring $\ringd=\ringdy$.

\begin{thm}\label{thm!pagoda}
If $f \in\ringd_{\geq 2}$ with $\corank(f)\le1$, 
then there is a unique polynomial $g$ of the form $g = x_1^2+\cdots+x_{d-1}^2+\upvarepsilon y^n$ for some
$n\ge2$ and $\upvarepsilon\in\{0,1\}$ such that $f\cong g$.
\begin{enumerate}
\item If $\upvarepsilon=1$,      %$g=x_1^2+\cdots+x_{d-1}^2+y^n$, 
then $\Jac(f)$ is commutative with $\dim_{\C}\Jac(f) = n-1$.
\item If $\upvarepsilon=0$,    %$g=x_1^2+\cdots+x_{d-1}^2$, 
then $\Jac(f)$ is commutative with $\JRdim\Jac(f) = 1$.
\end{enumerate}
\end{thm}
\begin{proof}
By the Splitting Lemma~\ref{thm!splitting}, there is $\sff\in\ringdy$ with $f\cong\sff$
and either
\[
\sff =
\begin{cases}
x_1^2+\cdots+x_{d-1}^2 + y^2 & \text{if }\corank(f)=0 \\
x_1^2+\cdots+x_{d-1}^2 + q(y) &\text{if }\corank(f)=1
\end{cases}
\]
for some $q\in\C\llsq y \rrsq$ with $\ord(q)\ge3$.

If $q$ is zero we are done, else after pulling out the lowest term, we can write $q = y^n \mathsf{u}$ for some $\mathsf{u}=c_n + c_{n+1}y + \hdots\in\C\llsq y\rrsq$ with $c_n\neq0$. The homomorphism $\ringdy\to\ringdy$ which sends $x_k\mapsto x_k$ and $y\mapsto y\sqrt[n]{\mathsf{u}}$ is an automorphism.  Since $\sqrt[n]{\mathsf{u}}$ is a power series only in $y$, it commutes with $y$, and so this automorphism sends $\sum x_i^2+y^n$ to $\sum x_i^2+y^n\mathsf{u}=\sff$.  Hence $\Jac(\sum x_i^2+y^n)\cong\Jac(\sff)\cong\Jac(f)$, as required.

Parts (1)--(2) are obvious, since $\Jac(\sum x_i^2+y^n)\cong\C\llsq y\rrsq/(y^{n-1})$ and $\Jac(\sum x_i^2)\cong\C\llsq y\rrsq$, and uniqueness then follows since $\sum x_i^2+y^{n_1}\cong 
\sum x_i^2+y^{n_2}$ if and only if $n_1=n_2$.
\end{proof}

Recall from the introduction our geometric applications in the setting of cDV singularities.  These correspond to Jacobi algebras in two non-commutating
variables, so we set $d=2$ and write the variables as $x,y$.

\begin{cor}\label{Type A all geometric}
Every $\Jac(f)$, where $f \in\ring_{\geq 2}$ with $f_2\not=0$, is geometric.
\end{cor}
\begin{proof}
Consider $\scrR=\C\llsq x,y,z\rrsq^{\frac{1}{2}(1,1,0)}$, and its unique crepant resolution $\scrX\to\Spec \scrR$.  This has contraction algebra $\C\llsq y\rrsq\cong\Jac(x^2)$, realising the second case in \ref{thm!pagoda}. On the other hand, by \cite[3.10]{DW1} the Type~$A$ $m$-Pagoda flop (with $m\geq 1$) has contraction algebra $\C\llsq y\rrsq /y^m\cong\Jac(x^2+y^{m+1})$, which realises the infinite family in \ref{thm!pagoda}.
\end{proof}

\begin{example}
Consider $f=x^2+\tfrac{2}3(xy^2+yxy+y^2x)\in\ringtwo$.
This has $3$\nobreakdash-dimensional Jacobi algebra
\[
\Jac(f) \cong \C\llangle x,y \rrangle / \lcl x+y^2, xy+yx\rcl \cong \C[y] / y^3
\]
so, by~\ref{thm!pagoda} or Reid's Pagoda \cite{Pagoda},
$f$ gives the same Jacobi algebra as $g = x^2 + y^4$. Commutatively, 
one would see this by completing the square, but that automorphism does not work directly in
the noncommutative context: $x\mapsto x-\frac{2}3y^2$ gives $f\mapsto x^2 + \frac{2}3yxy - \frac{8}9y^4$, and we cannot attack the $yxy$ term by coordinate changes that preserve $f_2=x^2$.
But $f\sim x^2 + xy^2+y^2x$, which then allows us to complete the square (and a scalar on $y$) to conclude.
This exemplifies the way $\sim$ helps to navigate the Jacobi isomorphism classes.
\end{example}

\subsection{Commutativity}\label{commutativity} 
The following characterisation of commutative Jacobi algebras in $d=2$ variables will be used later.
\begin{prop}\label{when comm}
$f\in\ringtwo_{\geq 2}$, then $\Jac(f)$ is commutative if and only if $\corank(f)\le1$.
\end{prop}
\begin{proof}
($\Leftarrow$) is clear from \ref{thm!pagoda}.  For ($\Rightarrow$), we prove the contrapositive.  If $\corank(f)\not\le1$, then $f_2=0$, and we need to prove that $\Jac(f)$ is not commutative. For this, it suffices to exhibit a factor that is not commutative. By \ref{linear change} without loss of generality we can assume that~$f$ equals 
\[
x^3+y^3+\scrO_4,\quad xy^2 +\scrO_4, \quad  x^3+\scrO_4\quad \mbox{or}\quad \scrO_4.
\]
Write $\scrM_3$ for the set of all noncommutative monomials of degree $3$, then factor by the ideal $\lcl \updelta_xf,\updelta_yf,\scrM_3 \rcl/\lcl \updelta_xf,\updelta_yf\rcl$ in $\Jac(f)$.   But in the four cases above, by differentiating then using the third isomorphism theorem  it follows that $\Jac(f)$ is one of
\[
\frac{\C\llangle x,y\rrangle}{\lcl x^2,y^2,\scrM_3\rcl},\quad
\frac{\C\llangle x,y\rrangle}{\lcl y^2, xy+yx,\scrM_3\rcl},\quad
\frac{\C\llangle x,y\rrangle}{\lcl x^2,\scrM_3\rcl},\quad\mbox{or}\quad
\frac{\C\llangle x,y\rrangle}{\lcl \scrM_3\rcl}.
\]
None of these factors is commutative, and so $\Jac(f)$ is not commutative. 
\end{proof}

%-----------------------------------------------------------------------------------------
\section{Type \texorpdfstring{$D$}{D} normal forms}
\label{sec!typeD}

This section considers the next case, namely those $f\in\ringd_{\geq 2}$ with $\corank(f)=2$ and $\corank_2(f)=2$.  Reducing to two variables by the Splitting Lemma~\ref{thm!splitting}, the assumption $\corank_2(f)=2$ is then equivalent to the first two cases in \ref{linear change}, namely those $f\in\mathbb{C}\llangle x,y\rrangle_{\geq 3}$ with $f_3\neq 0$ for which $f_3^{\ab}$ has either two or three distinct linear factors.   Full normal forms are obtained in both situations, and are then merged into a unified form in \S\ref{sec: overview D}. These are the Type~$D$ normal forms in the tables in \S\ref{tables section}. 

\medskip
Throughout this section, it will be convenient to adopt the following language.

\begin{defin}
We say that a monomial $m\in \ringtwo$ {\it contains~$x^2$}
if $m\sim nx^2$ for some monomial $n$, else $m$ {\it does not contain $x^2$}.
Similarly, an element $f\in \ringtwo$ contains~$x^2$ if for some (nonzero) term $\uplambda m$ of $f$, the monomial $m$ contains~$x^2$, else $f$ does not contain~$x^2$. We also use the analogous expressions for~$y^2$.
\end{defin}

%-----------------------------------------------------------------------------------------
\subsection{Abelianized Cubic with Three Factors}
\label{sec!3roots}
This subsection considers the case of \ref{linear change}
where $f_3^{\ab}$ has three distinct factors, that is $f\cong x^3+y^3+\bigO_4$,
and in \ref{NC3rootsnormal} and \ref{NC3rootsnormal2} provides two different, but equivalent, normal forms.

Recall from~\ref{lem:depth}\eqref{lem:depth 1} that a unitriangular automorphism $\phi$ of $\ringd$ has depth~$e\ge1$ if and only if $\phi(f)_{\le e} = f_{\le e}$ for all~$f\in\ringd$.

\begin{lemma}\label{lemB3}
Fix $t\ge 4$, and let $f=x^3+y^3 + f_4+\cdots + f_t + \bigO_{t+1}$.  For any $h_1,h_2\in\mathbb{C}\llangle x,y\rrangle$ with $\ord(h_i) \geq t-2$, there is a unitriangular automorphism $\uppsi$ of depth $\geq t-3$ such that
\[
\uppsi(f)\sim x^3+y^3 + f_4+\cdots + f_{t-1} + \left(f_t-h_1x^2-h_2y^2\right) + \bigO_{t+1}
\]
and $\uppsi(f)-(x^3+y^3 + f_4+\cdots + f_{t-1} )\in\n^t$.
\end{lemma}
\begin{proof}
Consider the unitriangular automorphism $\uppsi$ which sends $x\mapsto x-\tfrac{1}3h_1$, $y\mapsto y-\tfrac{1}{3}h_2$. The result follows since 
\[
\uppsi(x^3+y^3) \sim x^3-h_1x^2+\tfrac{1}3h_1^2x-\tfrac{1}{27}h_1^3+y^3-h_2y^2+\tfrac{1}3h_2^2y-\tfrac{1}{27}h_2^3,
\]
and $\uppsi(m) \equiv m \mod\n^{t+1}$ whenever $\deg(m) \ge 4$.
\end{proof}

With the preparatory lemma in place, the strategy is to first find a standard
power series form of each potential, and then distill that down to a polynomial normal form.

\begin{prop}\label{prop!x3y3prep}
Suppose that $f=f_3+\bigO_{4}$ where $f_3^{\ab}$ has three distinct factors.  Then 
$f\cong x^3+y^3+p(xy)$
for some power series $p(z)\in\mathbb{C}\llsq z\rrsq$ with $\ord(p)\ge2$.
\end{prop}
Recall the Conventions~\ref{conventions} on denoting graded pieces of sequence
elements: we denote sequence elements $\sff_n\in\ring$ in Greek font,
and we write $(\sff_n)_t$ for its
degree~$t$ piece, and $(\sff_n)_{<t}$ and $(\sff_n)_{>t}$
for sub- and super-degree~$t$ portions respectively. 
\begin{proof}
We construct a sequence of power series $\mathsf{f}_1,\mathsf{f}_2,\dots$
and unitriangular automorphisms $\phi^1, \phi^2,\dots$ inductively, with
each $\sff_t$ having the form of the target power series $x^3+y^3+p(xy)$ in small degree. Summary~\ref{isomain}\eqref{main chase cor} then constructs $\mathsf{f}=\lim \mathsf{f}_i$ of the required form with $\Jac(f)\cong\Jac(\mathsf{f})$.

By \ref{linear change}, $f\cong g$ where $g_3=x^3+y^3$.
After grouping together terms containing $x^2$ or $y^2$ and cyclically permuting,
we may write
\[
g \sim x^3 + y^3 + \sfh_2\cdot x^2 + \sfh_2'\cdot y^2 + \upmu_4 (xy)^2 + \bigO_{5}
\]
for $\sfh_2, \sfh_2'\in\ringtwo_2$ and $\upmu_4\in\C$.

Hence we begin the induction by setting
\[
\sff_1=x^3+y^3 + \sfh_2\cdot x^2 + \sfh_2'\cdot y^2 + \upmu_4 (xy)^2 + g_{\geq 5}
\]
and note that $\sff_1\sim g \cong f$.
Thus $\sff_1$ is in the desired form in degrees~$\le3$ and has its
degree~4 piece prepared in standard form for further analysis.

For the inductive step more generally, we may suppose that $\sff_{t}\in\ring$
has been constructed of the form
\[
\sff_{t} = \big(x^3+y^3 + \sfp_{t+2}(xy)\big) +
\big(  \sfh_{t+1}\cdot x^2 + \sfh_{t+1}'\cdot y^2 + \upmu_{t+3} (xy)^{\lfloor (t+3)/2\rfloor} \big) + 
\bigO_{t+4}
\]
with $\sfp_3=0$ and by convention $\upmu_{t+3}=0$ for even $t$, where
\begin{enumerate}
\item
$(\sff_{t})_{\leq t+2} = x^3+y^3 + \sfp_{t+2}(xy)$, for some polynomial $\sfp_{t+2}\in\C[z]_{\ge2}$ of degree $\leq (t+2)/2$, where the polynomials $\sfp_3,\dots, \sfp_{t+2}$ satisfy $\sfp_{i+1} = \sfp_i$ for even~$i$ and
$\sfp_{i+1} - \sfp_i=\upmu_{i+1}z^{(i+1)/2}$
for odd~$i$, and
\item
$(\sff_{t})_{t+3} = \sfh_{t+1}\cdot x^2 + \sfh_{t+1}'\cdot y^2 + \upmu_{t+3} (xy)^{\lfloor (t+3)/2\rfloor}$ for some homogeneous forms $\sfh_{t+1}$, $\sfh_{t+1}'$  of degree~$t+1$.
\end{enumerate}

Applying \ref{lemB3} with $h_1=\sfh_{t+1}$ and $h_2=\sfh_{t+1}'$, there exists a unitriangular $\phi^{t}$ of depth $\geq t$ such that
\[
\phi^{t}(\sff_{t})\sim \big(x^3+y^3 + \sfp_{t+2}(xy)\big) + \upmu_{t+3}(xy)^{\lfloor (t+3)/2\rfloor} + \bigO_{t+4}.
\]
In degree $t+4$, again grouping together the terms containing $x^2$ or $y^2$ and cyclically permuting, we may write
\[
\phi^{t}(\sff_{t})_{t+4}\sim \sfh_{t+2}\cdot x^2 + \sfh_{t+2}'\cdot y^2+\upmu_{t+4} (xy)^{\lfloor (t+4)/2\rfloor}
\]
for homogeneous forms $\sfh_{t+2}$, $\sfh_{t+2}'$ of degree $t+2$ and some $\upmu_{t+4}\in\C$, where again $\upmu_{t+4}=0$ for odd $t$.  Thus, after setting $\sfp_{t+3}(xy)=\sfp_{t+2}(xy)+\upmu_{t+3}(xy)^{\lfloor (t+3)/2\rfloor}$, define
\[
\sff_{t+1}=x^3+y^3 + \sfp_{t+3}(xy) + \big( \sfh_{t+2}\cdot x^2 + \sfh_{t+2}'\cdot y^2+\upmu_{t+4} (xy)^{\lfloor(t+4)/2\rfloor} \big) + \phi^{t}(\sff_{t})_{\geq t+5}.
\]
Note that $\phi^{t}(\sff_{t})\sim \sff_{t+1}$, and $\phi^{t}(\sff_{t})- \sff_{t+1}\in\hatM^{t+3}\subset\hatM^{t+1}$ using the last statement of \ref{lemB3}.

Thus we have constructed a sequence of power series $\sff_1,\sff_2,\dots$ and unitriangular
automorphisms $\phi^1,\phi^2,\dots$ to which \ref{isomain}\eqref{main chase cor} applies.
For $s\geq 3$ either $s$ is even, in which case $\sfp_{s+1}=\sfp_s$, or $s$ is odd, in which case $\sfp_{s+1}=\sfp_{s}+\upmu_{s+1}z^{(s+1)/2}$, thus it is clear that
$p\colonequals \lim \sfp_s=\sum_{s=2}^\infty\upmu_{2s}z^{s}$.  Further,  $\sff=\lim \sff_i=x^3+y^3+p(xy)$, since the difference
$(\sff_i-(x^3+y^3+p(xy)))_{i\geq 1}$ converges to zero.

It follows from \ref{isomain}\eqref{main chase cor} that $\Jac(f)\cong\Jac(x^3+y^3+p(xy))$, as required.
\end{proof}

The next step is to replace the power series $p(xy)$ by its leading term, without changing the Jacobi algebra. 

\begin{lemma}\label{lem!yxx}
If $f=x^3+y^3+p(xy)\in\ring$ for some $0\neq p(z)\in\C\llsq z\rrsq$ for which $s=\ord(p)\ge2$, then the following statements hold.
\begin{enumerate}
\item\label{yxxi}
$yx^2, xy^2\in \lcl \dcyc_x f, \dcyc_y f \rcl$.
\item\label{yxxiB}
$x^2y, y^2x\in \lcl \dcyc_x f, \dcyc_y f \rcl$.
\item\label{yxxii}
$(xy)^{s}x, (yx)^{s}y\in \lcl \dcyc_x f, \dcyc_y f \rcl$.
\end{enumerate}
\end{lemma}

\begin{proof}
(1) Write $J_f = (\dcyc_x f, \dcyc_y f )$, so that $\lcl \dcyc_x f, \dcyc_y f \rcl$ is the closure of $J_f$. Differentiating and pulling out the lowest terms, write
\begin{equation}\label{eq:qform}
\dcyc_xf = 3x^2 + y(xy)^{s-1}q(xy)
\qquad\textrm{and}\qquad
\dcyc_yf = 3y^2 + q(xy)(xy)^{s-1}x
\end{equation}
for some $q=\la_0 + \la_1z + \la_2z^2+\cdots\in\C\llsq z\rrsq$ with $\la_0\ne0$.
Writing $A\equiv B$ for $A-B \in J_f$, then in particular
\begin{equation}
x^2 \equiv -\tfrac{1}3 y(xy)^{s-1}q(xy)
\qquad\textrm{and}\qquad
y^2 \equiv -\tfrac{1}3 q(xy)(xy)^{s-1}x .\label{subs for x and y}
\end{equation}
Substituting for $x^2$ or $y^2$ at each step, we see that
\begin{align*}
y \cdot x^2 &\equiv \tfrac{-1}3 y^2 \cdot (xy)^{s-1}q(xy) &\in\n^{2s\phantom{+2}}\\
&\equiv
\tfrac{+1}{3^2} q(xy)(xy)^{s-1}x \cdot (xy)^{s-1}q(xy)& \\
&= \tfrac{+1}{3^2} q(xy)(xy)^{s-1} \cdot x^2 \cdot y (xy)^{s-2}q(xy) &\in\n^{4s-3}\\
&\equiv
\tfrac{-1}{3^3} q(xy)(xy)^{s-1} \cdot y(xy)^{s-1}q(xy) \cdot y (xy)^{s-2}q(xy)& \\
&=  
\tfrac{-1}{3^3} q(xy)(xy)^{s-2}x \cdot y^2 \cdot (xy)^{s-1}q(xy) y (xy)^{s-2}q(xy) &\in\n^{6s-6} \\
&=\hdots
\end{align*}
At each substitution, the resulting power series has order $2s-3\ge1$ higher than the previous one. It follows from the above that for all $t\geq 2s$ there exists $n_t\in\n^t$ such that  $yx^2-n_t\in J_f$. Hence  $yx^2\in\bigcap_{t\ge 2s}(J_f+\n^t)$, which is precisely the closure  $\lcl \dcyc_x f, \dcyc_y f \rcl$.  By symmetry in $x$ and $y$, the analogous statement $xy^2\in \lcl \dcyc_x f, \dcyc_y f \rcl$ also follows.\\
(2) This follows in an analogous way: start by writing
\[
\dcyc_xf = 3x^2 + r(yx)(yx)^{s-1}y
\qquad\textrm{and}\qquad
\dcyc_yf = 3y^2 + x(yx)^{s-1}r(yx)
\]
then consider $x^2\cdot y\equiv\tfrac{-1}{3}r(yx)(yx)^{s-1}y^2$, etc.\\
(3) Now write $A\equiv B$ for $A-B \in\lcl \dcyc_x f, \dcyc_y f \rcl$. Separating off the lowest term of $q(xy)$ in~\eqref{eq:qform}, we may write
\[
\dcyc_yf = 3y^2 + \la_0(xy)^{s-1}x + (q(xy)-\la_0)(xy)^{s-1}x
\]
and so  $\la_0(xy)^{s-1}x \equiv -3y^2 - (q(xy)-\la_0)(xy)^{s-1}x$ where $q(xy)-\la_0\in\n^2$.
Then
\begin{align*}
\la_0(xy)^{s}x & = \left(\la_0(xy)^{s-1}x\right)yx  \\
&\equiv \left( - 3y^2 - (q(xy)-\la_0)(xy)^{s-1}x\right)(yx) \\
&= - 3y\cdot y^2 \cdot x - (q(xy)-\la_0)(xy)^{s}x \\
&\equiv 
y\cdot q(xy)(xy)^{s-1}x \cdot x - (q(xy)-\la_0)(xy)^{s}x\tag{by \eqref{subs for x and y}} \\
&\equiv -(q(xy)-\la_0)(xy)^{s}x\tag{$yx^2\equiv 0$ by \eqref{yxxi}}
\end{align*}
The $\la_0(xy)^{s}x$ on each side cancel, showing that $q(xy)(xy)^{s}x\in \lcl \dcyc_x f, \dcyc_y f \rcl$.  Since $q(xy)$ is a unit, it follows that  $(xy)^{s}x\in \lcl \dcyc_x f, \dcyc_y f \rcl$.  Again, appealing to symmetry in $x$ and $y$ proves the final statement.
\end{proof}

\begin{prop}\label{NC3rootsnormal}
Suppose that $f=x^3+y^3+p(xy)$ where $p(z)\in\C\llsq z\rrsq$ with $s=\ord(p)\ge2$.  
Then
\[
f\cong
\begin{cases}
x^3+y^3&\text{when $p=0$}\\
x^3+y^3+(xy)^s&\text{when $p\ne0$}.\\
\end{cases}
\]
Furthermore
\begin{enumerate}
\item\label{x3y3i}
$\JRdim\Jac(f)\le1$, with equality if and only if $p=0$.
\item\label{x3y3ii}
If $p\ne0$, then $\dim_{\mathbb{C}}\Jac(f)=4s$, and $\dim_{\mathbb{C}}\Jac(f)^{\ab}=4$.
\end{enumerate}
Therefore the expressions $x^3+y^3+(xy)^s$ with $s\in\mathbb{Z}_{\geq 2}\cup\{\infty\}$ form a set of normal forms.
\end{prop}

\begin{proof}
For the first statement, if $p=0$ we are done, so suppose $p\ne0$. Continuing the notation in the proof of \ref{lem!yxx} above, after differentiating and pulling out the lowest terms, we may write
\begin{equation}
\dcyc_xf = 3x^2 + y(xy)^{s-1}q(xy)
\qquad\textrm{and}\qquad
\dcyc_yf = 3y^2 + q(xy)(xy)^{s-1}x\label{eqn:x2y2}
\end{equation}
for some $q=\la_0 + \la_1z + \la_2z^2+\cdots\in\C\llsq z\rrsq$ with $\la_0\ne0$. Set $g=x^3+y^3+(\la_0/s)(xy)^s$.  Now \ref{lem!yxx} applies equally well to both $f$ and $g$, hence both $(xy)^{s}x$ and $(yx)^sy$ belong to \emph{both} the Jacobi ideals associated to $f$ and $g$.  Consequently
\begin{align*}
\lcl \updelta_xf,\updelta_yf\rcl &= \lcl \updelta_xf,\updelta_yf, (xy)^{s}x, (yx)^sy\rcl\\
&= \lcl \updelta_xg,\updelta_yg, (xy)^{s}x, (yx)^sy\rcl\tag{cancel higher terms from $\updelta_xf$ and $\updelta_yf$}\\
&= \lcl \updelta_xg,\updelta_yg \rcl.
\end{align*}
It follows that  $f\simeq g$. The coordinate change $x\mapsto ax$, $y\mapsto ay$ for $a=\sqrt[2s-3]{s/\la_0}$ then normalises the constant factor $\la_0/s\ne0$, as required.\\
\eqref{x3y3i} Consider the case $p=0$. As in \ref{radical remark}, if $\mathfrak{J}$ is the Jacobson radical of $\Jac(f)$, then
\[
\frac{\mathfrak{J}^d}{\mathfrak{J}^{d+1}}
=\frac{\n^d+\lcl x^2,y^2\rcl}{\n^{d+1}+\lcl x^2,y^2\rcl}.
\] 
This is always a two-dimensional vector space, since if $d$ is even it has basis $(xy)^{d/2}$ and $(yx)^{d/2}$, whilst if $d$ is odd it has basis $x(yx)^{(d-1)/2}$ and $y(xy)^{(d-1)/2}$.  It follows that $\Jac(f)$ is an infinite-dimensional $\C$-algebra, with $\JRdim\Jac(f)=1$.

When $p\ne0$, \ref{lem!yxx} shows at once that $\Jac(f)$ is finite dimensional:
indeed any monomial of degree $t\ge2s+1$ either contains one of the monomials listed in~\ref{lem!yxx}(\ref{yxxi}--\ref{yxxii}) or is $x^{t}$ or $y^{t}$. But by~\eqref{eqn:x2y2}  $x^t=x^{t-4}\cdot x^2\cdot x^2$ and $y^t=y^2\cdot y^2\cdot y^{t-4}$
is equivalent, modulo $\lcl \updelta_xf,\updelta_yf\rcl$, to a monomial that contains one of those listed, and thus is equivalent to zero. Consequently, the entire graded piece of degree~$t$ is zero, and so $\Jac(f)$ is finite dimensional.

\noindent\eqref{x3y3ii} We compute a standard basis of $\lcl \updelta_xf,\updelta_yf\rcl$ with respect to a local graded monomial order, where we refer the reader to \ref{rem:gb} below for references to the formal theory and its properties in this case. In particular, leading terms have lowest degree, and lexicographical order selects the leading term when there is more than one of lowest degree.

The proof of~\eqref{x3y3i} above introduces a simplifying factor:
since all monomials of degree $t=2s+1$ lie in the closed Jacobian ideal, it is sufficient
to work in the quotient
\[
\Jac(f)\cong\frac{\mathbb{C}\langle x,y\rangle}{(\updelta_xf,\updelta_yf,\scrM_t)}
\]
where we write $\scrM_t$ for the set of all noncommutative monomials of degree $t$,
and closure is no longer an issue.

Following \cite[3.6]{GH98},
we compile a standard basis $\{g_1,g_2,\dots\}$ of $(\updelta_xf,\updelta_yf,\scrM^t)$
starting with the normalised derivatives
\[
g_1 = x^2 + \tfrac{s}3y(xy)^{s-1} = \tfrac{1}3\dcyc_xf
\quad\text{and}\quad
g_2 = y^2 + \tfrac{s}3x(yx)^{s-1} = \tfrac{1}3\dcyc_yf
\]
which have leading terms $x^2$ and $y^2$ respectively, since $s\ge2$.
The computation proceeds by resolving non-trivial overlaps among leading terms.
The leading term of $g_1$ overlaps non-trivially with itself to produce
(after scaling by $\tfrac{3}s$ to normalise coefficients)
\[
g_3 \colonequals \tfrac{3}s(xg_1-g_1x) = (xy)^s - (yx)^s.
\]
This does not reduce further modulo existing leading terms.
The analogous overlap $yg_2-g_2y$ gives the same $g_3$,
and all other overlaps have order $\ge t$, so reduce to zero modulo~$\scrM_t$.
Hence the standard basis is $\{g_1,g_2,g_3\}$.
Note that we can then dispense with $\scrM_t$, since those
monomials all reduce to zero under $g_1,g_2,g_3$ -- that is exactly what \ref{lem!yxx} demonstrates -- and so they do not appear in the standard basis.

As in the commutative case, the set of monomials not divisible by the leading term of any of $g_1$, $g_2$, $g_3$ descends to give a monomial $\C$-vector space basis for the quotient, by e.g.\ \cite[3.5 and 3.1--2]{GH98}.
Thus, working in increasing degree, $1,x,y$ are in the basis.  Then,
in each pair of degrees $2e,2e+1$ for $1\le e\le s-1$ the basis consists of the four monomials
\[
(xy)^{e},(yx)^{e},\ (xy)^{e}x, (yx)^{e}y,
\]
and finally $(xy)^{s}\equiv (yx)^{s}$ in degree~$2s$.
Summing up, this basis has size $4s$, as claimed.

In the abelianisation, if $s>2$ we may rewrite the derivatives as $x^2(\mathrm{unit})$ and $y^2(\mathrm{unit})$. Hence  $\Jac(f)^{\ab}\cong\mathbb{C}\llsq x,y\rrsq/(x^2,y^2)$, which is four dimensional. When $s=2$, it is also easy to verify that $\Jac(f)^{\ab}\cong\mathbb{C}\llsq x,y\rrsq/(x^2,y^2)$, so in all cases the dimension is four.
\end{proof}

\begin{remark}\label{rem:gb}
The theory of standard bases, also known as local Gr\"obner bases, of ideals and their closures in noncommutative power series rings is less well documented in the literature than either global polynomial Gr\"obner bases (commutative or not), or Mora's tangent cone algorithm for commutative power series rings; see for example \cite{mora94}
or \cite[III.1]{hironaka}.
Nevertheless, the theory exists following analogous ideas, and has analogous
conclusions.
The essential reference is \cite{GH98}, where \S3 establishes the existence and properties of
standard bases, whilst \S4--5 provides the tools needed to calculate.
Standard bases may be infinite in general, but within the context
of $\JRdim\le1$ examples in this paper, this issue does not arise.
\end{remark}

In order to state a unified theorem with the $xy^2$ case
in \S\ref{sec: overview D} below, it is convenient to mildly change basis.
This is rather cheap, largely because there are no moduli.

\begin{cor}\label{NC3rootsnormal2}
Suppose that $f\in\ringtwo_{\geq 3}$ where $f_3^{\ab}$ has three roots.  Then either
\[
f\cong
\begin{cases}
xy^2+x^3&\\
xy^2+x^3+x^{2n}&\mbox{for some }n\geq 2.\\
\end{cases}
\]
Furthermore, the above are normal forms.
\end{cor}
\begin{proof}
Each of the forms $f$ listed satisfies the condition that $f_3^{\ab}$ has three roots. Thus there exists some $g$ from the list in \ref{NC3rootsnormal} with $g\cong f$. Furthermore, $\dim_\mathbb{C}\Jac(xy^2+x^3)=\infty$, whereas $\dim_\mathbb{C}\Jac(xy^2+x^3+x^{2n})=4n$ (see e.g.\ \cite[\S5]{Okke} and \cite[\S5]{Kawamata}, or \ref{ex: d dimensions}), and so all options are uniquely covered.  Since the $g$ listed in \ref{NC3rootsnormal} are normal forms, it follows that the $f$ listed here are normal forms.
\end{proof}

\subsection{Isomorphisms on the Quantum Plane}
The following, which may be of independent interest, is one of the key reduction steps that will be used in \S\ref{sec!2roots}.
\begin{lemma}\label{mod xy+yx lemma}
For any units $v, w\in\C\llsq x^2\rrsq$, the unitriangular automorphism $\phi$ of $\C\llangle x,y\rrangle$ sending $x\mapsto xv$, $y\mapsto yw$ descends to a topological isomorphism
\[
\frac{\C\llangle x,y\rrangle}{\lcl xy+yx\rcl} \buildrel{\cong}\over{\longrightarrow} \frac{\C\llangle x,y\rrangle}{\lcl xy+yx \rcl}.
\]
\end{lemma}
\begin{proof}
The inverse of $\phi$, as an automorphism of $\C\llangle x,y\rrangle$, is clearly given by the unitriangular automorphism $\uppsi\colon x\mapsto xv^{-1}$, $y\mapsto yw^{-1}$.  Set $I=\lcl xy+yx\rcl$, then since $\phi$ is a topological isomorphism by \ref{max to max}, we just need to prove that $\phi(I)=I$.

Now $x$ commutes with $v$ and $w$, being power series in $x^2$, and also $vw=wv$. But, modulo $I=\lcl xy+yx\rcl$, $y$ commutes with $x^2$, thus since the ideal is closed $y$ commutes with both $v$ and $w$. It follows that
\begin{equation}
\phi(xy+yx)= xvyw+ywxv\equiv (xy+yx)vw\equiv 0 \quad \textnormal{mod } I.\label{magic inclusion}
\end{equation}
Since $\phi$ is a continuous isomorphism, and $I$ is the smallest closed ideal containing $xy+yx$,  $\phi(I)$ is the smallest closed ideal containing $\phi(xy+yx)$.  But by \eqref{magic inclusion} $\phi(xy+yx)$ also belongs to the closed ideal $I$, so by minimality $\phi(I)\subseteq I$.

Since $v^{-1}$ and $w^{-1}$ are also units in $\C\llsq x^2\rrsq$, exactly the same logic applied to $\uppsi$ shows that $\uppsi(I)\subseteq I$.  Applying $\phi$ to this inclusion, we see that $I=\phi\uppsi(I)\subseteq \phi(I)$. Combining inclusions gives $\phi(I)=I$. 
\end{proof}

%-----------------------------------------------------------------------------------------
\subsection{Abelianized Cubic with Two Factors}\label{sec!2roots}
This subsection considers the case of \ref{linear change}
where $f_3^{\ab}$ has two distinct factors, that is $f\cong x^2y+\bigO_4$, and in \ref{cor:D two roots main} provides normal forms. 
% In this subsection we consider the 2-roots case of Lemma~\ref{linear change}.
This is substantially harder than in \S\ref{sec!3roots}.
\begin{lemma}\label{lemB}
Fix $t\ge 4$, and let $f=xy^2 + f_4+\cdots + f_t + \bigO_{t+1}$.  For any $h\in\mathbb{C}\llangle x,y\rrangle$ with $\ord(h) = t-2$, the unitriangular automorphism $x\mapsto x-h$, $y\mapsto y$ sends
\[
f\mapsto xy^2 + f_4+\cdots + f_{t-1} + \left(f_t-hy^2\right) + \bigO_{t+1}.
\]
\end{lemma}
\begin{proof}
Write $\uppsi$ for the stated automorphism.  The result follows since $\uppsi(xy^2) = xy^2 - hy^2$ and $\uppsi(m) \equiv m \mod\n^{t+1}$ whenever $\deg(m) \ge 4$. 
\end{proof}

The next lemma is much less elementary.
\begin{lemma}\label{lemC}
Fix $t\ge 4$, and let $f=xy^2 + f_4+\cdots + f_t + \bigO_{t+1}$, where  furthermore
\begin{equation}\label{eq!fdlemC}
f_t=\sfh_{t-2}\cdot y^2+\sum\la_a x^{a_1}y\hdots x^{a_r}y+\upalpha x^t
\end{equation}
for some homogeneous form $\sfh_{t-2}$  of degree~$t-2$, 
each $a_i\ge1$, $r\ge1$ and $r+\sum a_i=t$, $\upalpha\in\C$ and each $\la_a=\la_{a_1\cdots a_r}\in\C$.
Then there exists a unitriangular automorphism $\phi$ of depth $\geq t-3$ such that
\[
\phi(f)= xy^2 + f_4+\cdots + f_{t-1} + (g_t+\upalpha x^t) + \bigO_{t+1}.
\]
where $g_t\in\ringtwo_t$ satisfies $g_t\sim 0$.
\end{lemma}

\begin{example}
It is worth considering an example to make the notation of both the statement and proof more transparent.  Consider
\[
f = xy^2 + \Big(\la_{51}x^5yxy + \la_{42}x^4yx^2y + \la_{33}x^3yx^3y\Big) + f_{\ge9}
\]
which has $f_8$ of the form~\eqref{eq!fdlemC}.
Applying $\phi_1\colon x\mapsto x$ and $y\mapsto y - \la_{33}x^3yx^2$, where we cancelled $xy$ from the right of the target $\la_{33}$ term to obtain the subtracted term, gives
\[
\phi_1(f) = xy^2 + \la_{51}x^5yxy + (\la_{42}-\la_{33})x^4yx^2y + g_1+ \bigO_9
\]
where $g_1=\la_{33}(x^3yx^3y - xyx^3yx^2)\sim0$. 
Ignoring $g_1$, the summation in degree~8 symbolically now has only two terms, which is
progress. 

An analogous automorphism $\phi_2$ sending $x\mapsto x$ and $y\mapsto - (\la_{42}-\la_{33})x^4yx$, where we cancelled $xy$ from the right of the next target term, gives
\[
\phi_2\phi_1(f) = xy^2 + (\la_{51} - \la_{42} + \la_{33})x^5yxy + g_2+ \bigO_9
\]
for some $g_2\sim0$, and again the number of terms in degree~8 (outside $g_2$)
has not increased.
Repeating again with an analogous automorphism $\phi_3$ gives 
\[
\phi_3\phi_2\phi_1(f) = xy^2 - (\la_{51} - \la_{42} + \la_{33})x^6y^2 + g_3+ \bigO_9
\]
with $g_3\sim0$.
We are now in a position to apply \ref{lemB} to leave only $g_3$ in degree~8.

The proof below confirms that this inductive idea works more generally.
\end{example}

\begin{proof}
If the middle sum in the expression for $f_t$ is zero, we are done by \ref{lemB} (with $g_t=0$), so we may assume that the sum is nonzero.  

Suppose that the middle sum contains a term $\sft_1 = \la_b x^{b_1}y\cdots x^{b_r}y$ with $r>1$.
In this case, consider the unitriangular automorphism $\phi$ defined by $x\mapsto x$, $y\mapsto y-\la_b x^{b_1}y\cdots yx^{b_r-1}$, where we have simply cancelled $xy$ from the right-hand side of the target term~$\sft_1$. As in \ref{lemB}, $\phi(m) \equiv m \mod\n^{t+1}$ whenever $\deg(m) \ge 4$, so any change in degree~$\le t$ comes from $\phi(xy^2)$, and thus
\begin{align}
\nonumber
\phi(f) = xy^2 &+ f_4 +\cdots+f_{t-1} \\
 &+(f_t- \la_b xyx^{b_1}y\cdots yx^{b_r-1} -\la_b x^{b_1+1}y\cdots yx^{b_r-1}y ) + \bigO_{t+1} \label{2 roots a}
\end{align}
Writing
$\sfg_1=\sft_1-\la_b xyx^{b_1}y\cdots yx^{b_r-1}\sim 0$,
then  
the degree~$t$ term of~\eqref{2 roots a} equals 
\[
\sfg_1+h\cdot y^2+\sum\la_a x^{a_1}y\hdots x^{a_r}y+\upalpha x^t
\] 
where, under the summand, the target term $\sft_1$ 
has been
replaced by a term of the form $\sft_2 = -\la_b x^{b_1+1}y\cdots yx^{b_r-1}y$, so the sum has the same
number of terms or fewer (depending on whether $\sft_2$ cancels with existing terms or not).

If $b_r=1$, then the new term $\sft_2$ 
equals $hy^2$ for $h=-\la_b x^{b_1+1}y\cdots yx^{b_{r-1}}$, and 
we so may apply \ref{lemB} to \eqref{2 roots a} to find $\uppsi$ such that
\[
\uppsi\phi(f) =  xy^2 + f_4 +\cdots+f_{t-1}+(f_t- \la_b xyx^{b_1}y\cdots yx^{b_r-1}  ) + \bigO_{t+1} 
\]
where the degree $t$ term is equal to 
\begin{align*}
f_t- \la_b xyx^{b_1}y\cdots yx^{b_r-1}&=\sfg_1+f_t-\la_b x^{b_1}y\cdots x^{b_r}y\\
&=\sfg_1+h\cdot y^2+\sum_{a\neq b}\la_a x^{a_1}y\hdots x^{a_r}y+\upalpha x^t
\end{align*}
and the number of terms under the summand is now strictly reduced.  

Otherwise, $b_r>1$. Set $\phi_1=\phi$, and repeating the original construction of a unitriangular automorphism by cancelling $xy$ from the right, we can construct $\phi_2$ such that
\begin{align}
\nonumber
\phi_2\phi_1(f)= xy^2 &+ f_4 +\cdots+f_{t-1} \\
 &+\big(\sfg_2+h\cdot y^2+\sum\la_a x^{a_1}y\hdots x^{a_r}y+\upalpha x^t\big)+\bigO_{t+1}
 \label{eq!phi2phi1}
\end{align}
where $\sfg_2\sim0$ is the sum of $\sfg_1$ and another binomial $\sim0$, and in the sum we have replaced the term $-\la_b x^{b_1+1}y\cdots yx^{b_r-1}y$ by $\la_b x^{b_1+2}y\cdots yx^{b_r-2}y$.  Repeating this, we find unitriangular automorphisms $\phi_1,\dots,\phi_{b_r-1}$ so that $\phi_{b_r-1}\cdots\phi_2\phi_1(f)$ has the form of~\eqref{eq!phi2phi1}, and
the sum has the same number of terms or fewer, but in which the target
monomial we are focussing on 
has become $x^{b_1+b_r-1}yx^{b_2}\cdots yxy$. A further repetition with a unitriangular
automorphism $\phi_{b_r}$ replaces that term by one that contains $y^2$, and once
again we may apply \ref{lemB} to find a unitriangular automorphism $\uppsi$ that moves this term into higher degree.
Thus after applying the single unitriangular automorphism $\uppsi\phi_{b_r}\cdots\phi_1$
to $f$, the number of terms in the summation when parsed in the form~\eqref{eq!fdlemC} has
strictly reduced.

We repeat this process inductively, and it will terminate when there are no terms
under the summation sign of the form $\la_a x^{a_1}y\hdots x^{a_r}y$ with $r>1$.
Each step was achieved by a single unitriangular automorphism (itself built as a composition
of unitriangular automorphisms), and composing each of these gives a single
unitriangular automorphism $\upvartheta$ such that
\[
\upvartheta(f)=xy^2 + f_4 +\cdots+f_{t-1}+(\sfg+h\cdot y^2+\la x^{t-1}y+\upalpha x^t ) + \bigO_{t+1} 
\]
for some $\sfg$ with $\sfg\sim 0$.

To conclude, the unitriangular automorphism $\phi$ defined by $x\mapsto x-h$, $y\mapsto y-\frac{\la}{2}x^{a_1-1}=y-\frac{\la}{2}x^{t-2}$ has depth $t-3$, so again $\phi(m) \equiv m \mod\n^{t+1}$ whenever $\deg(m) \ge 4$ and  thus
\[
\phi\upvartheta(f) = xy^2 + f_4 +\cdots+f_{t-1}+ (\sfg+\underbrace{\la x^{t-1}y-\tfrac{\la}{2}xyx^{t-2}-\tfrac{\la}{2}x^{t-1}y}_{\sfh}+\upalpha x^t)+ \bigO_{t+1}. 
\]
Set $g_t=\sfg+\sfh$, then since both $\sfg\sim 0$ and $\sfh\sim 0$, we are done.
\end{proof}

From here, the strategy of \S\ref{sec!3roots} remains: first find a standard power series form of each potential,  then simplify into polynomial normal form.

\begin{prop}\label{prop!xy2prep}
Suppose that $f=f_3+\bigO_{4}$ where $f_3^{\ab}$ has two distinct linear factors.  Then 
$f\cong xy^2+q(x)$ for some power series $q(x)\in\mathbb{C}\llsq x\rrsq$
with $\ord(q)\ge4$.
\end{prop}

Recall the Conventions~\ref{conventions}, used in \ref{prop!x3y3prep}, on graded pieces of sequence elements: namely sequence elements $\sff_n\in\ring$ are in Greek font,
whilst $(\sff_n)_t$ is the degree~$t$ piece of $\sff_n$.

\begin{proof}
We construct a sequence of power series $\mathsf{f}_1,\mathsf{f}_2,\dots$ and unitriangular
automorphisms $\phi^1,\phi^2,\dots$ inductively, with
each $\sff_t$ having the form of the target power series $xy^2+q(x)$ in low degree.
Summary~\ref{isomain}\eqref{main chase cor} will then construct $\mathsf{f}=\lim \mathsf{f}_i$ of the required form with $\Jac(f)\cong\Jac(\mathsf{f})$.

By \ref{linear change}, $f\cong g$ where $g_3=xy^2$.  After grouping together the terms containing~$y^2$, then the terms that contain~$y$ but not~$y^2$, and cyclically permuting, we may write
\[
g_4\sim \sfh_2\cdot y^2+\sum\la_a x^{a_1}y\hdots x^{a_r}y+\upmu_4 x^4
\]
for $\sfh_2\in\freering_2$, $r\ge1$ and each $a_i\ge1$ and $\upmu_4\in\C$
and where we use the abbreviated notation $\la_a := \la_{a_1\cdots a_r}\in\C$. It is convenient to write the sum as $\sum\la_a x^{a_1}y\hdots x^{a_r}y$
by analogy with the general case, noting that here it is nothing more than
$\la_{11}xyxy+\la_3x^3y$.

Hence we begin the induction by setting
\[
\sff_1=xy^2 + (\sfh_2\cdot y^2+\sum\la x^{a_1}y\hdots x^{a_r}y+\upmu_4 x^4) + g_{\ge5}
\]
and note that $\sff_1\sim g\cong f$.  Thus $\sff_1$ is in the desired form in degrees~$\le3$ and has its degree~$4$ piece prepared in  standard form for further analysis.

For the inductive step more generally, we may suppose that $\sff_{t}\in\C\llangle x,y\rrangle$ has been constructed of the form
\begin{align*}
\sff_{t} = \big( xy^2 + \sfq_{t+2}(x) \big) +
\big( {\sfh}_{t+1}\cdot y^2+\sum\la_a x^{a_1}y\hdots x^{a_r}y+\upmu_{t+3} x^{t+3} \big) +\bigO_{t+4}
 \end{align*}
with $\sfq_3=0$ and
\begin{enumerate}
\item
$(\sff_{t})_{\leq t+2} = xy^2 + \sfq_{t+2}(x)$, for some polynomial $\sfq_{t+2}\in\C[x]_{\ge4}$ of degree $\leq t+2$, where the polynomials $\sfq_3,\dots, \sfq_{t+2}$ satisfy $\sfq_{i+1} - \sfq_i=\upmu_{i+1}x^{i+1}$ for $\upmu_{i+1}\in\C$, and
\item
$(\sff_{t})_{t+3} = \sfh_{t+1}\cdot y^2+\sum\la_a x^{a_1}y\hdots x^{a_r}y+\upmu_{t+3} x^{t+3}$ for some homogeneous form $\sfh_{t+1}$  of degree~$t+1$, 
each $a_i\ge1$, $r\ge1$ and $r+\sum a_i=t+3$, and $\upmu_{t+3}\in\C$.
\end{enumerate}

By \ref{lemC} there exists a unitriangular $\phi^{t}$ of depth $t$ such that
\[
\phi^{t}(\sff_{t})= xy^2 + q_{t+2}(x) + (k_{t+3}+\upmu_{t+3}x^{t+3}) + \bigO_{t+4}.
\]
where $k_{t+3}\sim 0$.  In degree $t+4$, again grouping together the terms containing $y^2$, then the terms that contain $y$ but not $y^2$, and cyclically permuting, we may write
\[
\phi^{t}(\sff_{t})_{t+4}\sim \sfh_{t+2}\cdot y^2+\sum\la_a x^{a_1}y\hdots x^{a_r}y+\upmu_{t+4} x^{t+4}.
\]
Thus after setting $\sfq_{t+3}(x)=\sfq_{t+2}(x)+\upmu_{t+3}x^{t+3}$, define
\[
\sff_{t+1}=xy^2 + \sfq_{t+3}(x) + \big( \sfh_{t+2}\cdot y^2+\sum\la_a x^{a_1}y\hdots x^{a_r}y+\upmu_{t+4} x^{t+4} \big) + \phi^t(\sff_{t})_{\geq t+5}.
\]
Note that $\phi^{t}(\sff_{t})\sim \sff_{t+1}$, and that $\phi^{t}(\sff_{t})- \sff_{t+1}\in\hatM^{t+3}\subset\hatM^{t+1}$.

Thus we have constructed a sequence of power series $\sff_1,\sff_2,\dots$ and unitriangular
automorphisms $\phi^1,\phi^2,\dots$ to which \ref{isomain}\eqref{main chase cor} applies. 
Since at each stage $\sfq_s=\sfq_{s-1}+\upmu_sx^s$, it is clear that
$q\colonequals \lim \sfq_s=\sum_{s=4}^\infty\upmu_sx^s$,
and that $\sff=\lim \sff_i=xy^2+q$, since the difference
$(\sff_i-(xy^2+q))_{i\geq 1}$ converges to zero.
Hence $\Jac(f)\cong\Jac(xy^2+q)$, as required.
\end{proof}

The next step is to reduce the options for $q(x)$, using the following preliminary lemma.

\begin{lemma}\label{lempsx}
Let $u\in\C\llsq x\rrsq$ be an even power series: that is, $u$ is a power series in $x^2$.
\begin{enumerate}
\item\label{psxi}
If $u$ is a unit, then $u^{-1}$ and $\sqrt[n]{u}$ are also even power series
for any $n\ge2$.
\item\label{psxii}
Let $U\in\C\llsq x\rrsq$ be a unit and $n\in\Z$ a nonzero integer.
Then there is a unit $t\in\C\llsq x\rrsq$ with $t^n=U(xt)$.
Furthermore, if $U$ is even then $t$ is even.
\end{enumerate}
\end{lemma}

\begin{proof}
(1) Consider $v\in\C\llsq z\rrsq$ with $u(x) = v(x^2)$. If $u$ is a unit, then $v$ is a unit
and $v^{-1}$ and $\sqrt[n]{v}\in\C\llsq z\rrsq$ for all $n\ge2$.
Then $u^{-1}(x) = v^{-1}(x^2)$ and $\sqrt[n]{u}(x)=\sqrt[n]{v}(x^2)$.

\noindent
(2) Write $U=a_0+a_1x+a_2x^2+\cdots$ with $a_0\not=0$. Consider the case $n>0$.
We show that we may solve inductively for the coefficients $b_d$ of the expansion $t=b_0+b_1x+b_2x^2+\cdots$ in the equation $t^n=U(xt)$.

It is clear that the coefficient of $x^d$ in $t^n$ is a sum of $nb_0^{n-1}b_d$ with terms involving only coefficients $b_i$ with $i<d$. 
On the other side of the equation, the coefficient of $x^d$ in $U(xt)$ is a sum of terms involving $a_i$ and $b_j$ with $i\le d$ and $j<d$.
Putting these together, $b_d$ does not appear in the coefficient of $x^i$ for any $i<d$,
and it appears linearly with nonzero coefficient for the first time in the coefficient of~$x^d$, and so we may solve for it. Working inductively in increasing $d\ge0$,
and taking the limit, determines $t$ as claimed.
For $n<0$, the same argument proves the existence of the unit $t^{-1}$,
which is equivalent.

Suppose that $U$ is even, and let $b_{2n+1}$ be the smallest nonzero odd-degree coefficient of~$t$.
Then the odd-degree term with smallest degree in $t^n$ is
 $nb_0^{n-1}b_{2n+1}x^{2n+1}$, while in $U(xt)$ it is $2a_2b_0b_{2n+1}x^{2n+3}$
which appears in the summand $a_2(xt)^2$, a contradiction.
So $t$ must be even.
\end{proof}

The point is now simple: $y$ is under control, and so there is a relation $xy+yx$ in the Jacobi algebra.  Then \ref{mod xy+yx lemma} yields the following key preparation result.

\begin{prop}\label{prop!polypotential}
If $f=xy^2+p(x)$ where $p(x)\in\C\llsq x\rrsq$ with $\ord(p)\ge4$,  then
\[
f\cong
xy^2+\upalpha\, x^{2n}+\upbeta\, x^{2m+1}
\]
for some $n,m\geq 2$, and some $\upalpha,\upbeta\in\{0,1\}$.
Furthermore, $\upalpha=1$ if and only if $p$ has a nonzero even-degree term,
in which case $2n$ is the least even degree appearing in~$p$,
and similarly the analogous criterion for $\upbeta=1$ and least odd degree term in~$p$.
\end{prop}

\begin{proof}
First note that
\[
\Jac(f) = \C\llangle x,y\rrangle / \lcl xy+yx, y^2 + \dcyc_xp \rcl.
\]
We exhibit an automorphism of $\C\llangle x,y\rrangle$ that takes the two generators
of the Jacobian ideal to $(xy+yx)\text{(unit)}$ and
$\left(y^2 + \upbeta x^{2m} + \upalpha x^{2n-1}\right)\text{(unit)}$, respectively,
where either $\upalpha=0$ or $2n\ge4$ is the least even degree appearing in~$p$,
and either $\upbeta=0$ or $2m+1\ge5$ is the least odd degree appearing in~$p$.
This proves all the claims.

Parsing $\dcyc_xp$ into even and odd terms, write the Jacobi algebra relations as
\[
xy+yx
\quad\text{and}\quad
y^2 + ax^{2N}u + bx^{2M-1}v
\]
where $u,v\in\C\llsq x^2\rrsq$ are each either a unit or zero, $N,M\ge2$,
and $a,b\in\C$ are any nonzero numbers that carry through the calculation
undisturbed; we choose $a=2N+1$ and $b=2M$ at the end.

Suppose in the first place that $u\neq 0$, then fix a square root $s=\sqrt{u}\in\C\llsq x^2\rrsq$ and consider the unitriangular automorphism $\phi$ sending $x\mapsto x$, $y\mapsto ys$. By \ref{mod xy+yx lemma} this induces a topological isomorphism
\[
\bar{\phi}\colon \frac{\C\llangle x,y\rrangle}{\lcl xy+yx\rcl} \xrightarrow{\sim} \frac{\C\llangle x,y\rrangle}{\lcl xy+yx \rcl}.
\]
In the codomain of this map, $y$ commutes with $x^2$ and thus commutes with $s\in\C\llsq x^2\rrsq$. It follows that $\bar{\phi}(y^2)=ysys=y^2s^2=yu$ and thus 
\[
\bar{\phi}(y^2+ \dcyc_xp) = (y^2+ax^{2N})u+bx^{2M-1}v.
\]
By \ref{closed ideals 101}\eqref{closed ideals 101 A}\eqref{closed ideals 101 C} after right multiplying by the unit $u^{-1}$, we obtain an  isomorphism 
\begin{equation}
\frac{\C\llangle x,y\rrangle}{\lcl xy+yx, y^2+ \dcyc_xp\rcl} \xrightarrow{\sim} \frac{\C\llangle x,y\rrangle}{\lcl xy+yx,  y^2+ax^{2N}+bx^{2M-1}\tfrac{v}{u}\rcl}.\label{new unit 1}
\end{equation}
If $v=0$, then \eqref{new unit 1} asserts that $\Jac(f)\cong\Jac(xy^2+x^{2N+1})$, and so we are done. Hence we may assume that  $v\neq 0$.

As $u$ and $v$ are both unit power series in $x^2$, so is $\frac{v}{u}$. 
By \ref{lempsx}\eqref{psxii}, since $2N-2M+1$ is nonzero, we may choose a unit
$t\in \C\llsq x^2\rrsq$ such that $t^{2N}=t^{2M-1}v(xt)/u(xt)$.
Consider the unitriangular automorphism $\uppsi$ sending $x\mapsto xt$, $y\mapsto yt^N$. Again by \ref{mod xy+yx lemma} there is an induced topological isomorphism
\[
\bar{\uppsi}\colon \frac{\C\llangle x,y\rrangle}{\lcl xy+yx\rcl} \xrightarrow{\sim} \frac{\C\llangle x,y\rrangle}{\lcl xy+yx \rcl}.
\]
Clearly $x$ commutes with $t,t^N\in\C\llsq x^2\rrsq$, and further in the codomain of $\bar{\uppsi}$ the element $y$ commutes with $x^2$ and thus commutes with $t,t^N\in\C\llsq x^2\rrsq$. Thus
\begin{align*}
\bar{\uppsi}(y^2+ax^{2N}+bx^{2M-1}\tfrac{v}{u})&= y^2t^{2N}+ax^{2N}t^{2N}+bx^{2M-1}t^{2M-1}u(xt)/v(xt) \\
&=(y^2+ax^{2N}+bx^{2M-1})t^{2N}.
\end{align*}
Again \ref{closed ideals 101}\eqref{closed ideals 101 A}\eqref{closed ideals 101 C} then induces an isomorphism 
\begin{equation}
\frac{\C\llangle x,y\rrangle}{\lcl xy+yx,  y^2+ax^{2N}+bx^{2M-1}\tfrac{v}{u}\rcl} \xrightarrow{\sim} \frac{\C\llangle x,y\rrangle}{\lcl xy+yx,  y^2+ax^{2N}+bx^{2M-1}\rcl}.\label{new unit 2}
\end{equation}
Setting $a=2N+1$ and $b=2M$, the right hand side is $\Jac(xy^2+x^{2N+1}+x^{2M})$.  Hence composing \eqref{new unit 1} with \eqref{new unit 2} gives an isomorphism $\Jac(f)\cong\Jac(xy^2+x^{2N+1}+x^{2M})$.

The case $u=0$ and $v\neq 0$ works in exactly the same way as the case $u\neq 0$ and $v=0$ above, applying an automorphism with $s=\sqrt{v}$, while the case $u=v=0$ is trivial.
\end{proof}

The above is not quite yet in normal form, since some of the polynomial potentials in \ref{prop!polypotential} have isomorphic Jacobi algebras.  The next step is to discard cases where the odd term in $p(x)$ has significantly greater degree than the even term.

\begin{lemma}\label{LemmaD}
If $f=xy^2+x^{2n}+\upvarepsilon x^{2m+1}$ where $\upvarepsilon \in\{0,1\}$ and $m\geq n$, then $y^3\in (\dcyc_xf,\dcyc_yf)$. In particular $y^3\in\lcl\dcyc_xf,\dcyc_yf \rcl$.
\end{lemma}
\begin{proof}
Set $I=(\dcyc_xf,\dcyc_yf)=(y^2+2nx^{2n-1}+\upvarepsilon (2m+1)x^{2m},xy+yx)$, and below write $\equiv$ for an equality mod $I$.  Since $y^2\equiv -2nx^{2n-1}-\upvarepsilon (2m+1)x^{2m}$, multiplying on the left by $y$ and on the right by $y$ gives
\begin{align*}
-2nx^{2n-1}y-\upvarepsilon (2m+1)x^{2m}y\equiv y^3
&\equiv-2nyx^{2n-1}-\upvarepsilon (2m+1)yx^{2m}\\
&\equiv 2nx^{2n-1}y-\upvarepsilon (2m+1)x^{2m}y
\end{align*}
where the last line holds since $xy\equiv -yx$ using the second generator of $I$. Inspecting the right and lefthand sides, the $x^{2m}y$ terms cancel, and so $4n x^{2n-1}y\equiv 0$, thus $x^{2n-1}y\equiv 0$.  Finally, since $m\geq n$, taking out the common factor we see that
\[
y^3\equiv (2n-\upvarepsilon (2m+1)x^{2m-2n-1})x^{2n-1}y\equiv 0.
\]
Thus $y^3\in I$. The final statement follows immediately.
\end{proof}

\begin{cor}\label{CorD}
If $f=xy^2+x^{2n}+\upvarepsilon x^{2m+1}$ where $m\geq 2n-1$, then $x^{4n-2}\in\lcl\dcyc_xf,\dcyc_yf \rcl$.
\end{cor}
\begin{proof}
Continue to write $\equiv$ for an equality mod $(\dcyc_xf,\dcyc_yf )$.  Then
\begin{align*}
x^{4n-2}=(-x^{2n-1})^2
&\equiv\tfrac{1}{(2n)^2}(y^2 + \upvarepsilon (2m+1)x^{2m})^2\tag{since $\dcyc_xf\equiv 0$}\\
&\equiv\tfrac{\upvarepsilon (2m+1)}{(2n)^2}(y^2x^{2m} + x^{2m}y^2 + \upvarepsilon (2m+1)x^{4m})\tag{$y^3\equiv0$ by \ref{LemmaD}}\\
&\equiv\tfrac{\upvarepsilon (2m+1)}{(2n)^2}(2x^{2m}y^2 + \upvarepsilon (2m+1)x^{4m})\tag{$xy\equiv -yx$}.
\end{align*}
Taking out the $x^{2m}$ common factor from the front, we may write $x^{4n-2}\equiv x^{2m}g$ for some~$g$ with no constant term. Then, since $2m\ge4n-2$ by assumption, we see that $x^{4n-2}\equiv x^{4n-2}(x^{2m-(4n-2)}g)$, and so $x^{4n-2}(1-x^{2m-(4n-2)}g)\equiv 0$. 

Given this statement holds mod $(\dcyc_xf,\dcyc_yf)$, it also holds mod $\lcl\dcyc_xf,\dcyc_yf \rcl$, and hence  $x^{4n-2}(1-x^{2m-(4n-2)}g)=0$ in $\Jac(f)$.  But there, $1-x^{2m-(4n-2)}g$ is a unit, and so it follows that $x^{4n-2}=0$ in $\Jac(f)$, as required.
\end{proof}

The above two results combine to remove the case when the odd-degree $x$ term is sufficiently larger than the even-degree $x$ term, as follows.

\begin{cor}\label{CorDtwo}
If $f=xy^2+x^{2n}+x^{2m+1}$ where $m\geq  2n-1$, then $f\cong xy^2+x^{2n}$.
\end{cor}
\begin{proof}
By \ref{CorD} we have $x^{4n-2}\in\lcl\dcyc_xf,\dcyc_yf \rcl$ and $x^{4n-2}\in\lcl y^2+2nx^{2n-1},xy+yx\rcl$.  Since $2m\geq  4n-2$, it follows that $x^{2m}$ belongs to \emph{both} of the ideals above, and thus
\begin{align*}
\lcl\dcyc_xf,\dcyc_yf \rcl &= \lcl y^2+2nx^{2n-1}+(2m+1)x^{2m},xy+yx, x^{2m}\rcl\\
&= \lcl y^2+2nx^{2n-1},xy+yx,x^{2m}\rcl\\
&= \lcl y^2+2nx^{2n-1},xy+yx\rcl.
\end{align*}
As this final ideal is obtained from $xy^2+x^{2n}$ by differentiation, the result follows.
\end{proof}

Summarising the above gives the following, which is the main result of this subsection.
\begin{cor}\label{cor:D two roots main}
Suppose that $f\in\ringtwo_{\geq 3}$ where $(f_3)^{\ab}$ has two roots.  Then either
\[
f\cong
\begin{cases}
xy^2&\\
xy^2+x^{2m+1}&m\geq 2\\
xy^2+x^{2n}&n\geq 2\\
xy^2+x^{2m+1}+x^{2n}&2\le m \le n-1\\
xy^2+x^{2n}+x^{2m+1}&2\le n\leq m\leq 2(n-1)
\end{cases}
\]
All of the above are mutually non-isomorphic.
\end{cor}
\begin{proof}
The fact that the stated list covers all cases follows from \ref{prop!polypotential}, using \ref{CorDtwo} to discount the case when the odd-degree $x$ term is sufficiently larger than the even-degree $x$ term. We now claim that the potentials listed give pairwise non-isomorphic Jacobi algebras.

The first two families both have infinite dimensional Jacobi algebras, whereas the bottom three are all finite dimensional.  As such, the only possibilities for isomorphisms are between members in families one and two, or between members in families three, four and five.  But $\dim_{\mathbb{C}}\Jac(xy^2)^{\ab}=\infty$, whereas $\dim_{\mathbb{C}}\Jac(xy^2+x^{2m+1})^{\ab}=2m+2$, and so all members of families one and two are mutually non-isomorphic.

For the final three families, all members of families three and four and mutually non-isomorphic, as can be seen by extending the method of \cite[4.7]{BW}, or by using \cite[5.10]{Kawamata} directly. Further, all members of family five are also mutually non-isomorphic for dimension reasons, since for $f$ in family five $\dim_{\mathbb{C}}\Jac(f)^{\ab}=2m+2$ and $\dim_{\mathbb{C}}\Jac(f)=(2m+2)+4(n-1)$ by either \cite[\S5]{Okke} or \S\ref{sec:highcorank}, and thus we can distinguish between all different $m$ and $n$.  The only remaining possibility is an isomorphism between a member of family five, and a member of family three or four.  But by above the dimension of $\Jac(f)$ for $f$ in family five is even, and the dimension of $\Jac(g)$ for $g$ in families three and four is odd \cite[5.10]{Kawamata}, so there can be no such isomorphisms. 
\end{proof}

%-----------------------------------------------------------------------------------------
\subsection{Overview of Type~D normal forms}\label{sec: overview D}

The previous subsections combine to give the following, which is the main result of this section.

\begin{thm}\label{main Type D}
Let $f\in\ringd_{\geq 2}$ with $\corank(f)=2$ and $\corank_2(f)=2$. Then either
\[
f\cong
\begin{cases}
z_1^2+\cdots+z_{d-2}^2 + xy^2&\\
z_1^2+\cdots+z_{d-2}^2 + xy^2 + x^{2m+1}&\mbox{with }m\geq 1\\
z_1^2+\cdots+z_{d-2}^2 + xy^2 + x^{2n}&\mbox{with }n\geq 2\\
z_1^2+\cdots+z_{d-2}^2 + xy^2 + x^{2n} + x^{2m+1}&\mbox{with }n\geq 2,\,\, n\leq m\leq 2n-2\\
z_1^2+\cdots+z_{d-2}^2 + xy^2 + x^{2m+1} + x^{2n}&\mbox{with }m\geq 1,\,\, n\geq m+1.
\end{cases}
\]
The Jacobi algebras of these potentials are all mutually non-isomorphic, and furthermore the following statements hold.
\begin{enumerate}
\item Every $f$ in the first two families satisfies $\JRdim\Jac(f)=1$.
\item\label{main Type D B} Every $f$ in the last three families satisfies $\JRdim\Jac(f)=0$.
\begin{enumerate}
\item For any fixed $n\geq 2$, the algebras in families three and four combine to give $n-1$ non-isomorphic Jacobi algebras, all of which satisfy $\dim_{\mathbb{C}}\Jac(f)^{\ab}=2n+1$ and $\dim_{\mathbb{C}}\Jac(f)=(2n+1)+4(n-1)=6n-3$.
\item In the fifth family, $\dim_{\mathbb{C}}\Jac(f)^{\ab}=2m+2$ and $\dim_{\mathbb{C}}\Jac(f)=(2m+2)+4(n-1)$. 
\end{enumerate}
\end{enumerate}
\end{thm}
\begin{proof}
By the Splitting Lemma~\ref{thm!splitting} the condition $\corank(f)=2$ implies that $f\cong z_1^2+\hdots+z_{d-2}^2+g$ for some $g\in\mathbb{C}\llangle x,y\rrangle_{\geq 3}$.  The condition $\corank_2(f)=2$ is then equivalent to the first two cases in \ref{linear change}, namely those $g\in\mathbb{C}\llangle x,y\rrangle_{\geq 3}$ with $g_3\neq 0$ for which $g_3^{\ab}$ has either two or three distinct linear factors.  The options for all such $g$ thus follow from combining \ref{NC3rootsnormal2} and \ref{cor:D two roots main}

The fact that $\JRdim\Jac(xy^2+x^3)=1$ follows since $\Jac(xy^2+x^3)\cong\Jac(x^3+y^3)$ by linear change in coordinates, and $\JRdim\Jac(x^3+y^3)=1$ by \ref{NC3rootsnormal}\eqref{x3y3i}.  The statements that $\JRdim\Jac(xy^2+x^{2m-1})=1$ for all $m\geq 2$ and $\JRdim\Jac(xy^2)=1$ can be shown by a very similar explicit method as in the proof of \ref{NC3rootsnormal}\eqref{x3y3i}, or alternatively by using \ref{cor: Acon has Jdim leq 1} below, once we know (in \ref{D4 div to curve main}) that all such Jacobi algebras are contraction algebras.  The stated vector space dimensions of the Jacobi algebras in all remaining cases have already been justified in the proofs of \ref{NC3rootsnormal2} and \ref{cor:D two roots main} respectively.

The fact that the above are all mutually non-isomorphic, and thus a list of normal forms, then follows. Indeed, by inspecting $\mathfrak{J}$-dimension, the only possible isomorphisms are between members of families one and two, or between members of families three, four and five.  Given we have just added the normal forms of \ref{NC3rootsnormal2} to the normal forms of \ref{cor:D two roots main}, the only remaining possible isomorphisms are between these two cases. But again, either the dimension of the abelianisation, or the dimension of the contraction algebra itself, distinguishes in all cases.
\end{proof}

%-----------------------------------------------------------------------------------------
\section{Central Elements and General Elephants}\label{central sect}

This section algebraically extracts ADE information from the normal forms in \S\ref{tables section}, using generic central elements and contracted preprojective algebras.

\subsection{The Six Algebras}\label{subsec:sixalgebras}
As notation, consider the following ADE Dynkin diagrams, which we also furnish with the information of their highest roots.
\begin{equation}
\begin{array}{c}
\begin{tikzpicture}[scale=0.7]
\node at (-0.4,-1) {\qquad$A_1$};
\node at (0,0) {$
\begin{tikzpicture}[scale=0.45]
 \node (0) at (0,0) [DBlue] {};
 \node (0a) at (0,-0.5) {$\scriptstyle 1$};
 \node[draw=none] at (0,1) {};
\end{tikzpicture}$};
\node at (2.1,-1) {$D_4$};
\node at (2,0) {\begin{tikzpicture}[scale=0.45]
 \node (0) at (0,0) [DB] {};
 \node (1) at (1,0) [DBlue] {};
 \node (1b) at (1,0.75) [DB] {};
 \node (2) at (2,0) [DB] {};
 \node (0a) at (0,-0.5) {$\scriptstyle 1$};
 \node (1a) at (1,-0.5) {$\scriptstyle 2$};
 \node (1ba) at (1.4,0.75) {$\scriptstyle 1$};
 \node (2a) at (2,-0.5) {$\scriptstyle 1$};
\draw [-] (0) -- (1);
\draw [-] (1) -- (2);
\draw [-] (1) -- (1b);
\end{tikzpicture}};
\node at (5.2,-1) {$E_6$};
\node at (5.1,0) {$\begin{tikzpicture}[scale=0.45]
 \node (m1) at (-1,0) [DB] {};
 \node (0) at (0,0) [DB] {};
 \node (1) at (1,0) [DBlue] {};
 \node (1b) at (1,0.75) [DB] {};
 \node (2) at (2,0) [DB] {};
 \node (3) at (3,0) [DB] {};
\node (m1a) at (-1,-0.5) {$\scriptstyle 1$};
 \node (0a) at (0,-0.5) {$\scriptstyle 2$};
 \node (1a) at (1,-0.5) {$\scriptstyle 3$};
 \node (1ba) at (1.4,0.75) {$\scriptstyle 2$};
 \node (2a) at (2,-0.5) {$\scriptstyle 2$};
 \node (3a) at (3,-0.5) {$\scriptstyle 1$};
\draw [-] (m1) -- (0);
\draw [-] (0) -- (1);
\draw [-] (1) -- (2);
\draw [-] (2) -- (3);
\draw [-] (1) -- (1b);
\end{tikzpicture}$};
\node at (9.25,-1) {$E_7$};
\node at (9,-0) {$\begin{tikzpicture}[yscale=0.45,xscale=-0.45]
 \node (m1) at (-1,0) [DB] {};
 \node (0) at (0,0) [DB] {};
 \node (1) at (1,0) [DBlue] {};
 \node (1b) at (1,0.75) [DB] {};
 \node (2) at (2,0) [DB] {};
 \node (3) at (3,0) [DB] {};
 \node (4) at (4,0) [DB] {};
\node (m1a) at (-1,-0.5) {$\scriptstyle 2$};
 \node (0a) at (0,-0.5) {$\scriptstyle 3$};
 \node (1a) at (1,-0.5) {$\scriptstyle 4$};
 \node (1ba) at (0.6,0.75) {$\scriptstyle 2$};
 \node (2a) at (2,-0.5) {$\scriptstyle 3$};
 \node (3a) at (3,-0.5) {$\scriptstyle 2$};
 \node (4a) at (4,-0.5) {$\scriptstyle 1$};
\draw [-] (m1) -- (0);
\draw [-] (0) -- (1);
\draw [-] (1) -- (2);
\draw [-] (2) -- (3);
\draw [-] (3) -- (4);
\draw [-] (1) -- (1b);
\end{tikzpicture}$};
\node at (14,-1) {$E_8$};
\node at (13.5,0) {$\begin{tikzpicture}[yscale=0.45,xscale=-0.45]
 \node (m1) at (-1,0) [DB] {};
 \node (0) at (0,0) [DB] {};
 \node (1) at (1,0) [DBlue] {};
 \node (1b) at (1,0.75) [DB] {};
 \node (2) at (2,0) [DBlue] {};
 \node (3) at (3,0) [DB] {};
 \node (4) at (4,0) [DB] {};
  \node (5) at (5,0) [DB] {};
\node (m1a) at (-1,-0.5) {$\scriptstyle 2$};
 \node (0a) at (0,-0.5) {$\scriptstyle 4$};
 \node (1a) at (1,-0.5) {$\scriptstyle 6$};
 \node (1ba) at (0.6,0.75) {$\scriptstyle 3$};
 \node (2a) at (2,-0.5) {$\scriptstyle 5$};
 \node (3a) at (3,-0.5) {$\scriptstyle 4$};
 \node (4a) at (4,-0.5) {$\scriptstyle 3$};
 \node (5a) at (5,-0.5) {$\scriptstyle 2$};
\draw [-] (m1) -- (0);
\draw [-] (0) -- (1);
\draw [-] (1) -- (2);
\draw [-] (2) -- (3);
\draw [-] (3) -- (4);
\draw [-] (4) -- (5);
\draw [-] (1) -- (1b);
\end{tikzpicture}$};
\end{tikzpicture}
\end{array}\label{eqn: Dynkin}
\end{equation}
To each such Dynkin diagram, there is an associated preprojective algebra $\Uppi$ (see e.g.\ \cite{CBH}), which is a finite dimensional algebra. The vertices of the Dynkin diagram give rise to idempotents in the corresponding $\Uppi$. In each diagram in \eqref{eqn: Dynkin}, let $e$ be the idempotent corresponding to the unique vertex marked $\begin{tikzpicture}\node at (0,0) [DBlue] {};\end{tikzpicture}$, except for $E_8$ when there are two cases: $e$ is either the left $\begin{tikzpicture}\node at (0,0) [DBlue] {};\end{tikzpicture}$ or the right $\begin{tikzpicture}\node at (0,0) [DBlue] {};\end{tikzpicture}$.  From this information, consider the algebra $e\Uppi e$.  

\begin{remark}\label{rem:5thisstrange}
Before relating the above to the introduction, we remark that the fifth algebra in \eqref{eqn:6keyalgebras} has full presentation
\begin{equation}
\frac{\mathbb{C}\langle x,y,z\rangle}{\left(\begin{array}{c}
x^2+y+z,\\ x^4,\,\, z^2+xyx,\,\, yxy+y^2x+yx^3 \end{array}\right)}.\label{eqn:theawkwardone}
\end{equation}
There are many equivalent presentations, with the point being that all are necessarily less pretty than the other five in \eqref{eqn:6keyalgebras}. Indeed, \ref{lem:sixalgareePie} below proves that the algebra in \eqref{eqn:theawkwardone} is isomorphic to $e\Uppi e$, where $e$ is the vertex $\begin{tikzpicture}\node at (0,0) [DBlue] {};\end{tikzpicture}$ marked $5$ in $E_8$.  This is the only time that the chosen vertex~$\begin{tikzpicture}\node at (0,0) [DBlue] {};\end{tikzpicture}$ in \eqref{eqn: Dynkin} is not the central vertex in the Dynkin diagram, and so slightly different behaviour should be expected.  We also remark that this algebra is strictly needed in order for noncommutative singularity theory to distinguish between the two different types of $E_8$ flop (of length $5$ and $6$, respectively).
\end{remark}

The upshot from \eqref{eqn: Dynkin} is that there are six algebras $e\Uppi e$, corresponding to the six different vertices marked $\begin{tikzpicture}\node at (0,0) [DBlue] {};\end{tikzpicture}$. The following result asserts that these give a presentation-free description of the six algebras of \eqref{eqn:6keyalgebras} in the introduction.
\begin{lemma}\label{lem:sixalgareePie}
Consider in order the algebras $e\Uppi e$ where $e$ is the vertex $\begin{tikzpicture}\node at (0,0) [DBlue] {};\end{tikzpicture}$ with label \textnormal{$1,\hdots, 6$} in \eqref{eqn: Dynkin}.  This list is isomorphic to the list of algebras in \eqref{eqn:6keyalgebras}, reading left to right.
\end{lemma}
\begin{proof}
When $\begin{tikzpicture}\node at (0,0) [DBlue] {};\end{tikzpicture}$ is a central vertex (which covers all cases except the awkward one in \ref{rem:5thisstrange}), the isomorphism is a direct application of \cite[Theorem 1]{Melit};  see also \cite[p53]{CB}. The final remaining case when $e$ is the vertex $\begin{tikzpicture}\node at (0,0) [DBlue] {};\end{tikzpicture}$ marked $5$ in $E_8$, can be proved using Auslander--Reiten theory; alternatively we may simply observe that $e\Uppi e$ is the factor of the algebra in \cite[1.3(5)]{Karmazyn} by the ideal generated (in the notation there) by $a,a^*,t,T_0^\upbeta,\hdots,T_1^{\updelta}$.  The result \cite[1.3(5)]{Karmazyn} then gives a presentation 
\[
e\Uppi e\cong\,\, \frac{\mathbb{C} \langle\upbeta,\upgamma\rangle }{(\upbeta^4,\,\,\upgamma\upbeta\upgamma+\upgamma^2\upbeta+\upgamma\upbeta^3,\,\, (\upgamma+\upbeta^2)^2+\upbeta\upgamma\upbeta)}
\]
which is clearly isomorphic to the algebra in \eqref{eqn:theawkwardone}.
\end{proof}

\subsection{Generic Central Sections}
For any $f\in\ringd_{\geq 2}$, set $\scrZ=\scrZ(\Jac(f))$ to be the centre of $\Jac(f)$, and write $\m_\scrZ$ for the Jacobson radical of $\scrZ$. Recall that $\mathfrak{J}$ denotes the Jacobson radical of the local ring $\Jac(f)$.

\begin{lemma}
We have $\m_\scrZ=\mathfrak{J}\cap\scrZ$ and $\scrZ/\m_\scrZ\cong\C$.  Thus $\scrZ$ is also a local ring.  % could be a field C - but we show easily next that m \not= 0.
\end{lemma}
\begin{proof}
Set $J_f=\lcl \updelta_1f,\hdots,\updelta_df\rcl$, then it is clear that $\mathfrak{J}\cap\scrZ=\{ g+J_f\in\scrZ\mid g\in\n \}$. This set is clearly a two-sided ideal of $\scrZ$, and further $1+\mathfrak{J}\cap\scrZ$ consists of units in $\scrZ$.  These two properties imply that $\mathfrak{J}\cap\scrZ$ equals the Jacobson radical  $\m_\scrZ$ of $\scrZ$, see e.g.\ \cite[4.5]{Lam}.  The fact that $\scrZ/\m_\scrZ\cong\C$ is clear.
\end{proof}

Generic elements of the centre $\scrZ$ will be used to intrinsically extract ADE information. 
\begin{defin}\label{def:generic element}
Given $f\in\ringd_{\geq 2}$, we say that $\Jac(f)$ has Type $X$ if for all finite dimensional vector spaces $V\subset\m_\scrZ$ such that $V\twoheadrightarrow\m_\scrZ^{\phantom 2}/\m_\scrZ^2$, there exists a Zariski open subset $U$ of $V$ such that $\Jac(f)/(u)\cong e\Uppi e$ for all $u\in U$, where $\Uppi$ is the preprojective algebra of Type~$X$, and $e$ is an idempotent marked $\begin{tikzpicture}\node at (0,0) [DBlue] {};\end{tikzpicture}$ in \eqref{eqn: Dynkin}.   
\end{defin}
Equivalently, in the language of \cite[2.5]{C3F}, $\Jac(f)$ has Type $X$ provided that a general hyperplane section $u$ of $\scrZ$ satisfies $\Jac(f)/(u)\cong e\Uppi e$ where $\Uppi$ is the preprojective algebra of Type~$X$, and $e$ is an idempotent marked $\begin{tikzpicture}\node at (0,0) [DBlue] {};\end{tikzpicture}$ in \eqref{eqn: Dynkin}.  We also remark that there are two different Type $E_8$'s in \ref{def:generic element}, corresponding to the two different choices of $\begin{tikzpicture}\node at (0,0) [DBlue] {};\end{tikzpicture}$ in $E_8$ in \eqref{eqn: Dynkin}.  This feature matches the two different $E_8$ cases in the length classification of flops \cite{KM}. 

Much like the definition of cDV singularities, \ref{def:generic element} is only designed to be useful in specific situations.  Indeed, for general $f\in\ringd_{\geq 2}$, it is not clear whether the centre $\scrZ$ of $\Jac(f)$ is noetherian, nor whether its maximal ideal $\m_\scrZ$ is finitely generated as a $\scrZ$-module. Consequently, work is required to establish that $\m_\scrZ^{\phantom 2}/\m_\scrZ^2$ is finite dimensional, which is needed for there to exist a finite dimensional vector space $V$ surjecting onto it.

When $\Jac(f)$ is finite dimensional, these difficulties disappear, since $\Jac(f)$ and thus $\scrZ$, $\m_\scrZ$ and $\m_\scrZ^{\phantom 2}/\m_\scrZ^2$ are all finite dimensional vector spaces.  Other cases are more tricky, but for our purposes the following suffices.

\begin{lemma}\label{lem: D center}
If $f$ is a normal form from \textnormal{\ref{main Type D}},  then the following statements hold.
\begin{enumerate}
\item\label{lem: D center 1} If $u\in\m_\scrZ$, then  $u\equiv \uplambda x^2 + h$ in $\Jac(f)$ for some $h\in \ringtwo_{\geq 3}$.
\item\label{lem: D center 2} If further $f$ has Type~$D_{\infty, \infty}$ or $D_{\infty,m}$, then $\scrZ\cong\C\llsq x^2\rrsq$. 
\end{enumerate}
In particular, in all cases $\dim_\C(\m_\scrZ^{\phantom 2}/\m_\scrZ^2)<\infty$.
\end{lemma}
\begin{proof}
(1) In all cases $\updelta_yf=xy+yx$, and so certainly $x^2$ commutes with $y$ in $\Jac(f)$. Obviously $x^2$ commutes with $x$, thus since we are considering closed ideals, it follows that $x^2$ is central in $\Jac(f)$.   Similarly $y^2$ is central.  

We next claim that there are no elements in $\m_\scrZ$ that contain linear terms.  Write $u\in\m_\scrZ$ as $u= k + g$ for some $k=\uplambda_1x +\uplambda_2 y$ and $g\in\ringtwo_{\geq 2}$. Now $u$ remains central after factoring by $\mathfrak{J}^3$, so set $I=\lcl \updelta_xf,\updelta_yf, \n^3\rcl$, and observe that
\[
0+I = [x+I, u+I] =  2\uplambda_2xy +I
\]
For any (finite-dimensional) $\ringtwo/\lcl \updelta_xf,\updelta_yf, \n^r\rcl$, $r\ge3$,
we may choose a basis of the form
$1,x,\hdots,x^i,xy,x^2y,\hdots,x^jy$ for some $i,j\ge2$.
Since $xy$ forms part of this basis it follows that $\uplambda_2=0$. Repeating using the commutator $[y+I, u+I]$ shows that $\uplambda_1=0$. 

Thus $u= g$ for some $g\in\ringtwo_{\geq 2}$. Using the relation $xy+yx$ to move $x$'s to the left, and the other relation to move $y^2$ either to zero, to $x^2$, or into higher degree, write
\[
u\equiv \uplambda_1 x^2 + \uplambda_2 xy + h
\]
in $\Jac(f)$, for some $h\in\ringtwo_{\geq 3}$.  We next claim that $\uplambda_2=0$.  Since $u$ is central, and $x^2$ is central, it follows that $v\colonequals\uplambda_2 xy + h
$ is also central in $\Jac(f)$. In particular it is still central after factoring by $\mathfrak{J}^4$.  Set $I=\lcl \updelta_xf,\updelta_yf, \n^4\rcl$, so that
\[
0+I=[x+I, v+I]=2\uplambda_2x^2y+I
\]
and thus $2\uplambda_2 x^2y\in I$.  But $x^2y$ forms part of a basis of $\ringtwo/I$, so $\uplambda_2=0$.

\noindent
(2) Either set $I=( xy+yx, y^2)$, or $( xy+yx, y^2+x^{2m})$, so by assumption $\Jac(f)\cong\ringtwo/\lcl I\rcl$.  Consider an arbitrary element $u\in\scrZ$.  By using the first relation to move all the $x$'s to the left, and the second relation to either move $y^2$ to zero or to higher powers of $x$, we may write $u\equiv p+qy$ in $\Jac(f)$, where by \eqref{lem: D center 1} $p\in\C\llsq x\rrsq_{\geq 2}$ and $q\in\C\llsq x\rrsq_{\geq 2}$. Observe that in $\Jac(f)$
\[
0\equiv [x,u]\equiv [x,p+qy]=2xqy
\]
and so $xqy\in \lcl I\rcl$.  Thus $x q_{\leq t-3}y\in I+\n^t$ for all $t\geq 3$.  Now $\ringtwo/(I+\n^t)$ has basis $1,x,\hdots,x^{t-1}, y,xy,\hdots,x^{t-2}y$. Write $q_{\leq t-3}=\sum_{i=2}^{t-3}\uplambda_ix^i$, then the second part of this basis on the equation $x q_{\leq t-3}y\in I+\n^t$ shows that $\uplambda_2=\hdots=\uplambda_{t-3}=0$.  This holds for all~$t$, and so $q=0$.

Thus the central element $u\equiv p$.  Splitting into even and odd terms, write $u\equiv P(x^2)+xQ(x^2)$ in $\Jac(f)$ for some $P,Q\in\C\llsq x\rrsq_{\geq 1}$.  Then, in $\Jac(f)$, since $x^2$ is central
\[
0\equiv [y,u]\equiv [y,P+xQ]\equiv -2(xQ)y.
\]
Using the same argument as above, $Q=0$, and so $u\equiv P(x^2)$, as claimed.  This shows that $\scrZ\subseteq \C\llsq x^2\rrsq$, but since $x^2$ is central by \eqref{lem: D center 1}, equality holds, proving \eqref{lem: D center 2}.

For the very last statement, all the finite-dimensional $\Jac(f)$ satisfy $\dim_\C(\m_\scrZ^{\phantom 2}/\m_\scrZ^2)<\infty$. Since the only other potentials in \ref{main Type D} are those in \eqref{lem: D center 2}, where visibly $\dim_\C(\m_\scrZ^{\phantom 2}/\m_\scrZ^2)=1$, the final statement follows.
\end{proof}

It follows from \ref{lem: D center}  that $\m_\scrZ^{\phantom 2}/\m_\scrZ^2$ is finite dimensional for any $f\in\ringd_{\geq 2}$ with $\corank(f)=2$ and $\corank_2(f)=2$.  The is a rather remarkable use of normal forms: we have no method to prove such a result without using \ref{main Type D}.

\subsection{ADE preliminaries}\label{sec: ADE prelim}
The next problem is to exhibit a single element of the centre $\scrZ$ that gives an ADE quotient.  For Type $A$ and $D$ this turns out to be easy, but Type $E$ requires the following preparation.
Consider the elements
\[
g_{6,n}\colonequals 
\begin{cases}
x^2 + xyx + yx^2	&\mbox{if } n=3t+1 \mbox{ with }t\geq 1\\
x^2 + xyx + yx^2 + (-1)^{t}3^t(3t+2)x^{2t+1} &\mbox{if } n=3t+2 \mbox{ with }t\geq 1\\
x^2 + xyx + yx^2 + (-1)^{t+1}3^tt( xyx^{2t-2} + yx^{2t-1} )&\mbox{if }n=3t \mbox{ with }  t\geq 2.\\
\end{cases}
\]
The following establishes, in the cases $E_{6,n}$, that the centre of $\Jac(f)$ is non-trivial, and that $\m_\scrZ$ is at least two dimensional as a vector space.
\begin{lemma}\label{lem: E6 central}
If $f$ has Type~$E_{6,n}$, then $x^2$ is central in $\Jac(f)$, as is $g_{6,n}$.
\end{lemma}
\begin{proof}
The first statement follows from the relation $3x^2+y^3\equiv 0$, which implies that $yx^2\equiv -\tfrac{1}{3}y^4\equiv x^2y$ and thus $x^2$ commutes with $y$. Since $x^2$ clearly commutes with $x$ and we are considering closed ideals, it follows that $x^2$ is central in $\Jac(f)$. 

For the second statement, we establish the first case, with the proofs of all other cases being similar. For this, it suffices to show that $xyx + yx^2$ is also central in $\Jac(f)$ when $n=3t+1$ with $t\geq 1$.  We first claim that $xyxy-yxyx\in \lcl \updelta_xf,\updelta_yf\rcl$.  This follows since $n-1=3t$, then $y^{n-1}\equiv (-3x^2)^t$ is central, and thus the commutator
\[
\lcl\updelta_xf,\updelta_yf\rcl \ni [x,\updelta_yf]\equiv (xyxy-yxyx)+n(xy^{n-1}-y^{n-1}x)\equiv xyxy-yxyx.
\]
Using this, again with the fact that $x^2$ is central, it follows that 
\begin{align*}
[x,xyx+yx^2]&\equiv (yx^3+x^3y)-(x^3y+yx^3) =0\\
[y,xyx+yx^2]&\equiv (yxyx+y^2x^2)-(xyxy+yx^2y) \equiv 0.
\end{align*}
Thus $xyx+yx^2$ commutes with both $x$ and $y$, and so is central in $\Jac(f)$.
\end{proof}

\subsection{Extracting ADE}
We are now in a position to extract ADE using general hyperplane sections of the centre.  In what follows, in the case that $\Jac(f)$ is finite dimensional, all ideals are automatically closed.  In the cases when $\JRdim\Jac(f)=1$ this fact is also true for Type $A$ by inspection, and for Type $D$ by e.g.\ \ref{lovely completion fact} and \ref{all type D are geometric} below.  As such, in the following technically we should temporarily write $\Jac(f)/\lcl g\rcl$ when considering Type $D_{\infty,m}$ and $D_{\infty,\infty}$ until we have established \ref{lovely completion fact} and \ref{all type D are geometric}. However, since \ref{lovely completion fact} and \ref{all type D are geometric} are logically independent of what follows, we refrain from doing so, and drop the double bracket to ease notation.

\begin{thm}\label{central element main}
Consider the normal forms $A_n$, $D_{n,m}$, $D_{n,\infty}$, $E_{6,n}$, $A_{\infty}$, $D_{\infty,m}$, $D_{\infty,\infty}$ and $E_{6,\infty}$ from \S\textnormal{\ref{tables section}}.  In each case, define an element $s$ as follows
\[
\begin{tabular}{clp{3cm}c}
\toprule
{\bf Type}&{\bf Normal form}&{\bf Conditions}&$s$\\
\midrule
\textnormal{A} & $z_1^2+\hdots+z_{d-2}^2+x^2+\upvarepsilon_1y^n$ & $n\in\mathbb{N}_{\geq 2}\cup\{\infty \}$ & $y$\\[4pt]
\textnormal{D} & $z_1^2+\hdots+z_{d-2}^2+xy^2+\upvarepsilon_2x^{2n}+\upvarepsilon_3 x^{2m-1}$ & $m,n\in\mathbb{N}_{\geq 2}\cup\{\infty \}$ & $x^2$\\[4pt]
\textnormal{E} & $z_1^2+\hdots+z_{d-2}^2+x^3+xy^3+\upvarepsilon_4 y^n$ & $n\in\mathbb{N}_{\geq 4}$&{$g_{6,n}$}\\
\bottomrule
\end{tabular}
\]
where $g_{6,n}$ is defined in \S\textnormal{\ref{sec: ADE prelim}} above. Then the following statements hold.
\begin{enumerate}
\item\label{central element main 1}  The element $s$ is central in $\Jac(f)$, and $\Jac(f)/(s)\cong e\Uppi e$, where $\Uppi$ is the preprojective algebra of Type~$A_1$, $D_4$, or $E_6$, and $e$ is the idempotent marked $\begin{tikzpicture}\node at (0,0) [DBlue] {};\end{tikzpicture}$.  
\item\label{central element main 2} 
Normal forms of Type $A$ and $D$ give rise to Jacobi algebras which have Type $A$ and $D$ respectively, in the sense of \textnormal{\ref{def:generic element}}.
\end{enumerate} 
\end{thm}
\begin{proof}
For (1), Type~$A$ is clear, since $\Jac(f)\cong\mathbb{C}\llsq y\rrsq/y^{n-1}$ or $\C\llsq y\rrsq$ depending on whether $\upvarepsilon_1$ is $1$ or $0$.  In both cases $y$ is central, and the quotient is $\Jac(f)/y\cong\C\cong e\Uppi e$ where $\Uppi$ is the preprojective algebra of Type~$A_1$.\\
Type~$D$ is similar.  The fact $x^2$ is central was justified in \ref{lem: D center}\eqref{lem: D center 1}.   But then since $(x^2)$ is a closed ideal of $\Jac(f)$, setting $\uplambda_2=(2n)\upvarepsilon_2$ and $\uplambda_3=(2m-1)\upvarepsilon_3$ it follows that
\begin{align*}
\Jac(f)/(x^2)
&\cong\frac{\ringtwo}{\lcl x^2, xy+yx, y^2+\uplambda_2x^{2n-1}+\uplambda_3x^{2m-2}\rcl}\tag{by \ref{closed ideals 101}}\\
&\cong\frac{\ringtwo}{\lcl x^2, xy+yx, y^2\rcl}\tag{since $2n-1\geq 2$ and $2m-2\geq 2$}\\ 
&\cong e\Uppi e\tag{by \ref{lem:sixalgareePie}}
\end{align*}
where $\Uppi$ is the preprojective algebra of Type~$D_4$, and $e$ is the central vertex.

Type $E$ is more involved. All proofs turn out to be similar, so here we illustrate the technique by considering the case $f$ of Type~$E_{6,n}$ with $n=3t+1$ and $t\ge2$.  Certainly $g_{6,n}$ is central by \ref{lem: E6 central}.  After rescaling the $x$ and $y$ appropriately,
\[
\Jac(f)\cong\frac{\ringtwo}{\lcl -x^2+y^3,xy^2+yxy+y^2x+y^{3t}\rcl}\,\,\reflectbox{\colonequals} \,\,\Upgamma
\]
and we work on the right-hand side. Under this identification, the element $g_{6,n}$ becomes $\uplambda x^2 + \upmu(xyx+yx^2)$ for some non-zero scalars $\uplambda$ and $\upmu$.
We now claim that for any non-zero scalars $\uplambda$ and $\upmu$ the factor 
\[
A\colonequals \frac{\Upgamma}{(\uplambda x^2 + \upmu(xyx+yx^2) )}=\frac{\ringtwo}{\lcl -x^2+y^3,xy^2+yxy+y^2x+y^{3t}, \uplambda x^2 + \upmu(xyx+yx^2)\rcl}
\] 
is isomorphic to the model algebra $B \colonequals \ringtwo / \lcl x^2, y^3, (x+y)^3 \rcl \cong e\Uppi e$ in  \ref{lem:sixalgareePie}.  The result will then follow, since for particular $\uplambda,\upmu$, there is an isomorphism $A\cong \Jac(f)/(g_{6,n})$.

To establish the claim, note first that $x^3\equiv 0$ in $A$ for any $t\ge2$, as follows. The additional relation gives $-\uplambda x^3\equiv  \upmu (x^2yx+xyx^2)$, which equals $\upmu(yx^3+x^3y)$ since $x^2$ is central.  Repeating, we may push $x^3$ into higher and higher degrees, and so $x^3$ belongs to the closed ideal defining $A$, as claimed.  Given $x^3\equiv 0$ in $A$, it follows from the first relation that $y^6\equiv x^4\equiv 0$ in $A$, and since $t\geq 2$ also that $y^{3t}\equiv 0$ in $A$. Consequently
\begin{align}
A &\cong \frac{\ringtwo}{\lcl -x^2+y^3,xy^2+yxy+y^2x,\uplambda x^2 + \upmu(xyx+yx^2), y^{3t}\rcl}\label{eq: A1}\\
&= \frac{\ringtwo}{\lcl -x^2+y^3,xy^2+yxy+y^2x,\uplambda x^2 + \upmu(xyx+yx^2)\rcl}\label{eq: A2}.
\end{align}
where the last equality holds since $y^6$ belongs to the closed ideal in \eqref{eq: A1}, and $t\geq 2$.  This latter presentation has no dependence on $t$.

Now, composing the automorphism $\upphi\colon \ringtwo\rightarrow \ringtwo$ defined by
\[
x\mapsto x - (xy+yx) + yxy,
\qquad
y\mapsto (x+2y) - y^2 - yx
\]
with $\ringd\twoheadrightarrow B$ gives a surjective homomorphism $\ringd\twoheadrightarrow B$.  It is elementary to check that the three relations of $A$ in \eqref{eq: A2} map to zero, and hence since $\upphi$ is continuous, it extends to the closure of ideals and thus induces a surjection $\upphi\colon A\twoheadrightarrow B$.

Using the same method, it is also elementary to check that the automorphism $ \ringtwo\rightarrow \ringtwo$ given by
\[
x\mapsto  - \tfrac{2}{3}x(3 - 2x) + (y -2)xy , \qquad
y\mapsto x(1 + \tfrac{25}{48}x) + y(1 + \tfrac{1}{4}y +x)
\]
descends to a surjective map $B\twoheadrightarrow A$. Thus $\dim A \leq \dim B=12$, and in particular $\dim A$ is also finite.  The previous surjection $A\twoheadrightarrow B$ then implies that $\dim A=\dim B$, so that $\upphi\colon A\to B$, being a surjective map between algebras of the same dimension, is necessarily an isomorphism.  

\medskip
\noindent
(2) For potentials of Type $A$, clearly $\Jac(f)$ is either $\C\llsq y\rrsq/(y^{n-1})$ or $\C\llsq y\rrsq$, both of which are commutative, so $\scrZ=\Jac(f)$ and further $\m_\scrZ^{\phantom 2}/\m_\scrZ^2$ is spanned by the image of~$y$.  Given any finite dimensional vector space $V\subset\m_\scrZ$ such that  $\uppi\colon V\twoheadrightarrow\m_\scrZ^{\phantom 2}/\m_\scrZ^2$, set $\scrU_1=\{\uplambda\in \mathbb{A}^1 \mid \uplambda\neq 0\}$, $\scrU=\uppi^{-1}(\scrU_1)$, and let $u\in \scrU$.  Then for all $u\in\scrU$, write $u= \uplambda y+p$ in $\Jac(f)$ for some $\uplambda\neq 0$ and some $p\in\C\llsq y\rrsq_{\geq 2}$.  In particular $u$ equals $y$ multiplied by a unit, and so $\Jac(f)/(u)\cong \Jac(f)/(y)\cong \C\cong e\Uppi e$, where $\Uppi$ is the preprojective algebra of Type $A$. 

Lastly, consider Type $D$.  By \ref{lem: D center} all potentials $f$ in \ref{main Type D} satisfy $\dim_\C(\m_\scrZ^{\phantom 2}/\m_\scrZ^2)<\infty$.  Hence we can again consider a finite dimensional vector space $V\subset\m_\scrZ$, such that  $\uppi\colon V\twoheadrightarrow\m_\scrZ^{\phantom 2}/\m_\scrZ^2$.  Since $x^2\in\m_\scrZ$ by \eqref{central element main 1}, and $\m_\scrZ$ contains no linear terms (as justified in \ref{lem: D center}\eqref{lem: D center 1}), $x^2$ is non-zero in $\m_\scrZ^{\phantom 2}/\m_\scrZ^2$.  Thus set $b_1=x^2+\m_\scrZ^2$, and extend to a basis $b_1,\hdots,b_t$.  Set $\scrU_1=\{\sum\uplambda_ib_i\mid \uplambda_1\neq 0\}$, and $\scrU=\uppi^{-1}(\scrU_1)$.  

Let $u\in \scrU$, then by \ref{lem: D center}\eqref{lem: D center 1} $u\equiv \uplambda x^2 + h$ in $\Jac(f)$, for some $h\in\ringtwo_{\geq 3}$.  The assumption $u\in\scrU$ ensures that $\uplambda\neq 0$.  By the relation $xy\equiv -yx$, we may pull all the $x$'s in $h$ to the left, and since $h$ has order at least three, afterwards each term either starts with $x^2$, or ends with $y^2$.  Thus in $\Jac(f)$
\[
h\equiv x^2 r + xy^2p + y^2q 
\] 
for some $r\in\ringtwo_{\geq 1}$, $q\in\C\llsq y\rrsq_{\geq 1}$ and $p\in\C\llsq y\rrsq$. Consequently $u\equiv x^2(\uplambda+r) + xy^2p + y^2q$,
and further using the relation $\updelta_xf$, it follows that
\[
u\equiv x^2\big(\,\uplambda + r - (\uplambda_2x^{2n-2}+\uplambda_3x^{2m-3})p - (\uplambda_2x^{2n-3}+\uplambda_3x^{2m-4})q\,\big)
\]
where again $\uplambda_2=(2n)\upvarepsilon_2$ and $\uplambda_3=(2m-1)\upvarepsilon_3$.  The term in brackets is a unit: since $n,m\geq 2$ its only degree zero term is $\uplambda$, which by assumption is non-zero.  It follows that $\Jac(f)/(u)\cong \Jac(f)/(x^2)$, and so the result follows by \eqref{central element main 1}.
\end{proof}

\begin{remark}\label{rem: need generic}
The need for taking a generic, or at least well chosen, central element in \ref{central element main} is essential.  Indeed, by \ref{lem: E6 central}, for $f$ of Type~$E_{6,n}$ the element  $x^2$ is central in $\Jac(f)$. However
\begin{equation}
\Jac(f)/(x^2)
\cong\frac{\ringtwo}{\lcl x^2, y^3,  (x+y)^3-xyx\rcl}
\ncong\frac{\ringtwo}{\lcl x^2, y^3, (x+y)^3\rcl}
= e\Uppi e\label{eqn:noniso dim 12}
\end{equation}
even although both sides have dimension twelve.  Write $\scrU$ for the open set given by the non-vanishing of the co-efficient of $x^2$, then \eqref{eqn:noniso dim 12} together with \ref{central element main}\eqref{central element main 1} assert that the isomorphism class of $\Jac(f)/(u)$ is not constant along $u\in\scrU$.  Consequently, for Type $E$ a smaller generic open set is required. 
\end{remark}

\begin{remark}
In the proof of \ref{central element main}\eqref{central element main 1} above, the inverse of $\upphi\colon A\to B$ is not the constructed map $B\twoheadrightarrow A$, rather $\upphi^{-1}$ is induced by the much more non-obvious automorphism
\begin{align*}
x&\mapsto x + \tfrac{1}2(xy+yx) - \tfrac{1}8yxy \\
y&\mapsto \tfrac{1}2(-x+y) + \tfrac{1}8(x^2-3xy-yx+y^2) + {} \\
 & \hspace{3cm} \tfrac{1}{64}(5yx^2+12yxy+16y^2x) +\tfrac{3}{64}y^2xy + \tfrac{7}{128}y^2x^2.
\end{align*}
\end{remark}

%-----------------------------------------------------------------------------------------
\section{Geometric Corollaries}\label{geo section}

The previous results have geometric consequences.  Section \ref{classify flops sect} classifies contraction algebras, up to isomorphism, from all Type~$A$ and~$D$ flopping contractions.  This immediately gives, in \S\ref{sec:classflopsfinally}, a classification of Type $A$  and $D$ flops, and it also has consequences to GV invariants. Then \S\ref{sec: construct D div-to-curve} constructs the first, and conjecturally only, infinite family of Type~$D$ divisor-to-curve contractions.  Using this, and known results from flops, we then prove that the Realisation Conjecture (\ref{realisation conj}) is true, except possibly for some exceptional cases, establishing \ref{realisation main intro} in the introduction.  The last subsection classifies contraction algebras that can arise from Type~$A$ and Type~$D_4$ divisor-to-curve contractions.

\subsection{Classification of contraction algebras for A and D flops}\label{classify flops sect}
In this section we classify contraction algebras that can arise from Type~$A$ and~$D$ flops, in both cases without referring to any classification of such flops (noting that a Type~$D$ flop classification arises as a consequence, in \S\ref{sec:classflopsfinally}).  We will show that the only possible options are those finite dimensional Jacobi algebras in \ref{thm!pagoda} and \ref{main Type D} respectively.
We use the notation of~\S\ref{geo cor section intro} freely; in particular,
$\scrR$ is commutative noetherian and in applications $\Spec\scrR$ is the base of a simple 3-fold flop.

\begin{remark}
In both \ref{the only Type A CAs} and \ref{the only Type D CAs} below we classify the contraction algebras within a given type, but in fact more is true.  By \cite[4.6]{HuaToda} if $\mathrm{B}_{\con}$ is the contraction algebra of a Type~$X$ flopping contraction which is isomorphic to a contraction algebra $\CA$ of a Type~$Y$ flopping contraction, then $X=Y$. Hence, the algebras in \ref{the only Type A CAs} below are the contraction algebras of all possible Type~$A$ flops, and only of Type~$A$ flops, and the algebras in \ref{the only Type D CAs} are the contraction algebras of all possible Type~$D$ flops, and only of Type~$D$ flops.
\end{remark}

The classification in Type~$A$ is elementary.
\begin{prop}\label{the only Type A CAs}
If $\CA$ is a contraction algebra of a Type~$A$ flopping contraction, then $\CA\cong\Jac(x^2+y^n)$ for some $n\geq 2$.   Furthermore, any other contraction algebra of any other type cannot be isomorphic to such a Jacobi algebra. 
\end{prop}
\begin{proof}
Consider $\CA$ from an arbitrary Type~$A$ flop $\scrX\to\Spec\scrR$. By a now standard argument of Van den Bergh \cite[A.1]{VdB1d}, any indecomposable CM $\scrR$-module necessarily has rank one.  Further, since $\scrR$ is normal the endomorphism ring of any rank one CM module is isomorphic to $\scrR$.  Hence $\CA$, being defined to be a factor of $\End_\scrR(N)$ for some indecomposable CM $\scrR$-module $N$, is thus a factor of $\scrR$, and hence is commutative.

Since $\CA\cong\Jac(f)$ for some $f$, combining \ref{thm!pagoda} and \ref{when comm} we see that $\CA\cong\Jac(x^2+y^n)$ or $\CA\cong\Jac(x^2)$.  The last case is impossible, since $\dim_\C\CA<\infty$ given that $\scrX\to\Spec\scrR$ is a flop \cite{DW3}.
\end{proof}

\begin{remark}\label{Miles remark 1}
A more geometric proof of \ref{the only Type A CAs} 
uses Reid's Pagoda classification of Type~$A$ flops~\cite{Pagoda} and then applies \cite[3.10]{DW1} to conclude that $\CA\cong\Jac(x^2+y^n)$ for some $n$.  %The proof of \ref{the only Type A CAs} did not use Reid's classification, partly to show that it is not necessary, but mainly because in Type~$D$ below such a classification is not available. 
\end{remark}

Type~$D$ is much more involved, and requires multiple preliminary results. In the following, by an $\scrR$-algebra $\Upgamma$ we simply mean that there exists a homomorphism $\scrR\to Z(\Upgamma)$, where $Z(\Upgamma)$ is the centre. From \ref{radical remark}, let $\mathfrak{J}(\Upgamma)$ denote the Jacobson radical of $\Upgamma$.

\begin{lemma}\label{lovely completion fact}
Let $\Upgamma$ be an $\scrR$-algebra, where $(\scrR,\m)$ is commutative local noetherian, and suppose that $\Upgamma$ is finitely generated as an $\scrR$-module.
\begin{enumerate}
\item\label{lovely1}
The $\mathfrak{J}(\Upgamma)$-adic topology coincides with the $\m$-adic topology on $\Upgamma$.
\item\label{lovely2}
Every ideal of $\Upgamma$ is closed with respect to the $\mathfrak{J}(\Upgamma)$-adic and the $\m$-adic topologies.
\item\label{lovely3}
If $(\scrR,\m)$ is complete local, then $\Upgamma$ is complete with respect to both the $\mathfrak{J}(\Upgamma)$-adic and the $\m$-adic topologies. 
\end{enumerate}
\end{lemma}
\begin{proof}
(1) Since $(\scrR,\m)$ is local, and $\Upgamma$ is finitely generated as an $\scrR$-module, it follows immediately from e.g.\ \cite[20.6]{Lam} that there exists $n\geq 1$ such that
\begin{equation}
\mathfrak{J}(\Upgamma)^n\subseteq \m\Upgamma\subseteq \mathfrak{J}(\Upgamma).\label{eq: topologies equal}
\end{equation}
From this, it is clear that the $\mathfrak{J}(\Upgamma)$-adic and the $\m$-adic topologies coincide.

\noindent
(2) We show that for any finitely generated $\scrR$-module $M$, with submodule $N$, then $N$ is closed in $M$. Given this, applying~\eqref{lovely1} to $M=\Upgamma$ and $N=I$ proves the result.  But since $M$ is finitely generated, and $\scrR$ is noetherian, Krull's intersection theorem \cite[8.10(1)]{Matsumura} immediately shows that $M/N$ is separated, and hence $N$ is closed in $M$. 

\noindent
(3) Again, this is well known. Since $\Upgamma$ is a finitely generated $\scrR$-module, and $(\scrR,\m)$ is complete local and noetherian, it follows from e.g.\ \cite[21.33]{Lam} that $\Upgamma$ is $\m$-adically complete, and hence the result follows by~\eqref{lovely1}.
\end{proof}

\begin{cor}\label{cor: Acon has Jdim leq 1}
If $\CA$ is a contraction algebra associated to a crepant $\scrX\to\Spec\scrR$ as above, then $\JRdim\CA=0,1$.  Furthermore, the following statements hold.
\begin{enumerate}
\item   $\JRdim\CA=0$ if and only if $\scrX\to\Spec\scrR$ is a flop. 
\item $\JRdim\CA=1$ if and only if $\scrX\to\Spec\scrR$ is a divisorial contraction to a curve. 
\end{enumerate} 
\end{cor}
\begin{proof}
$\CA$ is module finite over $\scrR$, being a factor of an NCCR \cite{DW1}. 

Now if $M$ is any finitely generated $\scrR$-module (e.g.\ $M=\CA$), since $(\scrR,\m)$ is local $\dim_\scrR(M)$ can be defined using the $\m$-adic topology, as the growth rate of the function $\mathrm{length}(M/\m^iM)$, see e.g.\  \cite[\S12.1]{Eisenbud}.  Taking suitable powers of the inclusions in \eqref{eq: topologies equal}, it is elementary to see that the two sets
\begin{align*}
\scrS_1&=\left\{ r\in\R \mid 
\textrm{for some $c\in\R$, }\dim_\C(\CA / \mathfrak{J}^n) \le cn^r \text{ for every $n\in\N$} \right\}\\
\scrS_2&=\left\{ r\in\R \mid 
\textrm{for some $c\in\R$, }\dim_\C(\CA / \m^n\CA) \le cn^r \text{ for every $n\in\N$} \right\}
\end{align*} 
are equal, so $\JRdim\CA=\mathrm{inf}\,\scrS_1=\mathrm{inf}\,\scrS_2=\dim_\scrR(\CA)$.  

The main result of \cite{DW2} shows that $\Supp_\scrR(\CA)$ equals the contracted locus in $\Spec \scrR$.  Hence $\dim_\scrR(\CA)$ is either $0$ for flops, or $1$ for divisor-to-curves respectively.  It follows  that $\JRdim(\CA)$ is either 0 or 1, respectively.   
\end{proof}

\begin{thm}\label{central in n2}
If $f\in\C\llangle x,y\rrangle_{\geq 3}$ and $\Jac(f)$ is complete in its $\mathfrak{J}$-adic topology, then every central element $g+\lcl\updelta_xf,\updelta_yf\rcl\in\Jac(f)$  which is not a unit satisfies $g\in\n^2$.
\end{thm}
\begin{proof}
Set $R=\Jac(f)$, and $\mathfrak{J}=\mathfrak{J}(R)$.  Note that $I=\lcl \updelta f \rcl\subseteq\n^2$, hence by \ref{radical remark}\eqref{radical remark 1}
\begin{equation}
\mathfrak{J}^2=(\n^2+I)/I=\n^2/I. \label{rad2 statement}
\end{equation}
Consider the central element $g'=g+I$ in $R$. Since $g'$ is not a unit in $R$, certainly $g'\in\mathfrak{J}$. Further $g$  cannot be a unit in $\ring$, or it would descend to a unit, so $g\in\n$.

We now suppose that $g\notin\n^2$, and aim for a contradiction.  Since $I\subseteq \n^2$ we see that $g'=g+I\notin\mathfrak{J}^2$, and hence $0\neq g'+\mathfrak{J}^2\in\mathfrak{J}/\mathfrak{J}^2$. But by \eqref{rad2 statement} we have
\[
\frac{\mathfrak{J}}{\mathfrak{J}^2}
=
\frac{\n/I}{\n^2/I}
\cong
\frac{\n}{\n^2}
\]
and so $\dim_\C\mathfrak{J}/\mathfrak{J}^2=2$. Pick $0\neq h+\mathfrak{J}^2$ to complete $g'+\mathfrak{J}^2$ to a basis of $\mathfrak{J}/\mathfrak{J}^2$.

But now since by assumption $R$ is complete, and local, we may use \cite[3.1]{BIRS} to present $R$.  Consider the two-loop quiver $Q$, and map the trivial path to the identity of $R$, one of the loops $\ell_1$ to $h$ and the other loop $\ell_2$ to $g'$.  Then by \cite[3.1]{BIRS} the completeness of $R$ extends this to a surjective homomorphism  
\[
\upvarphi\colon\C\llangle \ell_1,\ell_2\rrangle\twoheadrightarrow R
\]
and the kernel is a closed ideal.  Since the kernel contains the relation $\ell_1\ell_2-\ell_2\ell_1$, given $g'$ is central and so commutes with $h$, it follows that $\lcl \ell_1\ell_2-\ell_2\ell_1\rcl\subseteq\Ker\upvarphi$.  In particular $\upvarphi$ induces a surjection $\C\llsq \ell_1,\ell_2\rrsq\twoheadrightarrow R$, and so $R$ is commutative.  Given this would contradict \ref{when comm}, we conclude that $g\in\hatM^2$.
\end{proof}

We will also require the following fact.
\begin{prop}\label{cut fact}
Suppose that $\CA$ is the contraction algebra associated to a $D_4$ contraction, then there is a central element $g\in\CA$ such that 
\[
\mathrm{A}_{\con}/(g)\cong \frac{\C\langle x,y\rangle}{(x^2,xy+yx,y^2)}.
\]
\end{prop}
\begin{proof}
Consider the 3-fold contraction $\scrX\to\Spec \scrR$, and for generic $g\in\scrR$ consider the pullback diagram
\[
\begin{tikzpicture}
\node (A) at (0,0) {$\scrY$};
\node (B) at (2,0) {$\scrX$};
\node (a) at (0,-1) {$\Spec\scrR/g$};
\node (b) at (2,-1) {$\Spec\scrR$};
\draw[->] (A)--(B);
\draw[->] (A)--(a);
\draw[->] (a)--(b);
\draw[->] (B)--(b);
\end{tikzpicture}
\]
By assumption, $\scrR/g$ is a $D_4$ Kleinian singularity.   Now let $\mathrm{A}=\End_{\scrR}(N)$ be the NCCR associated to $\scrX\to\Spec \scrR$, and view $g\in\scrR=Z(\mathrm{A})\subset \mathrm{A}$.  Since $g$ is generic, we can find such a $g$ which is not contained in any associated prime ideal of $\Ext^1_\scrR(N,N)$, from which \cite[Proof of 5.24]{IW6} (the assumptions there are Type~$A$, but the method is general) shows that there is an isomorphism
\[
\mathrm{A}/g\cong\End_{\scrR/g}(N/gN).
\]
From this isomorphism \cite[(3.C)]{DW1} establishes that $\CA/g$ is isomorphic to the contraction algebra associated to the surfaces contraction $\scrY\to\Spec\scrR/g$. The fact that the surfaces contraction algebra for this particular $D_4$ contraction is $\C\langle x,y\rangle/ (x^2,xy+yx,y^2)$ can be deduced from \cite[8.7]{DW1}; see also \cite{Melit}.
\end{proof}

We obtain the following remarkable consequence.
\begin{cor}\label{others not Type D}
No $\Jac(f)$ with $f\in\C\llangle x,y\rrangle_{\geq 3}$ such that either $f_3=0$ or $f_3\cong x^3$, can arise as a contraction algebra of a $D_4$ flop, or a $D_4$ divisor-to-curve contraction.
\end{cor}
\begin{proof}
Given such an $f$, suppose that $\Jac(f)\cong\CA$ for a contraction algebra of a $D_4$ flop, or a $D_4$ divisor-to-curve contraction $\scrX\to\Spec\scrR$.  Since $\CA$ is module finite over~$\scrR$, being a factor of a NCCR \cite{DW1}, \ref{lovely completion fact} shows that $\CA$ and hence $\Jac(f)$ is complete with respect to its radical-adic topology, and further every ideal is closed.

Further, by \ref{cut fact} we can find a central $g$ such that $\dim_\C\Jac(f)/(g)=4$, and since $\Jac(f)\cong\CA$ is complete, we can use \ref{central in n2} to write $g=g'+\lcl\updelta_xf,\updelta_yf\rcl$ where $g'\in\hatM^2$.  But since all ideals in $\Jac(f)$ are closed, it follows that
\[
\frac{\Jac(f)}{(g)}=\frac{\Jac(f)}{\lcl g\rcl}
\stackrel{\scriptstyle \ref{closed ideals 101}}{\cong}\frac{\ring}{\lcl \updelta_xf, \updelta_yf,g'\rcl}
\]
has dimension four.  We claim that this is impossible, by exhibiting a factor with higher dimension.  Reusing the notation in \ref{when comm}, write $\scrM_3$ for the set of all noncommutative monomials of degree $3$, and then we will factor by $\lcl \scrM_3\rcl$.  In the two cases $f_3=x^3$ and $f_3=0$, the factors are, respectively
\[
\frac{\ring}{\lcl x^2,g',\scrM_3\rcl}\quad\mbox{and}\quad
\frac{\ring}{\lcl g',\scrM_3\rcl}.
\]
The right hand algebra surjects onto the left hand algebra, so it suffices to prove that $\dim_{\C}\ring /\lcl x^2,g',\scrM_3\rcl>4$.  But since $g'\in \n^2$ by \ref{central in n2}, inside the ideal we can replace $g'$ by $\uplambda_1 xy + \uplambda_2 yx +\uplambda_3 y^2$, which gives at most one linear relation between $xy$, $yx$ and $y^2$.  From this, the statement is clear.
\end{proof}

The above gives rise to the following, which is the main result in this subsection.

\begin{cor}\label{the only Type D CAs}
If $\CA$ is a contraction algebra of a Type~$D$ flopping contraction, then $\CA\cong\Jac(x^2y+x^{2n})$ for some $n\geq 2$, or $\CA\cong\Jac(xy^2 + x^{2n} + x^{2m-1})$ for some $m,n\geq 2$ with $m\leq 2n-1$. 
\end{cor}
\begin{proof}
Consider $\CA$ from an arbitrary Type~$D$ flop.  By \cite{KM} necessarily the elephant is~$D_4$, so  $\CA$ is not commutative by \ref{cut fact}, since $\CA$ has a factor which is not commutative.  As $\CA\cong\Jac(f)$ for some $f\in\C\llangle x,y\rrangle_{\geq 2}$, appealing to \ref{when comm} then gives $f_2=0$.  

From this, \ref{others not Type D} asserts that $f_3\neq 0$, and $f_3\ncong \ell^3$.  Hence by \ref{main Type D} $\CA\cong\Jac(f)$ for some $f$ in the list stated there.  Only the bottom three families are possible,  since $\dim_\C\CA<\infty$ given the contraction is a flop \cite{DW3}.
\end{proof}

\subsection{Classification of Flops}\label{sec:classflopsfinally}

The above results for contraction algebras classify flops, given that the Donovan--Wemyss conjecture is true \cite{DW1, August, JKM}. 

\begin{thm}\label{thm:actualclassification}
With notation as above, the following statements hold.
\begin{enumerate}
\item Type $A$ flops are classified by Type $A$ normal forms in Table~\textnormal{\ref{tab!zero}}.
\item Type $D$ flops are classified by Type $D$ normal forms in Table~\textnormal{\ref{tab!zero}}.
\end{enumerate}
Furthermore, Type $E$ flops are classified by Type $E$ normal forms.
\end{thm}
\begin{proof}
The fact that the isomorphism class of the contraction algebra classifies flops is established in \cite{JKM}.  Thus (1) follows from the fact in \ref{the only Type A CAs} that the contraction algebras of Type $A$ flops are precisely those Jacobi algebras coming from the Type $A$ normal forms.  Similarly, (2) follows from \ref{the only Type D CAs} and the fact that each normal form is realisable from geometry \cite{Okke, Kawamata} (see also \ref{thm:TypeDgeoclassdiag} below).  Since the remaining Type $E$ normal forms cannot correspond to either Type $A$ or Type $D$ geometry (by either \ref{the only Type A CAs} or \ref{others not Type D}), it follows that the contraction algebras of the remaining Type $E$ flops must be isomorphic to Jacobi algebras of Type $E$ normal forms.  Note that the Type $E$ normal forms stated in Table~\textnormal{\ref{tab!zero}} are genuine examples; full details of others will appear elsewhere \cite{BW3}.
\end{proof}

The classification of contraction algebras in \ref{the only Type D CAs} then has the following consequence.

\begin{thm}\label{thm:TypeDgeoclassdiag}
There is a one-to-one correspondence between lattice points in the diagram in \textnormal{\ref{thm:TypeDclassintro}} and the base singularities $0\in\Spec\scrR$ of Type $D$ flops, given by 
\[
(n,m)\mapsto \Spec\left( \frac{\mathbb{C}\llsq u,v,x,y\rrsq}{ u^2+v^2y - x(y^{2n+1}+(x+\upvarepsilon y^m)^2)}\right)
\] 
where $\upvarepsilon=1$ if the lattice point is contained within the shaded region, and $\upvarepsilon=0$ otherwise.

In particular, Type $D$ flops do not admit moduli.  Furthermore, the following hold.
\begin{enumerate}
\item\label{thm:TypeDgeoclassdiag 1} The quasi-homogeneous Type D flops are precisely those outside the shaded region, and these are the standard Laufer family.   
\item\label{thm:TypeDgeoclassdiag 2}  The \textnormal{GV} invariants $n_1,n_2$ of the flopping contraction associated to a point $(n,m)$ are written in green. The ovals group together flops with the same \textnormal{GV} invariants.
\end{enumerate}
\end{thm}
\begin{proof}
It is immediate from \ref{thm:actualclassification} that the Type $D$ normal forms in Table~\textnormal{\ref{tab!zero}} classify Type $D$ flops. In particular, once we exhibit one flop for each Type $D$ normal form in Table~\textnormal{\ref{tab!zero}}, which has contraction algebra isomorphic to the prescribed Jacobi algebra, then the geometric classification is complete.  

In the indexing of the diagram, for $(n,m)$ with $n,m\geq 1$, consider the corresponding potential $xy^2+x^{2n+2}+x^{2m+1}$.   Under this assignment, the points $(n,m)$ in the shaded region correspond to the normal forms $D_{n+1,m+1}$, which by definition have a restriction on $m$ relative to $n$.  The lattice points not in the shaded region, namely those $(n,2n+1)$, correspond to $xy^2+x^{2n+2}+x^{4n+3}$. By \ref{CorDtwo} this is isomorphic to $xy^2+x^{2n+2}$, which is the normal form $D_{n+1,\infty}$ in Table~\textnormal{\ref{tab!zero}}.  Thus, the lattice points stated are in bijection with the Type $D$ normal forms.

Now by \cite[\S2.2]{Okke} and \cite{Kawamata}, given any potential $xy^2+x^{2n+2}+x^{2m+1}$, the claimed commutative ring in the statement is the base of a flopping contraction, and further the corresponding contraction algebra is isomorphic to $\Jac(xy^2+x^{2n+2}+x^{2m+1})$. The first statement regarding the bijection follows.

The fact that the quasi-homogeneous singularities correspond to those outside the shaded region follows from \cite[\S2.2.4]{OkkeThesis}, which computes the Milnor and Tjurina numbers.   For the final statement regarding GV invariants, by \ref{main Type D}\eqref{main Type D B} and Toda's dimension formula \cite[1.1]{TodaGV}, we can read off the GV invariants. Indeed, by \cite{TodaGV} the pair $n_1,n_2$ where $n_1=\dim_{\mathbb{C}}\Jac(f)^{\ab}$ and $n_2=\tfrac{1}{4}(\dim_{\mathbb{C}}\Jac(f)-\dim_{\mathbb{C}}\Jac(f)^{\ab}))$ are precisely the GV invariants for length two flops.  Only those pairs illustrated in the diagram in \textnormal{\ref{thm:TypeDclassintro}} (or Figure~\ref{fig!Ddiagram}) appear. 
\end{proof}

\begin{remark}
The map in \ref{thm:TypeDgeoclassdiag} is in fact well defined on all points $(n,m)$ with $m,n\geq 1$, not just  those marked in the picture in \ref{thm:TypeDclassintro}.  This follows since the commutative ring in \ref{thm:TypeDgeoclassdiag} is always the base of a Type $D$ flop, for every $(n,m)$ \cite[\S2.2]{Okke}. The point is that the domain needs to be restricted in order to obtain a bijection.  The previous result \ref{CorDtwo} shows that any lattice point $(n,m)$ with $m\geq 2n+1$ gives an isomorphic flop to the lattice point $(n,2n+1)$.  
\end{remark}

Simply inspecting \ref{thm:TypeDgeoclassdiag}\eqref{thm:TypeDgeoclassdiag 2} and the diagram in \ref{thm:TypeDclassintro} gives the following corollary, which illustrates the significant gaps in the possible GV invariants that can arise.
\begin{cor}\label{GV for D}
Consider $(a,b)\in\mathbb{N}^2$.  Then $a,b$ are the GV invariants for a Type~$D$ flopping contraction if and only if either
\begin{enumerate}
\item $(a,b)=(2m+3,m)$ for some $m\geq 1$, or
\item $(a,b)=(2n,b)$ for some $n\geq 2$, with $b\geq n-1$. 
\end{enumerate} 
Further, when $a=2m+3$ there are precisely $m+1$ distinct contraction algebras realising $(a,b)$, up to isomorphism, whilst for any given $(2n,b)$ the contraction algebra is unique.
\end{cor}

\begin{remark}
It is possible to instead index the GV invariants to the classifying potentials, and this is done in Figure~\ref{fig!Ddiagram} below.
\begin{figure}[ht]
\begin{tikzpicture}[xscale=1.4,yscale=1]

\draw[->] (-0.8,0.2) -- (5.75,0.2);

\coordinate (m1) at (-1.75,0.2);
\coordinate (0) at (0.5,-1.3);
\coordinate (1) at (-2.5,0.2);
\coordinate (2) at (1.5,-2.8);
\coordinate (3) at (-3.25,0.2);
\coordinate (4) at (2.5,-4.3);
\coordinate (5) at (-4,0.2);
\coordinate (6) at (3.5,-5.8);
\coordinate (7) at (-4.75,0.2);
\coordinate (8) at (4.5,-7.3);
\coordinate (d1) at (-0.8,0.2);
\coordinate (d2) at (-0.8,-5);

\coordinate (c1) at (0,-4.7);
\coordinate (c2) at (6,-4.7);

\draw[->,rounded corners] (intersection cs: 
first line={(m1)--(0)}, 
second line={(d1)--(d2)}) -- (0.5,-1.3) -- (5.75,-1.3);

\draw[->,rounded corners] (intersection cs: 
first line={(1)--(2)}, 
second line={(d1)--(d2)}) -- (1.5,-2.8) -- (5.75,-2.8);
\draw[->,rounded corners]  (intersection cs: 
first line={(3)--(4)}, 
second line={(d1)--(d2)}) --(2.5,-4.3) -- (5.75,-4.3);
\draw[-]  (intersection cs: 
first line={(5)--(6)}, 
second line={(d1)--(d2)}) -- (intersection cs: 
first line={(5)--(6)}, 
second line={(c1)--(c2)});
\draw[-]  (intersection cs: 
first line={(7)--(8)}, 
second line={(d1)--(d2)}) -- (intersection cs: 
first line={(7)--(8)}, 
second line={(c1)--(c2)});

\node at (0.1,0.4) {$\scriptstyle x^3+x^4$};
\node at (1.1,0.4) {$\scriptstyle x^3+x^6$};
\node at (2.1,0.4) {$\scriptstyle x^3+x^8$};
\node at (3.1,0.4) {$\scriptstyle x^3+x^{10}$};
\node at (4.1,0.4) {$\scriptstyle x^3+x^{12}$};
\node at (5.1,0.4) {$\scriptstyle x^3+x^{14}$};

\filldraw[white] (-0.8,-0.85) rectangle (0.3,-0.55);
\node at (-0.2,-0.7) {$\scriptstyle x^5 + x^4,$} ;
\node[color=green!50!black] at (0.4,-0.68) {$\scriptstyle x^4$};

\node at (1.1,-1.12) {$\scriptstyle x^5+x^6$};
\node at (2.1,-1.12) {$\scriptstyle x^5+x^8$};
\node at (3.1,-1.12) {$\scriptstyle x^5+x^{10}$};
\node at (4.1,-1.12) {$\scriptstyle x^5+x^{12}$};
\node at (5.1,-1.12) {$\scriptstyle x^5+x^{14}$};

\filldraw[white] (-0.4,-2.35) rectangle (1.3,-2.05);
\node at (-0.2,-2.2) {$\scriptstyle x^9 + x^6,$}; 
\node at (0.7,-2.2) {$\scriptstyle x^7 + x^6,$} ;
\node[color=green!50!black] at (1.3,-2.18) {$\scriptstyle x^6$};

\node at (2.1,-2.62) {$\scriptstyle x^7+x^8$};
\node at (3.1,-2.62) {$\scriptstyle x^7+x^{10}$};
\node at (4.1,-2.62) {$\scriptstyle x^7+x^{12}$};
\node at (5.1,-2.62) {$\scriptstyle x^7+x^{14}$};

\filldraw[white] (-0.4,-3.85) rectangle (2.3,-3.55);
\node at (-0.1,-3.7) {$\scriptstyle x^{13} + x^8,$} ;
\node at (0.8,-3.7) {$\scriptstyle x^{11} + x^8,$} ;
\node at (1.7,-3.7) {$\scriptstyle x^9 + x^8,$} ;
\node[color=green!50!black] at (2.3,-3.68) {$\scriptstyle x^8$} ;

\node at (3.1,-4.12) {$\scriptstyle x^9+x^{10}$};
\node at (4.1,-4.12) {$\scriptstyle x^9+x^{12}$};
\node at (5.1,-4.12) {$\scriptstyle x^9+x^{14}$};

%%%%%%%%%%%%
%% GV on top
%%%%%%%%%%%%
\node at (0,0) {$\scriptstyle 4,1$};
\node at (1,0) {$\scriptstyle 4,2$};
\node at (2,0) {$\scriptstyle 4,3$};
\node at (3,0) {$\scriptstyle 4,4$};
\node at (4,0) {$\scriptstyle 4,5$};
\node at (5,0) {$\scriptstyle 4,6$};

\filldraw[white] (-0.2,-1.15) rectangle (0.2,-0.85);
\node at (0,-1) {$\scriptstyle 5,1$};

\node at (1,-1.5) {$\scriptstyle 6,2$};
\node at (2,-1.5) {$\scriptstyle 6,3$};
\node at (3,-1.5) {$\scriptstyle 6,4$};
\node at (4,-1.5) {$\scriptstyle 6,5$};
\node at (5,-1.5) {$\scriptstyle 6,6$};

\filldraw[white] (0.8,-2.65) rectangle (1.2,-2.35);
\node at (1,-2.5) {$\scriptstyle 7,2$};

\node at (2,-3) {$\scriptstyle 8,3$};
\node at (3,-3) {$\scriptstyle 8,4$};
\node at (4,-3) {$\scriptstyle 8,5$};
\node at (5,-3) {$\scriptstyle 8,6$};

\filldraw[white] (1.8,-4.15) rectangle (2.2,-3.85);
\node at (2,-4) {$\scriptstyle 9,3$};

\node at (3,-4.5) {$\scriptstyle 10,4$};
\node at (4,-4.5) {$\scriptstyle 10,5$};
\node at (5,-4.5) {$\scriptstyle 10,6$};

\end{tikzpicture}
\caption{List of $p(x)$ for which $xy^2+p(x)$ is one
of the normal forms in $D_{n,m}$ or $D_{n,\infty}$. The pair $n_1,n_2$
associated to each $p(x)$ describes the GV invariants of any
simple flop having isomorphic contraction algebra.\label{fig!Ddiagram}}
\end{figure}
\end{remark}

\subsection{Constructing divisor to curve contractions}\label{sec: construct D div-to-curve}
In the list of potentials in \ref{main Type D}, the first appears as the contraction algebra of a divisor-to-curve contraction in \cite[2.18]{DW4}.  The second family, with $m=1$, is isomorphic to $x^3+y^3$, and so appears as a contraction algebra in \cite[2.25]{DW4}.  All the other three families are contraction algebras of $D_4$ flops by \cite{Okke, Kawamata}, and the above subsection.  

Motivated by Conjecture~\ref{realisation conj}, this subsection will fill the last remaining gap, and show that the whole of the second family in \ref{main Type D}, with arbitrary $m$, are realised as the contraction algebra of a divisor-to-curve contraction.  

\begin{remark}
In the proof below we will first construct the contraction algebraically, before passing to the formal fibre to realise the contraction algebra.  This algebraic construction is advantageous, since it conceptually distinguishes between the cases: in $\Spec R_\infty$ below, which locally realises $D_{\infty,\infty}$, the origin is $cD_4$ whilst all other points on the singular locus are $cA_2$.  In contrast, in $\Spec R_m$ below, which locally realises $D_{\infty,m}$, the origin is $cD_4$ whilst all other points on the singular locus are $cA_1$.   Compare the pictures in \cite[2.18]{DW4} and \cite[2.25]{DW4}, and also \cite{Wilson}.
\end{remark}

\begin{prop}\label{D4 div to curve main}
For $m\in\mathbb{N}\cup\{\infty\}$, consider the element of $\mathbb{C}\llsq X,Y,Z,T\rrsq$ defined by 
\[
F_m\colonequals
\left\{
\begin{array}{rl}
 Y(X^{m}+Y)^2 + XZ^2 -T^2& \mbox{if }m\geq 1\\
 Y^3 + XZ^2 -T^2& \mbox{if }m=\infty\\
 \end{array}
 \right.
\]
and set $\scrR_m=\mathbb{C}\llsq X,Y,Z,T\rrsq/F_m$. Then the following statements hold.
\begin{enumerate}
\item\label{D4-singlocus}
$\Sing(\scrR_m)^{\mathrm{red}}=(X^m+Y,Z,T)$ if $m\geq 1$, and $(Y,Z,T)$ if $m=\infty$.
\item\label{D4-blowup}
In either case, blowing up this locus gives rise to a crepant Type~$D$ divisorial contraction to a curve $\scrX_m\to\Spec \scrR_m$ where $\scrX_m$ is smooth.
\item\label{D4-potential}
The contraction algebra of $\scrX_m\to\Spec \scrR_m$ is isomorphic to $\Jac(xy^2+x^{2m+1})$ when $m\geq 1$, respectively $\Jac(xy^2)$ when $m=\infty$.
\item\label{D4-distinct}
$\scrR_m\cong\scrR_n$ if and only if $m=n$, and so the $F_m$ are all distinct up to isomorphism and so form an infinite family. 
\end{enumerate}
\end{prop}

\begin{proof}
(1) is immediate.\\
(2) Working first on the case of finite $m\ge1$, consider the affine algebra
\[
R_m = \frac{\C[r,s,u,v]}{u^2 - r(r-s^m)^2 - sv^2}
\]
whose completion at the origin is~$\scrR_m$, in coordinates $(r,s,u,v) = (-Y,X,T,Z)$.
The blowup along $(u,v,r-s^m)$ is covered by two affine patches: the first is $U = \Spec\C[s,y_0,y_1]$, with $y_0=u/(r-s^m)$ and $y_1=v/(r-s^m)$, the second chart is a smooth hypersurface, and the map from $U$ to the base is given by
\begin{align*}
(s,y_0,y_1)&\mapsto (y_0^2-sy_1^2, \,s, \,y_0(y_0^2-sy_1^2-s^m), \,y_1(y_0^2-sy_1^2-s^m)).
\end{align*}
The exceptional locus in $U$ is the divisor $y_0^2-sy_1^2-s^m=0$.
Pulling the canonical basis of differentials on $\Spec R_m$ back to $U$ gives
\[
f_2^*\left(\frac{dr\wedge ds \wedge dv}{u}\right) 
= 
\frac{d(y_0^2-sy_1^2)\wedge ds\wedge d(y_1(y_0^2-sy_1^2-s^m))}{y_0(y_0^2-sy_1^2-s^m)}=2 ds\wedge dy_0\wedge  dy_1
\]
which is a regular differential on~$U$, and in particular has no zero or pole along the exceptional divisor.
%which has no zero or pole along $E_2=(y_0^2-sy_1^2-s^m)$.
Thus the map is crepant, as claimed.  The case $m=\infty$ is similar, but easier, as both open charts are affine $3$-space.\\
(3) The easiest way to establish the claim is to recognise $F_m$ as a pullback from the universal $D_4$ flop and apply restriction theorems for contraction algebras. Consider the six-dimensional universal $D_4$ flop, given in \cite[(1.1)]{Karmazyn} as
\[
\mathcal{R} = \frac{\C[r,s,t,u,v,w,z]}{u^2 - rw^2 + 2zvw - sv^2 + (rs-z^2)t^2}
\]
and its universal family  $\scrY \longrightarrow \Spec\scrR$, which is an isomorphism away from the singular locus in $\Spec \mathcal{R}$. As observed by Van Garderen \cite[\S2.2.3]{OkkeThesis}, slicing by the sequence $g_1 = z$, $g_2 = r - w - s^m$, and $g_3 = t$ yields a commutative diagram
\[
\begin{tikzpicture}[yscale=1.25]
\node (A) at (0,0) {$Y = Y_3$};
\node (B) at (2,0) {$Y_2$};
\node (C) at (4,0) {$Y_1$};
\node (D) at (6,0) {$\scrY$};
\node (a) at (0,-1) {$\Spec\mathcal{R}_3$};
\node (b) at (2,-1) {$\Spec\mathcal{R}_2$};
\node (c) at (4,-1) {$\Spec\mathcal{R}_1$};
\node (d) at (6,-1) {$\Spec\mathcal{R}$};
\draw[->] (A)--(B);
\draw[->] (B)--(C);
\draw[->] (C)--(D);
\draw[->] (a)--(b);
\draw[->] (b)--(c);
\draw[->] (c)--(d);
\draw[->] (A) --node[left] {$\scriptstyle f$} (a);
\draw[->] (B) -- (b);
\draw[->] (C) -- (c);
\draw[->] (D) -- (d);
\end{tikzpicture}
\]
where $\mathcal{R}_1=\mathcal{R}/g_1$, $\mathcal{R}_2=\mathcal{R}_1/g_2$ and $\mathcal{R}_3=\mathcal{R}_2/g_3$.
The result is the affine algebra
\[
\mathcal{R}_3 = \frac{\C[r,s,u,v]}{u^2 - r(r-s^m)^2 - sv^2}
\]
whose completion at the origin is~$\scrR_m$. The pullback
$f\colon Y\rightarrow \Spec\mathcal{R}_3$ is visibly an isomorphism away from $\Sing (\mathcal{R}_3)^{\redu}=(u,v,r-s^m)$,
so in particular is birational.

Van Garderen observes that $\mathcal{R}_3$ is an integral domain \cite[2.12]{OkkeThesis} and $Y$ is smooth \cite[2.13]{OkkeThesis},
and that each $g_i$ is a {\em slice}, in the terminology of \cite[2.9]{OkkeThesis}, which
implies that $f$ is projective and surjective with $\mathbf{R}f_*\scrO=\scrO$ \cite[2.10]{OkkeThesis}.  

Furthermore, the tilting bundle yielding a derived equivalence between $\scrY$ and $\Lambda\in\CM\mathcal{R}$ restricts to give a derived equivalence between $Y$ and $\Lambda\otimes \mathcal{R}_3$ \cite[(2.11)]{OkkeThesis}. Since $g_1,g_2,g_3$ is a regular sequence, $\Lambda\otimes \mathcal{R}_3\in\CM \mathcal{R}_3$ and so in particular $f$ is crepant \cite[4.14]{IW_Qfact}. Since visibly both the blowup in \eqref{D4-blowup} and $f$ are crepant resolutions of the same variety, and both containing no flopping curves, they must be isomorphic (as varieties over the base $\Spec R_m$) since minimal models are unique up to flop. Thus the contraction algebra associated to \eqref{D4-blowup} is isomorphic to the contraction algebra of the formal fibre of $f$. But by \cite[2.8]{OkkeThesis} this is the claimed Jacobi algebra, namely $\Jac(xy^2+x^{2m+1})$ when $m\geq 1$, respectively $\Jac(xy^2)$ when $m=\infty$.\\
(4) If  $\scrR_m\cong\scrR_n$ are isomorphic, then the contraction algebras of $\scrX_m\to\Spec\scrR_m$ and of $\scrX_n\to\Spec\scrR_n$ must be isomorphic. But then their abelianisations must also be isomorphic, and so in particular must have the same dimension.  But the abelianisations have dimension $2(m+1)$ and $2(n+1)$ respectively, and hence $m=n$. 
\end{proof}

The start of this subsection, combined with \ref{D4 div to curve main}, then gives the following.
\begin{cor}\label{all type D are geometric}
All the potentials in \textnormal{\ref{main Type D}} are geometric.
\end{cor}

In turn, this establishes \ref{realisation main intro} in the introduction.
\begin{cor}\label{realisation main}
Conjecture~\textnormal{\ref{realisation conj}} is true, except for the one remaining unresolved case when $f\cong x^3+\scrO_4$, where some further analysis is required.
\end{cor}
\begin{proof}
Every $\CA\cong\Jac(f)$ for some $f\in\C\llangle x,y\rrangle_{\geq 2}$. If $f_2\neq 0$ the result is \ref{Type A all geometric}, so we can assume $f_2=0$. We need $f_3\neq 0$ so that $\JRdim\Jac(f)\leq 1$, and \ref{linear change} splits into three cases.  The first two cases are covered by \ref{main Type D}, and \ref{all type D are geometric} asserts that these are all geometric.  The only remaining, unresolved, case from \ref{linear change}  is when $f_3\cong x^3$.
\end{proof}

\begin{remark} It is possible to change variables to see that all type $D$ normal forms can be realised in a uniform way.  Indeed, the contraction algebra associated to the $cD_4$ singularity
\[
\frac{\mathbb{C}\llsq u,v,x,y\rrsq}{u^2+v^2y-x(\upvarepsilon_2y^{2n+1}-(x-\upvarepsilon_3y^m)^2)}
\]
realises the general Type $D$ potential $f=xy^2+\upvarepsilon_2x^{2n}+\upvarepsilon_3 x^{2m-1}$, with the convention that each $\upvarepsilon_i$ is either $0$ or $1$.
\end{remark}

\begin{remark}
Much like in Pagoda \cite{Pagoda} for Type $A$, it is also possible to view each Type $D$ divisor-to-curve contraction as an infinite limit of flops.  The Type $D$ situation is more delicate, since there are more possible directions in which to take such a limit.  In relation to the classification of Type $D$ flops in \ref{thm:TypeDgeoclassdiag}, the following are the limits which give rise to the divisor-to-curve contractions in \ref{D4 div to curve main}.
\begin{center}
\begin{tikzpicture}[xscale=0.6,yscale=0.6]
\filldraw[color=gray!50!white] (0.5,1.6)--(3.5,7.5)--(8.6,7.5)--(8.6,0.5)--(0.5,0.5)--cycle;
\draw[->] (1,1)--(9,1);
\draw[->] (1,2)--(9,2);
\draw[->] (2,3)--(9,3);
\draw[->] (2,4)--(9,4);
\draw[->] (3,5)--(9,5);
\draw[->] (3,6)--(9,6);
\draw[->] (4,7)--(9,7);

\draw[->] (1,3)--(3.5,8);

%\node at (10,1) {$\scriptstyle xy^2+x^3$};
%\node at (10,2) {$\scriptstyle xy^2+x^5$};
%\node at (10,3) {$\scriptstyle xy^2+x^7$};
%\node at (10,4) {$\scriptstyle xy^2+x^{9}$};
%\node at (10,5) {$\scriptstyle xy^2+x^{11}$};
%\node at (10,6) {$\scriptstyle xy^2+x^{13}$};
%\node at (10,7) {$\scriptstyle xy^2+x^{15}$};
\node at (9.5,1) {$\scriptstyle F_1$};
\node at (9.5,2) {$\scriptstyle F_2$};
\node at (9.5,3) {$\scriptstyle F_3$};
\node at (9.5,4) {$\scriptstyle F_4$};
\node at (9.5,5) {$\scriptstyle F_5$};
\node at (9.5,6) {$\scriptstyle F_6$};
\node at (9.5,7) {$\scriptstyle F_7$};

\node at (3.75,8.25) {$\scriptstyle F_\infty$};
%\node at (3.75,8.25) {$\scriptstyle xy^2$};

\def\coords{1,...,8}
\foreach \y in {1,...,3}
\foreach \x in \coords \node at (\x,\y) {$\scriptstyle \bullet$};
\foreach \x in {2,...,8} \node at (\x,4) {$\scriptstyle \bullet$};
\foreach \x in {2,...,8} \node at (\x,5) {$\scriptstyle \bullet$};
\foreach \x in {3,...,8} \node at (\x,6) {$\scriptstyle \bullet$};
\foreach \x in {3,...,8} \node at (\x,7) {$\scriptstyle \bullet$};

\node at (9,3.5) {$\hdots$};
\node at (7,8) {$\vdots$};
\end{tikzpicture}
\end{center}
Without the normal forms from noncommutative singularity theory, it is hard to either see or predict the above purely geometrically.
\end{remark}

\subsection{Classification for A and D divisor-to-curve contractions}
This section is the divisor-to-curve analogue of \S\ref{classify flops sect}.

\begin{prop}\label{the only Type A CAs d2c}
If $\CA$ is a contraction algebra of a Type~$A$ divisor-to-curve contraction, then $\CA\cong\Jac(x^2)$. 
\end{prop}
\begin{proof} 
For the same homological reason as in \ref{the only Type A CAs}, in Type~$A$ the contraction algebra $\CA$ is necessarily commutative.  Since the contraction is divisor-to-curve, necessarily $\dim_\C\CA=\infty$ \cite{DW3}.  Combining \ref{when comm} and \ref{thm!pagoda}, we see that $\CA\cong\Jac(x^2)$, since $x^2$ is the only infinite dimensional example in \ref{thm!pagoda}.
\end{proof}

The following is the analogue of \ref{the only Type D CAs}. However, it is slightly weaker, due to two key geometric facts having only been developed in the flops setting: (1) Katz--Morrison \cite{KM} asserts that Type~$D$ is generically Type~$D_4$ for flops, but this is open in the divisor-to-curve setting, and (2) Hua--Toda \cite{HuaToda} asserts that the isomorphism class of a contraction algebra determines the type, but again only for flops.

However, we can say the following, without using any geometric classifications. 
\begin{prop}\label{the only Type D CAs d2c}
If $\CA$ is a contraction algebra of a Type~$D_4$ divisor-to-curve contraction, then $\CA\cong\Jac(x^2y)$, or $\CA\cong\Jac(xy^2+x^{2m+1})$ for some $m\geq 2$. 
\end{prop}
\begin{proof}
The proof is very similar to \ref{the only Type D CAs}. The algebra $\CA$ is not commutative by \ref{cut fact}, since $\CA$ has a factor which is not commutative.  As $\CA\cong\Jac(f)$ for some $f\in\C\llangle x,y\rrangle_{\geq 2}$, appealing to \ref{when comm} then gives $f_2=0$.  

From this, \ref{others not Type D} asserts that $f_3\neq 0$, and $f_3\ncong \ell^3$.  Hence by \ref{main Type D} $\CA\cong\Jac(f)$ for some $f$ in the list stated there.  Only the top two families are possible,  since $\dim_\C\CA=\infty$ given the contraction is divisor-to-curve \cite{DW3}.
\end{proof}

\appendix
\section{\texorpdfstring{$\mathfrak{J}$}{J}-dimension restrictions}\label{J appendix}
The papers \cite{ISmok,IS} introduce several new ideas that substantially strengthen the Golod--Shafarevich estimates \cite{GolodShafarevich} for the growth of algebraic Jacobi rings, and prove that almost all have cubic or higher growth.  In this appendix we extend the main results of \cite{IS}  into the setting of formal noncommutative Jacobi algebras of \ref{defin:Jacobi}, in a manner which should be viewed as the analogue of Vinberg's \cite{Vinberg} extension of Golod--Shafarevich  into the setting of topological rings.

\subsection{Algebraic Notation}\label{app:notation}
Throughout this appendix we let $d\geq 2$ and consider $\fringd=\mathbb{C}\langle x_1,\hdots,x_d\rangle$. An element $F\in\mathbb{C}\langle \mathsf{x}\rangle$ is called a \emph{superpotential} if it is cyclically symmetric, in the sense of~\ref{def:cyc symm}. For $m\geq k\geq 3$ write 
\[
\SP_{k,m}=\{ F\in\fringd\mid F \mbox{ is a superpotential with } F_j=0 \mbox{ for }j< k \mbox{ and }j>m\},
\]
where $F_j$ is the homogeneous component of $F$ of degree $j$, as in \S\ref{sect: power notation}.  In the special case $m=k$, write $\SP_{k}\colonequals \SP_{k,k}$, which consists of all homogeneous superpotentials of degree~$k$, together with zero.  Throughout, we will write elements of $\fringd$ and $\ringd$ by small letters $f$ and $g$, and superpotentials by capital letters $F$, $G$.

With the (left) strike-off derivatives $\partial_i$ defined as in \eqref{defin:strike off}, the \emph{algebraic Jacobi algebra} associated to a superpotential $F$ is the algebra
\[
\AlgJac(F)\colonequals\frac{\mathbb{C}\langle x_1,\hdots,x_d\rangle}{(\partial_{1}F,\hdots, \partial_{d}F)}
= \frac{\fringd}{I_F}
\]
where $I_F=(\partial_{1}F,\hdots \partial_{d}F)$ is the two-sided ideal generated by $\partial_1F,\hdots, \partial_{d}F$.
We write $\m=(x_1,\dots,x_d)\subset\fringd$, a maximal two-sided ideal,
and denote its image in $\AlgJac(F)$ by $\mathfrak{R}=\m/I_F$, the powers of which are
$\mathfrak{R}^i=(\m^i+I_F)/I_F$.

The use of strike-off derivatives $\partial_i$ on superpotentials, as we use here to align
with the statements and results of \cite{ISmok,IS},
or cyclic derivatives $\dcyc_i$ on any potential, as in \ref{sec:differentiation}, give equivalent theories but with minor differences in detail,
which we address in~\S\ref{sec: power series}.

\subsection{Exact Potentials and Hilbert series}\label{sec: exact pot}
%There are two main reasons why passing to superpotentials $F$ is technically convenient.
The differentiation package has two useful tools.
The first is the following Euler relation.
\begin{lemma}\label{lem:Euler}\cite[3.5]{ISmok}
If $F$ is a superpotential, then $\sum_{i=1}^d[x_i,\partial_iF]=0$.
\end{lemma}
The second is a sequence of right $\AlgJac(F)$-modules  
\begin{equation}
0\to \AlgJac(F)\xrightarrow{\mathsf{d}_3}\AlgJac(F)^{\oplus d}\xrightarrow{\mathsf{d}_2}\AlgJac(F)^{\oplus d}\xrightarrow{\mathsf{d}_1}\AlgJac(F)\xrightarrow{\mathsf{d}_0}\mathbb{C}\to 0 \label{defin:exact}
\end{equation}
defined in e.g.\ \cite[3.4]{ISmok}.  The precise form of the morphisms $\mathsf{d}_i$ will not concern us, as below we will only require the following two facts.
\begin{enumerate}
\item \cite[3.6]{ISmok} For any superpotential $F$, the sequence \eqref{defin:exact} is a complex, which is exact at the three right-most non-zero terms.
\item\label{deg of di} If further $F$ is homogeneous, say $0\neq F\in\mathsf{SP}_k$, then the morphisms in the complex \eqref{defin:exact} satisfy $\deg(\mathsf{d}_3)=1$, $\deg(\mathsf{d}_2)=k-2$, $\deg(\mathsf{d}_1)=1$, and $\deg(\mathsf{d}_0)=0$. 
%Note that $\mathsf{d}_0(f)=f_0$ is the natural augmentation map.  
\end{enumerate}

\begin{defin}
An element $F\in\SP_{k,m}$ is called \emph{exact} if \eqref{defin:exact} is exact.  
\end{defin}
If $G$ is homogeneous, then the ideal $(\partial_{1}G,\hdots, \partial_{d}G)$ is a homogeneous ideal and so the graded decomposition of $\fringd$ induces a decomposition
\begin{equation}
\AlgJac(G) =\bigoplus_{i\geq 0}\AlgJac(G)_i.\label{eq: alg Jac graded}
\end{equation}
For $G\in\mathsf{SP}_k$,
the boundary maps $\mathsf{d}_i$ are homogeneous, and so furthermore the sequence \eqref{defin:exact}
also decomposes into graded pieces, or {\em homogeneous slices}, each of which is a complex of finite-dimensional vector spaces, exact at the codomains of the restrictions of $\mathsf{d}_0$, $\mathsf{d}_1$ and $\mathsf{d}_2$.

\begin{defin}
For $G\in\mathsf{SP}_k$, \eqref{eq: alg Jac graded} determines the {\em Hilbert series}
of $\AlgJac(G)$
\[
\scrH_G = \sum_{i\ge0} \dim_\C\big(\AlgJac(G)_i\big) t^i \in \C\llsq t\rrsq.
\]
\end{defin}
\noindent
Throughout, recall that $\SP_{k}$ and $\SP_{k,m}$ are defined only when $k\geq 3$.  The following is an easy consequence of the degree of the morphisms in \eqref{deg of di} above, see e.g.\ \cite[\S3]{ISmok}. 

\begin{cor}\label{cor:hgsHilb}
If $G\in\SP_k$ is exact, then $\scrH_G = (1-dt+dt^{k-1}-t^{k})^{-1}$.
\end{cor}
%\begin{proof}
%The only homogeneous slice \eqref{eq:exactHilbseq} of \eqref{defin:exact} not considered
%in \ref{lem: exact Hilb} is when $j=0$, but that is simply 
%the identification $\AlgJac(G)_0=\C$ of the augmentation map $\mathsf{d}_0$.
%Therefore multiplying \eqref{eq:degreecount} by $t^j$ and summing over all $j\ge0$ gives
%\[
%\scrH_G - dt\scrH_G + dt^{k-1}\scrH_G - t^k\scrH_G - 1 = 0.\qedhere
%\]
%\end{proof}
%The linear independence in \ref{lem: exact Hilb}\eqref{lem: exact Hilb 1} extends to one-sided independence over $\fringd$.
%\begin{cor}\label{inject1}
%Let $G\in\SP_k$ be exact. 
%If $\sum_{i=1}^d (\partial_iG)u_i=0$ for $u_1,\dots,u_d\in\fringd$, then $u_i=0$ for all $i$.
%\end{cor}
%\begin{proof}
%Since $\fringd$ is graded, we can assume that the $u_i$ are also homogeneous, say of degree $t$.  From here is a simple induction, with the $t=0$ case being \ref{lem: exact Hilb}\eqref{lem: exact Hilb 1}. For each $i$, clearly we may write $u_i=\sum_{j=1}^dw_jx_j$ for some $w_j$.  Hence 
%\[
%0=\sum_{i=1}^d (\partial_iG)u_i=\sum_{i=1}^d(\partial_iG)\sum_{j=1}^dw_jx_j=\sum_{j=1}^d\left(\sum_{i=1}^d(\partial_iG)w_i\right) x_j.
%\]
%Given this holds in the free algebra, it follows that $\sum_{i=1}^d(\partial_iG)w_i=0$, and so by induction the result follows.
%\end{proof}
%
\noindent
It will be convenient to consider the following subset of homogeneous superpotentials
\[
\mathsf{ESP}_k\colonequals\{ G\in\mathsf{SP}_k \mid G\mbox{ is an exact potential, and }\AlgJac(G)\xrightarrow{x_1\cdot}\AlgJac(G)\mbox{ is injective}\}.
\]
%\begin{lemma}\label{inject:101}
%If $G\in\mathsf{ESP}_k$ then 
%there is a set $N\subset\fringd$ of monomials, satisfying $x_1N\subset N$, 
%which descends to a basis of the $\C$-vector space $\AlgJac(G)$.
%\end{lemma}
%
%\begin{proof}
%Setting $M_0 =\{1\}$, inductively by degree in \eqref{eq: alg Jac graded} we can construct sets of monomials $M_j$ of $\fringd$ of degree $j$ such that the following two conditions are satisfied.
%\begin{enumerate}
%\item\label{Closed under x1}  There is an inclusion $x_1M_j\subseteq M_{j+1}$. 
%\item $M_j$ descends to give a basis of $\AlgJac(G)_j$.  
%\end{enumerate}
%It is clear that the necessarily disjoint union $N= \bigcup_{j\ge0} M_j$ descends to a basis of $\AlgJac(G)$, since it does so in each degree.
%\end{proof}
%
%
%
Recall that $d\geq 2$ is the number of variables in $\fringd$.
\begin{lemma}[{\cite[2.1, 2.2]{IS}}]\label{lem: ESP non empty}
If $k\geq 3$ with $(d,k)\neq(2,3)$, then $\ESP_k\neq\emptyset$.
\end{lemma}
\begin{proof}
Set
\[
g=\begin{cases}
\sum_{\upsigma\in\mathfrak{S}_{d-1}}x_d^{k-d+1}x_{\upsigma(1)}\hdots x_{\upsigma(d-1)}&\mbox{if }k\geq d\\
x_dx_{d-1}\hdots x_{d-k+1}+\sum_{j\in\scrS} x_jx_d\mathsf{m}_j&\mbox{if }k< d
\end{cases}
\]
where $\mathfrak{S}_{d-1}$ is the symmetric group, $\scrS=\{j\mid 1\leq j\leq d-1,\, j\neq d-k+1\}$, and the $\mathsf{m}_j$ are explicit monomials explained in \cite[2.2]{IS}. It is a reasonably elementary calculation to show that $G=\sym(g)\in\mathsf{ESP}_k$; see \cite[2.1, 2.2]{IS}.
\end{proof}
It turns out, although we do not need this, that if $d=2$ then $\mathsf{ESP}_3=\emptyset$.  This is why the argument in \ref{thm: appendix main} below fails in the $(d,k)=(2,3)$ case.

Recalling that $\mathfrak{R}^i=(\m^i+I_F)/I_F$, the following is one of the main insights of \cite{IS}.
\begin{cor}\label{cor:x inj}
If $F\in\mathsf{SP}_{k,m}$ with $F_k\in\mathsf{ESP}_k$, then left multiplication
\[
\AlgJac(F)/\mathfrak{R}^j\xrightarrow{x_1\cdot}\AlgJac(F)/\mathfrak{R}^{j+1}
\]
is injective for all $j\geq 1$.
\end{cor}
\begin{proof}
This is \cite[3.1]{IS}.  A proof in the notation used here is in \href{https://arxiv.org/abs/2111.05900v1}{arXiv:2111.05900v1}, as Proposition A.12.
% For convenience, and in the notation used here, a proof can also be found in the first ArXiv version of arXiv:2111.05900, as Proposition A.12.
\end{proof}
\subsection{Very general elements}

Fixing, once and for all, a basis $f_1,f_2,\dots,f_r$ of $\mathsf{SP}_k$,
we treat $\mathsf{SP}_k$ as an irreducible algebraic family of superpotentials, and identify
it with $\C^r=\Spec\C[\mathsf{t}]$, where $\mathsf{t}=t_1,\dots,t_ r$ are parameter variables
and any element of $\mathsf{SP}_k$ is the specialisation $\mathsf{t}=a$ at a point $a\in\C^r$ of
the `generic' superpotential $G_{\mathsf{t}} = \sum t_if_i$.
In particular, $\mathsf{SP}_k$ inherits the Zariski topology from~$\C^r$.
Under this identification, it is natural to abbreviate $G_a\in\mathsf{SP}_k$ by $a\in\mathsf{SP}_k$.
%(This is a low brow but nevertheless precise approach to the {\em variety of algebras}
%$\scrW_k=\{\AlgJac(G)\mid G\in\mathsf{SP}_k\}$ used in \cite{ISmok,IS}.)

\begin{lemma}[c.f.\ {\cite[3.9]{ISmok}}]\label{lem:Zaropens}
Fix $(d,k)$ and consider $\mathsf{SP}_k$ with its Zariski topology as above.
For any $i\ge0$,
\begin{enumerate}
\item\label{Zaropens1}
there is a non-empty Zariski open subset $U_i\subset\mathsf{SP}_k$
on which $\dim\AlgJac(G_a)_i$ is constant for all $a\in U_i$, and takes the minimum value of any $G\in\mathsf{SP}_k$.
\item\label{Zaropens2}
there is a largest Zariski open subset $V_i\subset U_i\cap U_{i+1}$
on which the rank of the restriction
$\mathsf{d}_3\colon \AlgJac(G_a)_i\to (\AlgJac(G_a)^{\oplus d})_{i+1}$ for any $a\in V_i$
is the maximum possible for a linear map between spaces of these dimensions
(i.e.\ is injective).
\item\label{Zaropens3}
there is a largest Zariski open subset $W_i\subset V_i\cap U_{i+k-1}$
on which the rank of the restriction $\mathsf{d}_2\colon (\AlgJac(G_a)^{\oplus d})_{i+1}\to(\AlgJac(G_a)^{\oplus d})_{i+k-1}$
for any $a\in W_i$ is the
maximum possible for a linear map whose kernel contains the image of $\mathsf{d}_3$.
\end{enumerate}
Thus if $W_i$ not empty, then for $a\in W_i$, the homogeneous slice of \eqref{defin:exact} with the domain of $\mathsf{d}_3$ in degree~$i$
is an exact sequence of finite-dimensional vector spaces.
\end{lemma}
\begin{proof}
(1) This is \cite[2.2]{IS1}, proved as in \cite[2.1]{IS1}.
The point is this: let $a_0\in\SP_k$ be a point at which the minimum dimension is achieved,
and choose a subset of the set $\Phi_i$ of all monomials of degree~$i$
for which the (square) coefficient matrix of a basis of the kernel of $\fringd_i\to\AlgJac(G_{a_0})_i$
with respect to $\Phi_i$ is invertible (a section of the surjection). The entries of this square matrix are algebraic in~$a$, so 
it remains invertible as $a$ varies in a Zariski open subset $a_0\in U_i\subset\mathsf{SP}_k$. Thus the kernel cannot get smaller, so must have constant dimension, and thus the minimum is also achieved for all~$a\in U_i$.

\noindent
(2) Within $U_i\cap U_{i+1}$, the map
$\mathsf{d}_3\colon \AlgJac(G_a)_i\to(\AlgJac(G_a)^{\oplus d})_{i+1}$
is determined by a matrix of fixed size whose entries are functions in $a$.
Maximising its rank is an open condition on $a$, since it occurs on
the complement of a locus of vanishing minors.
(It is possible that all relevant minors vanish identically, in which case
the conclusion is simply that $V_i$ is empty.)

\noindent
(3) Similarly, maximising the rank of~$\mathsf{d}_2$ is an open condition prescribed by
minors that are functions in~$a$. The condition that
the kernel contains the image of $\mathsf{d}_3$ is already
imposed by the entries of the matrix, since $a\in V_i$ and \eqref{defin:exact} is a complex.
\end{proof}

\begin{cor}[{\cite[3.2]{ISmok}}]\label{cor!nonempty}
Let $k\ge3$ and $(d,k)\ne(2,3)$.
\begin{enumerate}
\item
For each $i\ge0$, the minimum value achieved by $\dim\AlgJac(G)_i$ for any $G\in\SP_k$,
as in \textnormal{\ref{lem:Zaropens}\eqref{Zaropens1}}, is the coefficient of $t^i$ in the expansion of  $(1-dt+dt^{k-1}-t^{k})^{-1}$.
\item
If $G\in\SP_k$ is exact, then $\scrH_G$ attains the minimum coefficient
for every term~$t^i$.
\end{enumerate}
In particular, each $W_i$ in \textnormal{\ref{lem:Zaropens}\eqref{Zaropens2}} is non-empty,
and the minimum Hilbert series in the set $\left\{\scrH_F\mid F\in\mathsf{SP}_{k}\right\}$ is $(1-dt+dt^{k-1}-t^{k})^{-1}$.
\end{cor}

\begin{proof}
The proof follows \cite[3.2, 3.8]{ISmok}, which is for $d=2$ but generalises with no change,
using the existence of the exact potentials in \ref{lem: ESP non empty}.
\end{proof}

%By induction on $m\ge1$.
%Consider $V=V_0\cap V_1\cap \hdots \cap V_{m-k}$.
%Any exact potential lies in~$V$ by~\ref{cor:hgsHilb},
%so $V$ is nonempty by~\ref{lem: ESP non empty}.
%Now consider $W=V\cap W_0\cap W_1\cap\hdots\cap W_{m-k}\subset U_{m-1}$.
%Once again, any exact potential lies in~$W$, so $W\subset\SP_k$ is a nonempty Zariski open.
%
%Therefore we may find a potential $G\in W\cap U_m$.
%Writing $A=\AlgJac(G)$, the degree $m-k$ slice of \ref{cor:hgsHilb} is
%\begin{equation}\label{eq!Am}
%0\rightarrow A_{m-k} \xrightarrow{\mathsf{d}_3} (A^d)_{m-k+1}\xrightarrow{\mathsf{d}_2} (A^d)_{m-1} \xrightarrow{\mathsf{d}_1} A_m \rightarrow 0.
%\end{equation}
%The map $\mathsf{d}_3$ is injective (as $W\subset V_m$) and $\mathsf{d}_2$
%has corresponding maximal possible rank for a complex (as $W\subset W_{m-k}$).
%Thus $\dim\Im{\mathsf{d}_2}$ is the same as the corresponding dimension for any
%exact potential in~$W$.
%But \eqref{eq!Am} is exact at the domain and codomain of $\mathsf{d}_1$,
%and so $\Im{\mathsf{d}_2}=\ker\mathsf{d}_1$ and $\dim A_m=d_m$.

For formal power series $\phi(t)$ and $\uppsi(t)$, write $\phi\ge\uppsi$ to mean that
the coefficients of $\phi-\uppsi$ are all non-negative.
If $\mathcal{P}$ is a family of power series, then $\uppsi\in\mathcal{P}$ is called the {\em minimum}
if $\phi\ge\uppsi$ for all $\phi\in\mathcal{P}$, noting that the minimum does not necessarily exist.

\begin{prop}\label{prop: find homog general}
If $k\ge3$ and $(d,k)\neq (2,3)$, then there exists a countable intersection~$\scrU$ of non-empty Zariski opens of $\mathsf{SP}_k$ such that $G\in\mathsf{ESP}_k$ for all $G\in\scrU$.
\end{prop}

\begin{proof}
Let $\scrU_1$ be the countable intersection over all $i\ge0$ of the Zariski open subsets $W_i$ of~\ref{lem:Zaropens}.
By \ref{cor!nonempty}, each $W_i$ is non-empty, thus this intersection is non-empty.
Similarly, there is another such countable intersection $\scrU_2\subset\mathsf{SP}_k$
on which the left-multiplication map $x_1\colon \AlgJac(G)\to\AlgJac(G)$
is injective, as injectivity maximises rank in each degree.
Since $x_1$ is injective on the free algebra (at $\mathsf{t}=0$), or again applying
\ref{lem: ESP non empty}, this intersection is also non-empty.
Thus $\scrU=\scrU_1\cap\scrU_2$ is a non-empty countable intersection.
\end{proof}

For a given $F\in \SP_{k,m}$,
recall the notation $\m$ and $I_F$ from~\S\ref{app:notation}.
The {\em Poincar\'{e} series} of $F$ is defined to be
\begin{equation}
\scrP_F=\sum_{i\geq 0}\dim_\mathbb{C}\big(\tfrac{\fringd}{\m^{i+1}+I_F}\big)t^i.\label{eq: Poincare alg}
\end{equation}
This measures the growth of the quotients of $\AlgJac(F)$ by the ideals $(\m^{i+1}+I_F)/I_F$.
The general situation is more delicate than the homogeneous case of \ref{prop: find homog general},
but still a minimum is achieved by a very general element.

\begin{cor}\label{cor: find nonhomog F}
If $k\ge3$ and $(d,k)\neq (2,3)$, then for any $m\geq k$ there exists $F\in\mathsf{SP}_{k,m}$ such that the following statements hold.
\begin{enumerate}
\item $\scrP_F$ is the minimum in $\{ \scrP_H\mid H\in\mathsf{SP}_{k,m}\}$.
\item $F_{\ord(F)}\in\scrU$, where $\scrU\subset\mathsf{SP}_k$ is defined in~\textnormal{\ref{prop: find homog general}}.
\end{enumerate}
\end{cor}
\begin{proof}
Similar to \ref{lem:Zaropens}\eqref{Zaropens1} (cf.\ \cite[2.1]{IS1}),
minimising each coefficient of the Poincar\'e series is an open condition in a family, and so the
minimum is realised on 
a countable intersection of non-empty Zariski open subsets $\scrV\subset\mathsf{SP}_{k,m}$.
% N.B. Elements of this set V ONLY achieve the minimum - there are no exactness or other claims.

Consider the map $\mathsf{SP}_{k,m}\to\mathsf{SP}_{k}$ given by $F\mapsto F_k$.
Intersecting the preimage of $\scrU\subset\mathsf{SP}_{k}$ with $\scrV$ determines a countable intersection of open subsets of
$\mathsf{SP}_{k,m}$ on which the claims hold.
Since $\C$ is uncountable, this set contains a closed point.
\end{proof}

\subsection{Power Series}\label{sec: power series}

Given $f\in\ringd$, its Poincar\'{e} series is defined to be
\[
\hat\scrP_f=\sum(\dim_{\mathbb{C}}\Jac(f)/\mathfrak{J}^i)t^i
\]
where recall that $\mathfrak{J}$ is the Jacobson radical of $\Jac(f)$.

When $F\in\mathsf{SP}_{k,m}$, we can view $F$ as either a polynomial superpotential and form $\scrP_F$ in \eqref{eq: Poincare alg}, or we can view $F$ as an element of $\ringd$ and form $\hat{\scrP}_F$.  Since $\Jac(f)$ is defined with respect to the cyclic derivatives $\updelta$, and $\AlgJac(f)$ is defined with respect to strike-off derivatives $\partial$, it is not quite true that $\hat{\scrP}_F=\scrP_F$.  

\begin{lemma}\label{lem: Poincare are equal}
Given $G\in\mathsf{SP}_{k,m}$, then $\scrP_G=\hat{\scrP}_{\mathbb{G}}$, where $\mathbb{G}=\sum_{i=k}^m\tfrac{1}{i}G_i$.
Thus, if $G\in\mathsf{SP}_k$, then $\hat{\scrP}_{\mathbb{G}}=\scrP_G=\tfrac{1}{1-t}\scrH_G$ where $\mathbb{G}=\tfrac{1}{k}G$.
\end{lemma}
Now for $f\in\ringd$ and for any $i\geq 0$, set $\sff_i=f_{\leq i}$ and $\mathsf{F}_i=\sym(\sff_i)$.  Then, continuing the notation of \S\ref{completion section}, and
recalling that $\mathfrak{R}^i=(\m^i+I_F)/I_F$,
\begin{align}\label{eq:truncate to alg}
\Jac(f)/\mathfrak{J}^i
&\cong \frac{\ringd}{\lcl \n^i,\updelta_1f,\hdots,\updelta_df\rcl} \nonumber\\
&\cong\frac{\fringd}{(\m^i,\updelta_1\sff_i,\hdots,\updelta_d\sff_i)} \\
&\cong\frac{\fringd}{\big(\m^i,\partial_1(\sym\sff_i),\hdots,\partial_d(\sym\sff_i)\big)} 
\cong\AlgJac(\mathsf{F}_i)/\mathfrak{R}^i \nonumber
\end{align}
This gives a term-by-term algebraicisation of the Poincar\'{e} series, by 
\[
\hat{\scrP}_f=\sum_{i\geq 0}(\dim_{\mathbb{C}}\AlgJac(\mathsf{F}_i)/\mathfrak{R}^i)t^i.
\]

\subsection{Main Result}
The main result, \ref{thm: appendix main}, requires the following elementary lemma.
\begin{lemma}\label{lem: b lemma}
For $d\in\R$ and $k\geq 2$ consider the formal power series
\[
\frac{1}{(1-t)(1-dt+dt^{k-1}-t^k)}=\sum_{i\geq 0}b_it^i.
\]
Setting $b_j=0$ for $j<0$, the following statements hold.
\begin{enumerate}
\item\label{lem: b lemma 1}
There is an equality $b_j=db_{j-1}-db_{j-k+1}+b_{j-k}+1$ for all $j\geq 0$.
\item\label{lem: b lemma 2}
$b_0=1$, and further $b_j=1+d+\hdots +d^{j}$ for all $1\leq j\leq k-2$.
\end{enumerate}
\end{lemma}
\begin{proof}
(1) Treating the series as a sum over $i\in\Z$, with $b_{<0}=0$, and multiplying up shows
at once that $b_0=1$ and that for any~$j\in\Z\setminus\{0\}$ 
\[
b_j - (d+1)b_{j-1} + db_{j-2} + db_{j-k+1} - (d+1)b_{j-k} + b_{j-k-1} = 0.
\]
When $j=0$, the claimed equality in~\eqref{lem: b lemma 1} holds.
For $j\ge1$, splitting off a single $b_{j-1}$ summand from the equation above we see by induction that
\begin{align*}
b_j &= db_{j-1} - db_{j-2} - db_{j-k+1} + (d+1)b_{j-k} - b_{j-k-1} \\
&\qquad\qquad + (db_{j-2}-db_{j-k}+b_{j-k-1}+1) \\
&= db_{j-1} - db_{j-k+1} + b_{j-k} + 1.
\end{align*}

\noindent
(2) For $1\le j\le k-2$, the equality in~\eqref{lem: b lemma 1} reads $b_j = b_{j-1}d + 1$, and so the result follows since $b_0=1$.
\end{proof}

The following is the main result of this appendix, and it asserts that, in almost all cases, the $\mathfrak{J}$-dimension of $\Jac(f)$ is $\geq 3$.  In particular, in almost all cases the Jacobi algebra is infinite dimensional, as a vector space.  Recall from \S\ref{sect: power notation} that $\ringd_{\geq k}$ consists of all those $f\in\ringd$ for which $f_j=0$ for all $j<k$, and note that $0\in\ringd_{\geq k}$.
\begin{thm}\label{thm: appendix main}
Suppose that $d=2$ and $k\geq 4$, or $d\geq 3$ and $k\geq 3$.  If $f\in\ringd$ has order $k$, then $\JRdim\Jac(f)\geq 3$.
\end{thm}
The proof will show that the coefficients of the Poincar\'e series of $\Jac(f)$
are no smaller than those of
\[
\frac{1}{(1-t)(1-dt+dt^{k-1}-t^k)}.
\]
When $d=2$ and $k=4$ this lower bound is $1/\big((1-t)^3(1-t^2)\big)$,
and when $d=k=3$ it is $1/(1-t)^4$, both of which have polynomial growth of degree~3.
For all other $d,k$ in the scope of the theorem, the growth is exponential
and $\JRdim\Jac(f)=\infty$.
\begin{proof}
Associated to the fixed $k=\ord(f)$ and $d$ is the power series
\[
\frac{1}{(1-t)(1-dt+dt^{k-1}-t^k)}=\sum_{i\geq 0}b_it^i.
\]
The $b_i$s are integers that depend only on $k$ and $d$.  Similarly, associated to~$k$ and~$d$ are the positive integers $a_0,a_1,\hdots$ defined to be
\begin{align*}
a_j&\colonequals \textnormal{min}\{\dim_{\mathbb{C}}\big(\tfrac{\Jac(g)}{\mathfrak{J}^{j+1}}\big)\mid g\in\ringd_{\geq k}\}\\
&= \textnormal{min}\{\dim_{\mathbb{C}}\big(\tfrac{\Jac(g)}{\mathfrak{J}^{j+1}}\big)\mid g\in\fringd_{\geq k}\}\tag{truncate terms mod $\mathfrak{J}^{j+1}$}\\
&=\textnormal{min}\{\dim_{\mathbb{C}}\big(\tfrac{\Jac(g)}{\mathfrak{J}^{j+1}}\big)\mid g\in\fringd_{\geq k}, \,g_{t}=0\mbox{ for }t>j+1\}\\
&=
\begin{cases}
\dim_{\mathbb{C}}\big(\tfrac{\ringd}{\n^{j+1}}\big)&\mbox{if }j\leq k-2\\
\textnormal{min}\left\{\dim_{\mathbb{C}}\big(\tfrac{\AlgJac(G)}{\mathfrak{R}^{j+1}}\big)\mid G\in\mathsf{SP}_{k,j+1}\right\}&\mbox{if }j\geq k-1.
\end{cases}
.\tag{by \eqref{eq:truncate to alg}}
\end{align*}
Certainly $\hat{\scrP}_f\geq \sum_{i\geq 0}a_it^i$, by the minimality of the $a_i$.  We claim that $a_i=b_i$ for all $i\geq 0$, since then 
\[
\hat{\scrP}_f\geq \sum_{i\geq 0}a_it^i=\sum_{i\geq 0}b_it^i=\frac{1}{(1-t)(1-dt+dt^{k-1}-t^k)},
\]
which has the prescribed growth as in the statement of the result.

So, from here on we discard the original $f$, and instead prove that $a_i=b_i$ for all $i\geq 0$. This is a statement which depends only on $k$ and $d$.  

\medskip
Since by assumption $(d,k)\neq (2,3)$, by \ref{lem: ESP non empty} $\mathsf{ESP}_k\neq \emptyset$ and so choose $G\in\mathsf{ESP}_k$.  Since $G$ is exact,  by \ref{cor:hgsHilb} $\scrH_G=(1-dt+dt^{k-1}-t^k)^{-1}$, and so by \ref{lem: Poincare are equal} for $\mathbb{G}=\tfrac{1}{k}G$ we have
\[
\hat{\scrP}_{\mathbb{G}}=\frac{1}{1-t}\cdot\scrH_G=\frac{1}{(1-t)(1-dt+dt^{k-1}-t^k)}=\sum_{i\geq 0}b_it^i.
\]
Since $\mathbb{G}$ exists, it follows immediately by minimality of the $a_i$s that $a_i\leq b_i$ for all $i\geq 0$. 

Now clearly $a_0=b_0=1$, and further for all $1\leq j\leq k-2$ we have $a_j=\dim_{\mathbb{C}}\big(\tfrac{\ringd}{\n^{j+1}}\big)$, which equals $b_j$ by \ref{lem: b lemma}\eqref{lem: b lemma 2}.  Further, since $d$ relations of degree $k$ can cut down the dimension of $\ringd/\mathfrak{J}^{k+1}$ by at most $d$, it follows that $a_{k-1}\geq 1+d+\hdots +d^k - d$.  This equals $b_{k-1}$ by \ref{lem: b lemma}\eqref{lem: b lemma 2}, and so $a_{k-1}\geq b_{k-1}$, which in turn forces $a_{k-1}=b_{k-1}$.

Thus, by induction we can suppose that $a_j=b_j$ for all $0\leq j\leq s$, for some $s\geq k-1$.  The proof will be completed once we show that $a_{s+1}=b_{s+1}$.  

Now by \ref{cor: find nonhomog F} applied to $m=s+2$, there exists $F\in\mathsf{SP}_{k,s+2}$ for which $\scrP_{F}$ is the minimum in $\{ \scrP_H\mid H\in\mathsf{SP}_{k,s+2}\}$, and further $F_k\in\scrU$ where $\scrU$ is from \ref{prop: find homog general}.  By the first of these facts, since for all $j\leq s+1$ by truncation
\[
a_{j}=\textnormal{min}\left\{\dim_{\mathbb{C}}\big(\tfrac{\AlgJac(\mathsf{H})}{\mathfrak{R}^{j+1}}\big)\mid \mathsf{H}\in\mathsf{SP}_{k,j+1}\right\}
=\textnormal{min}\left\{\dim_{\mathbb{C}}\big(\tfrac{\AlgJac(\mathsf{H})}{\mathfrak{R}^{j+1}}\big)\mid \mathsf{H}\in\mathsf{SP}_{k,s+2}\right\}
\]
it follows that 
\begin{equation}
\scrP_{F}=\sum_{j=0}^{s+1}a_jt^i +\scrO_{s+2}.\label{eq: up to s plus 1}
\end{equation}
On the other hand, since $F_k\in\scrU$, by definition $F_k\in\mathsf{ESP}_k$. Thus, by \ref{cor:x inj} for all $j\geq 1$ the left multiplication by $x_1$
\[
\AlgJac(F)/\mathfrak{R}^j\xrightarrow{x_1\cdot}\AlgJac(F)/\mathfrak{R}^{j+1}
\]
is injective. 

Set $I=(\partial_1F,\hdots,\partial_dF)\subset\fringd$. The above asserts that there is an injection
\[
\frac{\fringd}{I+\m^j}\xrightarrow{x_1\cdot}\frac{\fringd}{I+\m^{j+1}}
\]
for all $j\geq 1$. This allows us to pick inductively sets of monomials $M_j$ of $\fringd$ of degree~$j$, starting with $M_0=\{1\}$, such that the following two conditions are satisfied.
\begin{enumerate}
\item\label{closed under x1}  There is an inclusion $x_1M_j\subseteq M_{j+1}$. 
\item The necessarily disjoint union $N_j=M_0\cup\hdots\cup M_j$ projects down to give a basis of $\fringd/(I+\m^{j+1})$.  
\end{enumerate}
To fix notation, set $B_j=\mathrm{Span}_{\mathbb{C}}(N_j)\subset\fringd$, and define $\sfb_j$ via the equality
\[
\dim_{\mathbb{C}}B_j=\dim_{\mathbb{C}}\big(\tfrac{\fringd}{I+\m^{j+1}}\big) = \sfb_j.
\] 
Note that the second equality implies that $\scrP_F=\sum_{i\geq 0}\sfb_jt^j$.  

Write $V=\mathrm{Span}_{\mathbb{C}}\{x_1,\hdots,x_d\}\subseteq \fringd$, $R=\mathrm{Span}_{\mathbb{C}}\{\partial_1F,\hdots,\partial_dF\}\subseteq \fringd$, and for $j\geq 0$ consider the quotient map $\uppi_j\colon\fringd\to\fringd/\m^{j+1}$.
By the definition of $\sfb_{j}$, for every $j\ge0$
\begin{equation}\label{pijI}
\dim_{\mathbb{C}}\uppi_{j}(I)=1+d+\hdots+d^{j}-\sfb_{j}.
\end{equation}

We now apply the adapted Vinberg argument:
elements of the two-sided ideal $I$ are sums of elements, each of which either starts
with an $x_i$ or starts with a derivative $\partial_iF$, so we can write
\[
I=VI+R\,\fringd.
\]
Applying $\uppi_{j+1}$ then gives an equality
\[
\uppi_{j+1}(I)=\uppi_{j+1}(VI) +\uppi_{j+1}(R\,\fringd).
\]
Since $B_j$ descends to span $\fringd/(I+\m^{j+1})$, every element of $\fringd$ may be written as an element in $B_j$, plus an element in $I$, plus an element in $\m^{j+1}$.  Projecting down this sum via $\uppi_{j+1}$, and noting that $RI\subset R\,\fringd$ gets absorbed into $\uppi_{j+1}(I)$,
and elements of $R$ have degree $\geq k$, it follows that mod $\m^{j+1}$ there is an equality
\[
\uppi_{j+1}(I)=\uppi_{j+1}(VI) +\uppi_{j+1}(RB_{j+2-k}).
\]
Write $R'=\mathrm{Span}_{\mathbb{C}}\{\partial_2F,\hdots,\partial_dF\}\subseteq \fringd$, then using the Euler relation $\sum_{i=0}^d[x_i,\partial_iF]=0$ of \ref{lem:Euler} we may get rid of any appearance of the product $(\partial_1F)x_1$ at the cost of terms in the other summands. It follows that
\[
\uppi_{j+1}(I)=\uppi_{j+1}(VI) +\uppi_{j+1}(R'B_{j+2-k})+\uppi_{j+1}((\partial_1F)B^+_{j+2-k}),
\]
where $B^+_{j+2-k}=\mathrm{Span}_{\mathbb{C}}\{n\in N_{j+2-k}\mid n\neq x_1m \mbox{ for any }m\}$.

The proof is completed by estimating the dimension of each of the three individual summands.
Applying \eqref{pijI} for the first summand, 
\[
\dim_{\C}\uppi_{j+1}(I)\leq d(1+d+\hdots+d^j-\sfb_j)+ (d-1)\sfb_{j+2-k}+(\sfb_{j+2-k}-\sfb_{j+1-k})
\]
since $\dim_{\mathbb{C}}B^+_{j+2-k}\leq\sfb_{j+2-k}-\sfb_{j+1-k}$
 holds by construction of the $N_j$ in \eqref{closed under x1}. 
Plugging \eqref{pijI} for $\uppi_{j+1}$
into the above displayed equation, and then cancelling, it follows that
\[
1-\sfb_{j+1}\leq -d\sfb_j+d\sfb_{j+2-k}-\sfb_{j+1-k}.
\]
which after re-arranging gives
\begin{equation}
\sfb_{j+1}\geq \sfb_jd-\sfb_{j+2-k}d+\sfb_{j+1-k}+1.\label{eq: c inequality}
\end{equation}
Since $\scrP_F=\sum_{i\geq 0}\sfb_jt^j$, by \eqref{eq: up to s plus 1} we see that $\sfb_j=a_j$ for $0\leq j\leq s+1$, and hence
\begin{align*}
a_{s+1}&\geq a_sd-a_{s+2-k}d+a_{s+1-k}+1 \tag{\eqref{eq: c inequality} for $j=s$}\\
&= b_sd-b_{s+2-k}d+b_{s+1-k}+1 \tag{$a_j=b_j$ for $j\leq s$ by induction}\\
&=b_{s+1}\tag{\ref{lem: b lemma}\eqref{lem: b lemma 1} for $j=s+1$}
\end{align*}
Since we already know $a_{s+1}\leq b_{s+1}$ by minimality, the above forces $a_{s+1}=b_{s+1}$.  Hence by induction $a_{j}= b_{j}$ for all $j\geq 0$, and the result follows.
\end{proof}

\begin{remark}
The above theorem establishes that often $\JRdim\Jac(f)\geq 3$, whilst earlier sections considered the case $\JRdim\Jac(f)\leq 1$.  In general we do not know what, if anything, satisfies $1<\JRdim\Jac(f)< 3$,  even in the case when $f$ is a polynomial.   Nor do we know whether $\JRdim\Jac(f)$ is always an integer. 
\end{remark}

\subsection{Contractibility Consequence}
The following is an immediate geometric consequence of the above.

\begin{thm}\label{thm:lastcontractmain}
Let $\mathrm{C}\subset\scrX$ be an irreducible rational curve in a smooth \textnormal{CY} \textnormal{3}-fold, with \textnormal{NC} deformation algebra $\Lambda_{\mathrm{def}}$, such that $\scrN_{\mathrm{C}|\scrX}\neq (-3,1)$.  Then $\mathrm{\Curve}\subset\scrX$ contracts to a point suitably locally, without contracting a divisor, if and only if $\dim_\mathbb{C}\Lambda_{\mathrm{def}}<\infty$.
\end{thm}
\begin{proof}
The case of $(-1,-1)$ and $(-2,0)$-curves is of course well known \cite{Pagoda,DW1}.  The point is that $(-4,2),(-5,3),\hdots$ curves never contract \cite[4.1]{Laufer}, and further as a consequence of \cite{VdBCY} their noncommutative deformations are given as the Jacobi algebra quotient of a free power series ring in $3,4,\hdots$ variables.  By \ref{thm: appendix main}, these can never be finite dimensional. Thus in all these cases the statement is true, just since the curves never contract, and the deformation algebras are always infinite dimensional.
\end{proof}

%-----------------------------------------------------------------------------------------

\end{document}